\documentclass[reqno]{article}

\usepackage{mathtools,amscd,amsthm,amssymb}
\usepackage{graphicx,xcolor}
\usepackage[pagebackref]{hyperref}
\usepackage{rotating}
\usepackage{diagbox,booktabs,array}
\usepackage{tikz-cd}
\usepackage{tikz}
\usepackage{pgfplots}
\usepgfplotslibrary{groupplots}
\pgfplotsset{compat=1.18}
\usetikzlibrary{positioning}
  \tikzset{vertex/.style={circle,fill,inner sep=1.pt}}

\usepackage[mathscr]{eucal}
\usepackage{setspace}
\usepackage{enumerate}
\usepackage{subcaption}

\usepackage[T1]{fontenc}
\usepackage[utf8]{inputenc}
\usepackage[a4paper,margin=1in]{geometry}
\hypersetup{colorlinks=true,linkcolor=blue,citecolor=blue,urlcolor=blue}

\newtheorem{theorem}{Theorem}
\newtheorem{prop}[theorem]{Proposition}
\newtheorem{cor}[theorem]{Corollary}
\newtheorem{lem}[theorem]{Lemma}

\theoremstyle{definition}
\newtheorem{defn}[theorem]{Definition}

\theoremstyle{remark}
\newtheorem{remark}[theorem]{Remark}

\theoremstyle{plain}
\newtheorem{THEO}{Theorem}

\newcommand{\C}{\mathbb C}
\newcommand{\Z}{\mathcal Z}

\newcommand{\rad}{\operatorname{rad}}
\newcommand{\ord}{\operatorname{ord}}
\newcommand{\lc}{\operatorname{lc}}
\newcommand{\dd}{\,\mathrm d}

\title{Zero asymptotics for successive derivatives of hyperexponential functions with finite essential singularities}
\author{Christian H\"agg\thanks{Department of Mathematics, Stockholm University, SE-106\,91 Stockholm, Sweden. Email: \texttt{hagg@math.su.se}, \texttt{shapiro@math.su.se}}\and Boris Shapiro\footnotemark[1]}
\date{}

\begin{document}
\maketitle

\begin{abstract}
P\'olya's shire theorem identifies the final set of zeros of successive derivatives of an arbitrary meromorphic function with at least one pole with the Voronoi diagram of its finite poles. We prove a fixed-scale zero-counting law for hyperexponential functions \(f=(P/Q)\exp(S/T)\), allowing ordinary poles and finite essential singularities of arbitrary order and position, thus extending P\'olya's picture beyond the rational, polynomial-exponential, and one-dimensional finite-essential-singularity settings. After the forced singular factors are removed from the numerator of \(f^{(n)}\), the normalized zero-counting measures converge in the original \(z\)-plane to the classical Voronoi edge measure generated by all finite singular sites, augmented by explicitly weighted atoms at the finite essential singularities, which thereby enter P\'olya's picture both as Voronoi sites and as sources of linear-size zero clusters. If \(S/T\) has a nonconstant polynomial part, the complementary mass escapes to infinity. We determine the microscopic laws of these clusters, obtaining the reciprocal Marchenko--Pastur law for simple poles of \(S/T\) and higher-order multiple-Laguerre, equivalently Laguerre Muttalib--Borodin, limits for higher-order poles. Finally, inside essential Voronoi cells we identify the first sublinear zero layer, including its Stokes geometry, densities, and final-set consequences away from transition loci.
\end{abstract}

\section{Introduction}

\subsection{Short historical account}
The zeros of successive derivatives are a classical probe of the analytic character of a function. P\'olya's 1922 paper \cite{Po} established two foundational results of this kind.

\begin{defn}\label{def:finalset}
Let $\mathcal S\subset\C$ be closed and discrete, and let $f$ be holomorphic on $\C\setminus \mathcal S$. The \emph{final set} of $f$ is
$$L(f):=\{z\in\C:\text{ every neighborhood of }z\text{ contains zeros of }f^{(n)}\text{ in }\C\setminus \mathcal S\text{ for infinitely many }n\}.$$
\end{defn}

\begin{THEO}[P\'olya, rational case]\label{th:shire}
Let $f$ be a rational function with at least one finite pole. For each finite pole $A$, let its \emph{shire} be the set of points in $\C$ closer to $A$ than to any other finite pole. Then the final set $L(f)$ is precisely the union of the boundaries of these shires.
\end{THEO}

Thus, for rational functions, the final set is the Voronoi diagram of the finite poles.

\begin{figure}[!hbt]
\centering
\includegraphics[width=0.55\linewidth]{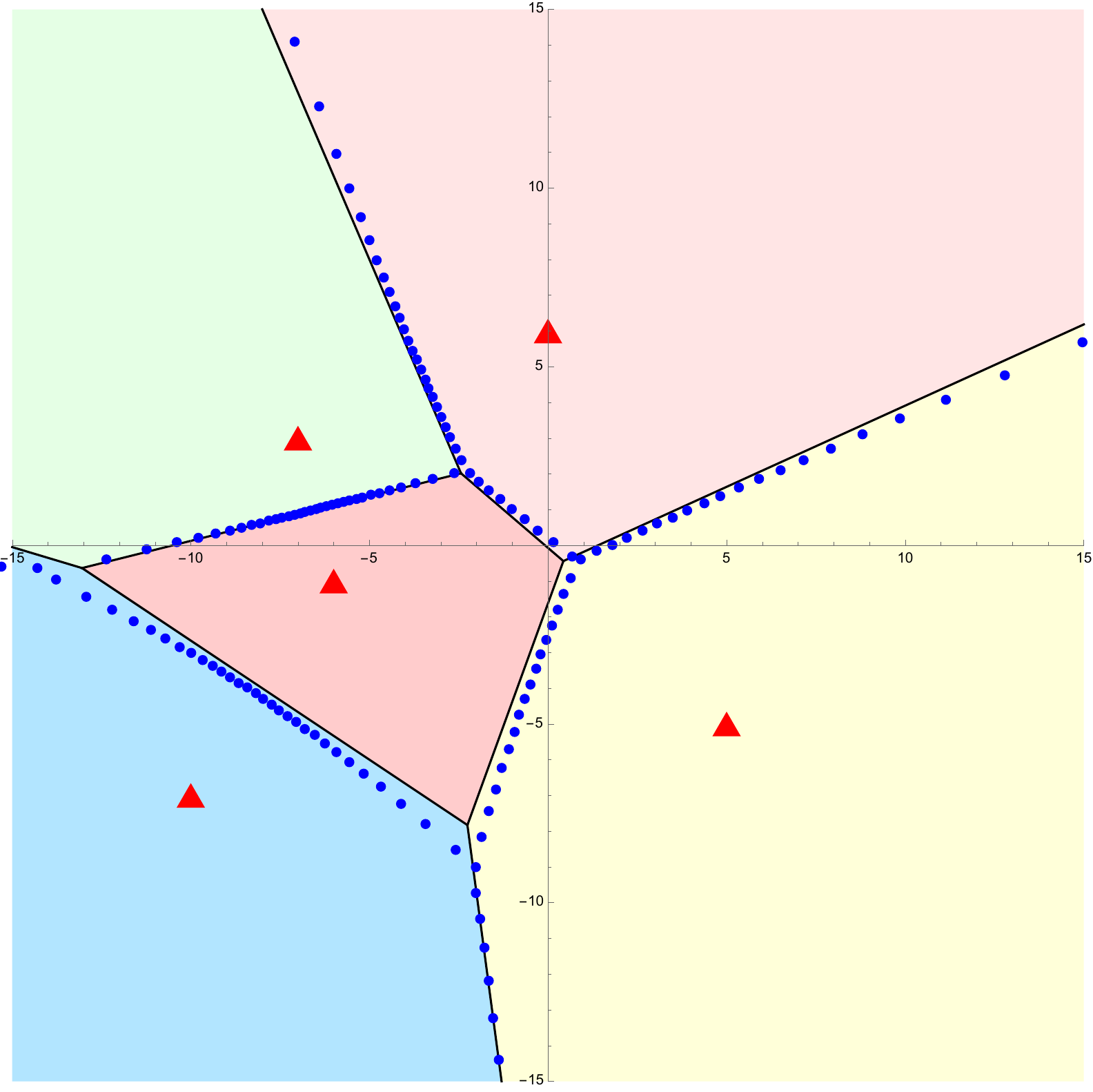}
\caption{Illustrative Voronoi diagram determined by the five finite poles of a rational function~$f$. Triangular markers indicate poles, and small circular markers indicate zeros of $f^{(20)}$.}\label{vv3}
\end{figure}

\medskip
P\'olya also treated polynomial-exponential entire functions. If $f(z)=P(z)e^{R(z)}$, with $P$ and $R$ polynomials and
$R(z)=cz^q+\delta z^{q-1}+\cdots$ nonconstant, then the following description holds.

\begin{THEO}\label{th:rays}
In the above notation, if $q\ge2$, then $L(f)$ consists of the $q$ rays emanating from the point $-\frac{\delta}{qc}$, parallel to the directions determined by the solutions of $cz^q+1=0$; these rays are equally spaced. If $q=1$, then $L(f)=\varnothing$.
\end{THEO}

Taken together, these theorems exhibit two mechanisms by which derivatives detect singularities: nearest-site competition among finite poles and saddle directions generated by exponential growth. Later work developed both mechanisms; see \cite{BH, BSTW, ClEd, Ge, Ge2, H, PrSh1, PrSh2, Rob, Wei} and the references therein. The fixed-scale zero-counting refinements most closely related to this paper are the rational case of B\o gvad and H\"agg \cite{BH}, and H\"agg's treatment \cite{H} of \(g(z)e^{p(z)}\), where \(g\) is rational with at least two finite poles, all simple, and \(p\) is a nonconstant polynomial; in this setting the only essential singularity is at infinity. Other extensions of the rational case, in particular to algebraic singularities \cite{PrSh2}, automorphic functions \cite{Wei}, and translation surfaces \cite{BSTW}, do not provide a planar fixed-scale law for several finite essential singularities in arbitrary position.

Finite essential singularities already appear in P\'olya's work. In \S~4 of \cite{Po2} he considered
\[
f(z)=z^{-1}e^{-1/z}
\]
and proved that its final set is the nonnegative real axis. The derivative identity
\[
f^{(n)}(z)=(-1)^n n!\,z^{-n-1}e^{-1/z}L_n(1/z)
\]
shows that final-set information alone does not determine fixed-scale zero counting. The finite zeros are the reciprocals of the zeros of the Laguerre polynomial $L_n=L_n^{(0)}$; as $n$ varies they are dense on $(0,\infty)$, but their normalized zero-counting measures assign asymptotically zero mass to every fixed compact subinterval of this ray. The missing mass instead concentrates at the essential singularity at $0$ on a microscopic scale. Section~\ref{sec:sublinear} proves the corresponding sublinear ray measure, and Section~\ref{sec:microscopic} proves the microscopic cluster law.

The works \cite{Ed,ClEd} treat finite essential singularities under special one-dimensional hypotheses and only for a single such singularity. They therefore do not yield a fixed-scale planar Voronoi law for arbitrary configurations of finite essential singularities and ordinary poles, nor the accompanying atomic and microscopic cluster laws at the essential sites.

\subsection{Our setup}
We study functions of the form
\begin{equation}\label{eq:mixed}
f(z)=\frac{P(z)}{Q(z)}\exp\!\left(\frac{S(z)}{T(z)}\right),
\end{equation}
where $P,Q,S,T\in\C[z]$, $P,Q\ne0$, $\gcd(P,Q)=\gcd(S,T)=1$, and $T$ is nonconstant. Their finite singular set is
\[
\Sigma=\{z:T(z)Q(z)=0\}=\{a_1,\ldots,a_N\}.
\]
A zero of \(T\) is a finite essential singularity of \(f\); if \(Q\) also vanishes at that point, this only changes the local algebraic order. The ordinary poles are precisely the zeros of \(Q\) that are not zeros of \(T\). For a site \(a_i\in\Sigma\), let \(m_i\) be its multiplicity as a zero of \(T\), and set \(m_i=0\) at ordinary-pole sites. Let $H$ be the polynomial part of $S/T$, and put $h=\deg H$ if $H$ is nonconstant and $h=0$ otherwise.

After the monic normalizations specified in Section~\ref{sec:prelim}, let $T_0$ be the product of one factor $z-c$ over each distinct zero $c$ of $T$, and let $Q_\ast$ be the corresponding product over the distinct zeros of $Q$ that are not zeros of $T$. The pole-clearing polynomial used throughout is
\[
W=T\,T_0\,Q_\ast,
\qquad
 d:=\deg W=\sum_{i=1}^N m_i+N,
\qquad
\kappa:=d+h-1.
\]
With $E=S/T$ and $P_T$ denoting the factor of $P$ supported on the zeros of $T$, the reduction to a polynomial problem is exact: there are polynomials $B_n$ such that
\begin{equation}\label{eq:nthder}
f^{(n)}(z)=\frac{P_T(z)\,B_n(z)}{Q(z)\,W(z)^n}e^{E(z)}.
\end{equation}
On $\C\setminus\Sigma$, the factor outside $B_n$ is holomorphic and nonvanishing, so the zeros of $f^{(n)}$ away from the singular set are exactly the zeros of $B_n$. Proposition~\ref{prop:rec} gives the recurrence for $B_n$, and Proposition~\ref{prop:deg} shows that $\deg B_n=\kappa n+O(1)$.

\subsection{Main results}\label{subsec:main-results}
Here, vague convergence on the finite plane means convergence against compactly supported continuous test functions. Fixed scale means that this convergence is taken in the original \(z\)-plane, with normalization by the full degree \(\deg B_n=\kappa n+O(1)\), and without rescaling about a singular point. With the notation of \eqref{eq:nthder}, the main conclusions are the following.
\begin{enumerate}
\item \textbf{Fixed-scale law.} Theorem~\ref{thm:limit} proves that
\[
\mu_n:=\frac{1}{\deg B_n}\sum_{B_n(\zeta)=0}\operatorname{mult}_\zeta(B_n)\,\delta_\zeta
\]
converges vaguely on the finite plane to
\[
\mu_{\mathrm{fix}}
=\frac{1}{\kappa}\sum_{m_i\ge1}m_i\,\delta_{a_i}
+\frac{1}{2\pi\kappa}\sum_{1\le i<j\le N}
\frac{|a_i-a_j|}{|\cdot-a_i|\,|\cdot-a_j|}\,\dd\ell\big|_{E_{ij}},
\]
where $E_{ij}$ is the Voronoi edge between $a_i$ and $a_j$, $\dd\ell$ is arclength, and an empty edge contributes zero. If $h=0$, this is a probability measure on $\C$; if $h>0$, its finite-plane mass is $1-h/\kappa$ and the remaining mass $h/\kappa$ escapes to infinity on $\widehat\C$.
\item \textbf{Derivative zeros.} On $\C\setminus\Sigma$, the factor multiplying $B_n$ in \eqref{eq:nthder} is holomorphic and nonvanishing, and Proposition~\ref{prop:local} shows that $B_n$ has no zeros at the singular sites. Thus the same fixed-scale law is equivalently a law for the zeros of $f^{(n)}$ in $\C\setminus\Sigma$, counted with the normalization $\deg B_n$; Corollary~\ref{cor:fixed-scale-derivative-zeros} records this explicitly.
\item \textbf{Microscopic clusters.} The atom at an essential site $a_i$ represents $m_i n+o(n)$ zeros in every sufficiently small fixed disk around $a_i$. On the scale $z-a_i\asymp n^{-1/m_i}$, Theorem~\ref{thm:full-local-microscopic} gives a limiting law depending only on the leading principal term of $S/T$ at $a_i$: the reciprocal Marchenko--Pastur law for $m_i=1$, and the higher-order multiple-Laguerre, equivalently Laguerre Muttalib--Borodin, law for $m_i\ge2$.
\item \textbf{Sublinear chamber law and final set.} Inside the Voronoi cell of an essential singularity $a_i$ (an \emph{essential Voronoi cell}), away from $a_i$ itself, the first nontrivial normalization is $n^{m_i/(m_i+1)}$. Theorem~\ref{thm:sublinear-essential} identifies the chamberwise zero measure as $(2\pi)^{-1}\Delta$ of the maximum of the active Wright saddle phases, with the active set computed by Proposition~\ref{prop:active-wright-coefficients}. Corollary~\ref{cor:regular-final-set} gives the corresponding final-set statement away from the explicitly excluded finite-plane loci.
\end{enumerate}
The transition regimes not treated by these theorems are the wall and tie scales in the sublinear saddle problem, the local transition limits near Voronoi edges and vertices, interactions between microscopic and sublinear scales near an essential singularity, and the outer scale associated with a nonconstant polynomial part of $S/T$; Section~\ref{sec:open} lists these as open problems.

\begin{figure}[t]
\centering
\includegraphics[width=\linewidth]{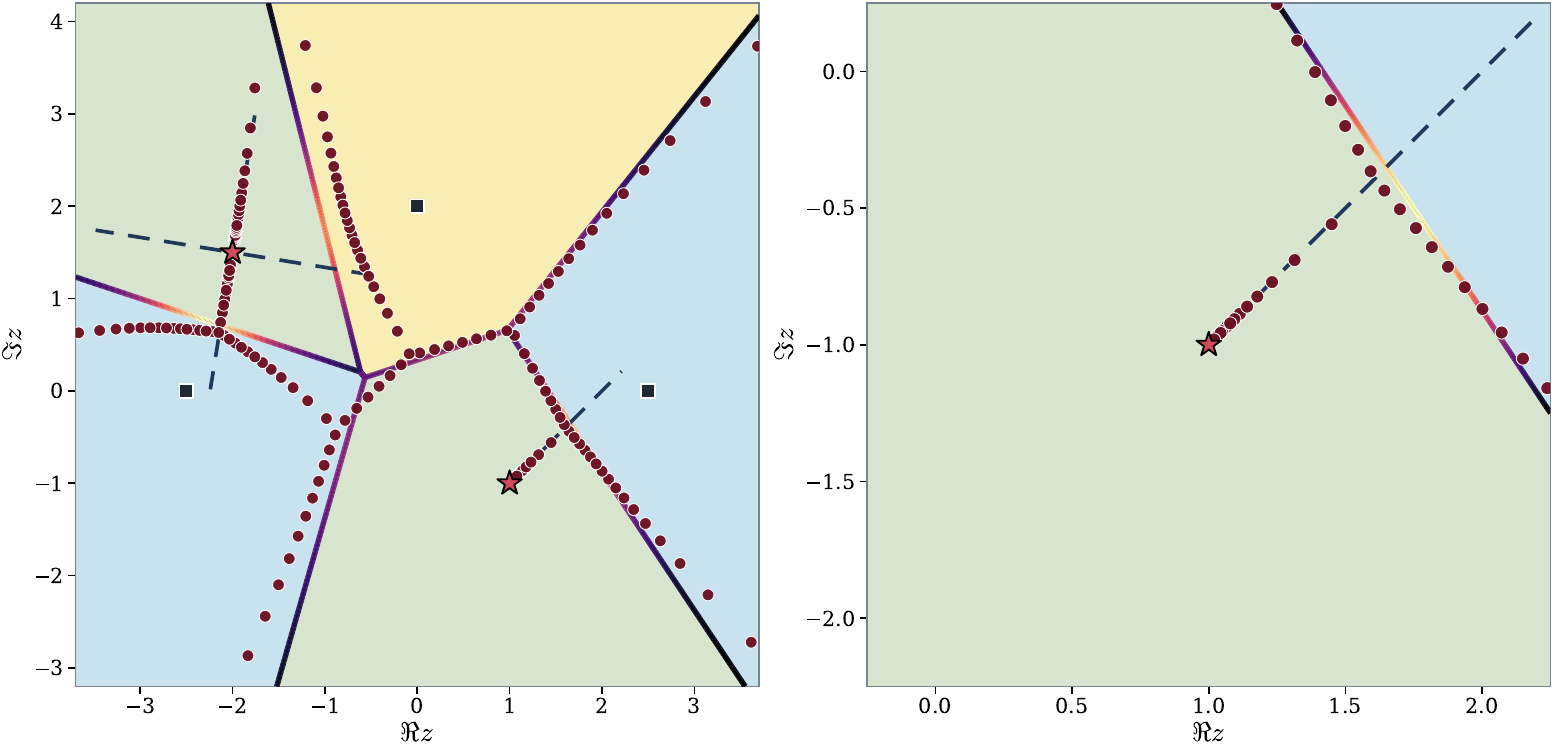}
\caption{Numerical illustration for $f(z)=\frac{1}{(z+5/2)(z-2i)(z-5/2)}\exp\!\left(\frac{(3-i)/2}{(z+2-3i/2)^2}+\frac{-1-i}{z-1+i}\right).$ Here $S/T$ has a double pole at $z=-2+3i/2$ and a simple pole at $z=1-i$. The small dark dots shown are finite zeros of $f^{(30)}$ in the displayed plotting windows; square markers denote ordinary poles and star markers denote finite essential singularities. The left panel shows the fixed-scale picture, and the right panel zooms near $z=1-i$. The background marks the Voronoi cells of the finite singular sites, the black line segments and rays are Voronoi edges, the shaded band schematically indicates the limiting edge density in \eqref{eq:measure}, and the dashed rays mark local saddle-change curves.}
\label{fig:iterated-derivative-voronoi}
\end{figure}

\subsection{Outline of the paper}\label{subsec:outline}
Section~\ref{sec:prelim} fixes notation for the hyperexponential class, Voronoi diagrams, and zero-counting measures. Section~\ref{sec:rec} constructs the polynomial sequence $B_n$, proves its local factorizations at finite singularities, and computes its degree growth. Section~\ref{sec:coef} proves the coefficient asymptotics on Voronoi cells, including the finite-annulus Picard--Lefschetz input for essential singularities. Section~\ref{sec:global} proves the fixed-scale Voronoi law. Section~\ref{sec:microscopic} proves the microscopic laws at essential singularities, and Section~\ref{sec:sublinear} proves the chamberwise sublinear zero-counting and final-set statements inside essential Voronoi cells. Section~\ref{sec:consequences} records the resulting final-set and cluster corollaries, Section~\ref{sec:open} lists open problems, and Appendix~\ref{app:mixed} gives the elementary characterization of the hyperexponential functions used here.

\subsection*{Acknowledgements}
We are grateful to Rikard B\o gvad and Per Alexandersson for valuable discussions.

\section{Notation and preliminaries}\label{sec:prelim}

\subsection{Polynomials, zeros, and orders}

 For a nonzero polynomial $R\in\C[z]$, we write $\deg R$ for its degree, $\Z(R)$ for its set of distinct finite zeros, $\rad(R)$ for its \emph{squarefree part} (the monic polynomial whose zeros are exactly the elements of $\Z(R)$, each with multiplicity one), and $\lc(R)$ for its leading coefficient. For a meromorphic function $F$ and a point $a\in\C$, we write $\ord_a F$ for the order of $F$ at $a$ (positive for a zero, negative for a pole, and zero otherwise).

\subsection{Voronoi diagrams}\label{sec:voronoi}

Let $\Sigma=\{a_1,\dots,a_N\}\subset\C$ be a finite set of points called \emph{sites}. The \emph{Voronoi cell} of $a_i$ is
$$\mathcal V_i:=\bigl\{z\in\C:|z-a_i|\le|z-a_j|\text{ for all } 1\le j\le N\bigr\},$$
and $\mathcal V_i^\circ$ denotes its interior, consisting of all points for which $a_i$ is the unique nearest site. For $i\ne j$, define
$$E_{ij}:=
\begin{cases}
\mathcal V_i\cap\mathcal V_j, & \text{if } \mathcal V_i\cap\mathcal V_j \text{ is one-dimensional},\\
\varnothing, & \text{otherwise}.
\end{cases}$$
Each nonempty $E_{ij}$ lies on the perpendicular bisector of $[a_i,a_j]$. The \emph{Voronoi diagram} of $\Sigma$, denoted $\operatorname{Vor}(\Sigma)$, is the union of the edges $E_{ij}$ and the Voronoi vertices. We write $\dd\ell$ for arclength measure on a curve.

\subsection{Logarithmic potentials and zero-counting measures}
\label{sec:pottheory}

A subharmonic function $u:\C\to[-\infty,\infty)$ has a distributional Laplacian $\Delta u$, and in our normalization $(2\pi)^{-1}\Delta u$ is a nonnegative Radon measure; see Ransford \cite{Ra}. For a nonzero polynomial $R$ of degree $n\ge1$, the function $\frac{1}{n}\log|R|$ is subharmonic and its distributional Laplacian is $2\pi$ times the \emph{normalized zero-counting measure}
$$\mu_R:=\frac{1}{n}\sum_{\zeta:\,R(\zeta)=0}\operatorname{mult}_\zeta(R)\,\delta_\zeta,$$
where $\delta_\zeta$ denotes the unit point mass at $\zeta$. We say that a sequence of measures $\mu_n$ converges \emph{vaguely} on $\C$ to $\mu$ if $\int\varphi\dd\mu_n\to\int\varphi\dd\mu$ for every continuous function $\varphi$ with compact support. We say that $\mu_n$ converges \emph{weakly} on $\C$ to $\mu$ if this holds for every bounded continuous function $\varphi:\C\to\mathbb R$, and \emph{weakly} on $\widehat\C$ if it holds for every continuous function $\varphi:\widehat\C\to\mathbb R$, where $\widehat\C=\C\cup\{\infty\}$ is the Riemann sphere. We use \(\xrightarrow{\;v\;}\) and \(\xrightarrow{\;w\;}\), when they appear, for vague and weak convergence, respectively, on the space specified in the surrounding text.

Here and below, $D\Subset U$ means that $D$ is a relatively compact subset of $U$, and a subscript on an $O$-term indicates the permitted dependence of the implied constant.

We use $\log^+x:=\max\{\log x,0\}$ for $x>0$ and $\log^+0:=0$.

For a finite Borel measure $\mu$ on $\C$ for which the integral below converges absolutely, we write
$$C_\mu(z):=\int_{\C}\frac{\dd\mu(\zeta)}{z-\zeta},\qquad z\in\C\setminus\operatorname{supp}\mu,$$
for its Cauchy transform. If $\phi:X\to Y$ is measurable and $\mu$ is a measure on $X$, then $\phi_\#\mu$ denotes the pushforward measure on $Y$.

\subsection{Notation for hyperexponential functions}

Let $P,Q,S,T\in\C[z]$ with $P,Q\ne0$, $\gcd(P,Q)=\gcd(S,T)=1$, and $T$ nonconstant. Consider
\begin{equation}\label{eq:fdef}
f(z):=\frac{P(z)}{Q(z)}\exp\!\left(\frac{S(z)}{T(z)}\right).
\end{equation}
Then $f$ is holomorphic on $\C\setminus\bigl(\Z(T)\cup \Z(Q)\bigr)$. Each point of $\Z(T)$ is an essential singularity of $f$, while each point of $\Z(Q)\setminus \Z(T)$ is a pole.
To avoid recording an inessential constant rational prefactor, we normalize the case $P/Q\equiv\alpha\ne0$: choose $c\in\C$ with $e^c=\alpha$ and replace $(P,Q,S)$ by $(1,1,S+cT)$, which leaves $f$ unchanged. Thus, after this normalization,
\begin{equation}\label{eq:prefactor}
\frac{P}{Q}\equiv 1\qquad\text{or}\qquad\frac{P}{Q}\not\equiv\text{const}.
\end{equation}
Set
$$p:=\deg P,\qquad q:=\deg Q.$$
After common rescalings of $P,Q$ and of $S,T$, we assume that $T$ is monic and that $Q=1$ if $q=0$, whereas $Q$ is monic if $q>0$. Write
\begin{equation}\label{eq:Edef}
E(z):=\frac{S(z)}{T(z)}=H(z)+\frac{M(z)}{T(z)},\qquad \deg M<\deg T,
\end{equation}
where $H\in\C[z]$ is the polynomial part and $M\in\C[z]$ is the remainder. Since $M=S-HT$ and $\gcd(S,T)=1$, one has $\gcd(M,T)=1$; in particular $M\ne0$. Define
\begin{equation}\label{eq:hdef}
h:=
\begin{cases}
\deg H, & \text{if } H \text{ is nonconstant},\\
0, & \text{if } H \text{ is constant}.
\end{cases}
\end{equation}

Write $\check t:=|\Z(T)|$ for the number of distinct zeros of~$T$, and let $c_1,\dots,c_{\check t}$ be those zeros with multiplicities $m_1,\dots,m_{\check t}$. Since $\gcd(S,T)=1$, $m_j$ is also the pole order of $E=S/T$ at $c_j$. Put
$$T_0:=\rad(T)=\prod_{j=1}^{\check t}(z-c_j).$$
Write $\check q:=|\Z(Q)\setminus \Z(T)|$, and let $b_1,\dots,b_{\check q}$ be the distinct zeros of $Q$ that do not belong to $\Z(T)$, with multiplicities $\ell_1,\dots,\ell_{\check q}$. Set
$$Q_\ast(z):=\prod_{k=1}^{\check q}(z-b_k).$$
Set
$$P_T(z):=\prod_{j=1}^{\check t}(z-c_j)^{p_j},\qquad p_j:=\ord_{c_j}P,\qquad \nu_j:=\ord_{c_j}Q,$$
and put
\begin{equation}\label{eq:Psh}
P_\sharp:=\frac{P}{P_T}.
\end{equation} 
Then $P_\sharp$ is coprime to $T$, so $P_\sharp(c_j)\ne0$ for all $j$. The factor $P_T$ records the zeros of $P$ at $\Z(T)$, which contribute fixed local orders at the essential singularities; zeros of $P$ away from $\Z(T)$ are accounted for separately in the polynomial factor of $f^{(n)}$ (see~\eqref{eq:fnrep} below).

Define
\begin{equation}\label{eq:Wdef}
W(z):=T(z)\,T_0(z)\,Q_\ast(z)=\prod_{j=1}^{\check t}(z-c_j)^{m_j+1}\prod_{k=1}^{\check q}(z-b_k).
\end{equation}
It is monic, of degree
\begin{equation}\label{eq:ddef}
d:=\deg W=\deg T+\check t+\check q.
\end{equation}
The finite singular set of $f$ is
\begin{equation}\label{eq:Sigmadef}
\Sigma:=\Z(T)\cup \Z(Q)=\{a_1,\dots,a_N\},\qquad N:=\check t+\check q.
\end{equation}
We write $\rho(z):=\operatorname{dist}(z,\Sigma)=\min_{1\le i\le N}|z-a_i|$ for the Euclidean distance from $z$ to $\Sigma$.

Set
\begin{equation}\label{eq:Lambdadef}
\Lambda(z):=E'(z)+\frac{P_T'(z)}{P_T(z)}-\frac{Q'(z)}{Q(z)},
\end{equation}
and
\begin{equation}\label{eq:Udef}
U(z):=W(z)\,\Lambda(z).
\end{equation}
By construction, $U$ is a polynomial: at each $c_j\in \Z(T)$ the factor $(z-c_j)^{m_j+1}$ in $W$ cancels the pole of $E'$ of order $m_j+1$, while $P_T'/P_T$ and $Q'/Q$ have at most simple poles there; at each $b_k\in \Z(Q)\setminus \Z(T)$ the factor $z-b_k$ cancels the simple pole of $Q'/Q$.

Finally, set
\begin{equation}\label{eq:kappadef}
\kappa:=d+h-1,
\end{equation}
and
\begin{equation}\label{eq:sigmadef}
\sigma:=
\begin{cases}
0, & \text{if } h=0,\\
\log|h\tau_h|, & \text{if } h>0 \text{ and } H(z)=\tau_h z^h+\cdots.
\end{cases}
\end{equation}

The integer $h$ governs the growth of $e^{E(z)}$ at infinity: when $h=0$, no positive proportion of the zeros escapes to $\infty$; when $h\ge1$, the polynomial part of the exponent sends a fraction $h/\kappa$ of the normalized zero mass to $\infty$ in the spherical limit described later.

Since $d=\deg T+\check t+\check q\ge2$, one has $\kappa\ge1$.

\section{Properties of the polynomial factor \texorpdfstring{$B_n$}{B\_n}}\label{sec:rec}

\subsection{Recurrence for \texorpdfstring{$B_n$}{B\_n}}

\begin{prop}[recurrence for \(B_n\)]\label{prop:rec}
For $f(z)$ given by \eqref{eq:fdef}, there is a unique polynomial sequence $\{B_n\}_{n\ge0}$ such that
\begin{equation}\label{eq:fnrep}
f^{(n)}(z)=\frac{P_T(z)\,B_n(z)}{Q(z)\,W(z)^n}\,e^{E(z)},\qquad B_0=P_\sharp,
\end{equation}
and satisfying the recurrence
\begin{equation}\label{eq:rec}
B_{n+1}=W\,B_n'+(U-nW')\,B_n.
\end{equation}
Moreover, $W$ is the unique monic polynomial of least degree such that $W\Lambda\in\C[z]$.
\end{prop}

\begin{remark}
Since $f=P_T P_\sharp Q^{-1}e^E$, the representation \eqref{eq:fnrep} gives
\begin{equation}\label{eq:Anrep}
A_n(z):=\frac{f^{(n)}(z)}{f(z)}=\frac{B_n(z)}{P_\sharp(z)\,W(z)^n}
\end{equation}
on $\C\setminus(\Sigma\cup \Z(P_\sharp))$. The prefactor in \eqref{eq:fnrep} is holomorphic and nonvanishing on $\C\setminus\Sigma$, so the zeros of $f^{(n)}$ away from $\Sigma$ are exactly the zeros of $B_n$.
\end{remark}

\begin{proof}[Proof of Proposition~\ref{prop:rec}]
Write
$$f(z)=\frac{P_T(z)P_\sharp(z)}{Q(z)}\,e^{E(z)}.$$
The poles of $\Lambda$ are easy to read off locally. At a zero $c_j$ of $T$ of multiplicity $m_j$, the function $E'$ has a pole of order $m_j+1$, while $P_T'/P_T$ and $Q'/Q$ have at most simple poles there. At a zero $b_k$ of $Q$ outside $\Z(T)$, the only pole of $\Lambda$ is the simple pole of $-Q'/Q$. There are no other poles. Therefore the unique monic polynomial of least degree clearing all poles of $\Lambda$ is precisely $W$ from~\eqref{eq:Wdef}, so $U=W\Lambda\in\C[z]$.

Now define $B_0:=P_\sharp$ and $B_{n+1}$ by~\eqref{eq:rec}. Assume \eqref{eq:fnrep} holds for some $n$. Differentiating and using~\eqref{eq:Lambdadef} gives
\begin{align*}
f^{(n+1)}
&=\left(\frac{P_T B_n}{Q W^n}\right)'e^E+E'\frac{P_T B_n}{Q W^n}\,e^E \\
&=\frac{P_T}{Q W^{n+1}}\Bigl(W B_n'+W\Bigl(E'+\tfrac{P_T'}{P_T}-\tfrac{Q'}{Q}\Bigr)B_n-nW'B_n\Bigr)e^E \\
&=\frac{P_T}{Q W^{n+1}}\bigl(WB_n'+(U-nW')B_n\bigr)e^E \\
&=\frac{P_T B_{n+1}}{Q W^{n+1}}\,e^E.
\end{align*}
By induction, \eqref{eq:fnrep} holds for all $n$. Since~\eqref{eq:rec} with the initial value $B_0=P_\sharp$ determines the sequence uniquely, this proves the proposition.\qedhere
\end{proof}

\subsection{Local factorizations of \texorpdfstring{$B_n$}{B\_n} at finite singular points}\label{sec:local}

\begin{prop}[exact local factorizations]\label{prop:local}
\begin{enumerate}
\item For every zero $c_j$ of $T$ and every $n\ge0$,
$$B_n(c_j)=U(c_j)^n P_\sharp(c_j)\ne0.$$
Equivalently, $f^{(n)}e^{-E}$ has a meromorphic continuation to a neighborhood of $c_j$ with $\ord_{c_j}\!\bigl(f^{(n)}e^{-E}\bigr)=p_j-\nu_j-n(m_j+1)$; explicitly, there exists a holomorphic function $\phi_{j,n}$ near $c_j$ with $\phi_{j,n}(c_j)\ne0$ such that
\begin{equation}\label{eq:localessentialorder}
f^{(n)}(z)=(z-c_j)^{p_j-\nu_j-n(m_j+1)}\phi_{j,n}(z)\,e^{E(z)}.
\end{equation}
\item For every pole $b_k$ of $P/Q$ outside $\Z(T)$ and every $n\ge0$,
$$B_{n+1}(b_k)=-(\ell_k+n)\,W'(b_k)\,B_n(b_k),$$
so $B_n(b_k)\ne0$ for all $n$. Equivalently, $f^{(n)}$ is meromorphic near $b_k$ with $\ord_{b_k}f^{(n)}=-\ell_k-n$; explicitly, there exists a holomorphic function $\psi_{k,n}$ near $b_k$ with $\psi_{k,n}(b_k)\ne0$ such that
\begin{equation}\label{eq:localpoleorder}
f^{(n)}(z)=(z-b_k)^{-\ell_k-n}\psi_{k,n}(z).
\end{equation}
\end{enumerate}
\end{prop}

\begin{proof}[Proof of Proposition~\ref{prop:local}]
Fix~$j$. Since $W$ has a zero of order $m_j+1\ge2$ at $c_j$, both $W(c_j)$ and $W'(c_j)$ vanish. Evaluating~\eqref{eq:rec} at $c_j$ gives
$$B_{n+1}(c_j)=U(c_j)\,B_n(c_j).$$
Write
$$W(z)=(z-c_j)^{m_j+1}\widetilde W_j(z),\qquad\widetilde W_j(c_j)\ne0.$$
Because $\gcd(S,T)=1$, the Laurent expansion of $E$ at $c_j$ begins with a nonzero principal term
$$E(z)=\lambda_{j,m_j}(z-c_j)^{-m_j}+O\bigl((z-c_j)^{-m_j+1}\bigr),\qquad\lambda_{j,m_j}\ne0,$$
so
$$E'(z)=-m_j\lambda_{j,m_j}(z-c_j)^{-m_j-1}+O\bigl((z-c_j)^{-m_j}\bigr).$$
The other two terms in $\Lambda$ have at most simple poles at $c_j$; multiplying by $W$ annihilates them at $c_j$. Hence
$$U(c_j)=-m_j\lambda_{j,m_j}\widetilde W_j(c_j)\ne0.$$
Since $B_0=P_\sharp$ and $P_\sharp(c_j)\ne0$, the first claim follows by induction. The local form~\eqref{eq:localessentialorder} is just~\eqref{eq:fnrep} rewritten with the exact orders of $P_T$, $Q$, and $W$ at $c_j$.

Now fix~$k$. Since $b_k\notin \Z(T)$, the function $E'+P_T'/P_T$ is holomorphic at $b_k$, whereas
$$-\frac{Q'(z)}{Q(z)}=-\frac{\ell_k}{z-b_k}+O(1).$$
Write
$$W(z)=(z-b_k)\widehat W_k(z),\qquad\widehat W_k(b_k)=W'(b_k)\ne0.$$
Then $U(b_k)=-\ell_k W'(b_k)$, and evaluating~\eqref{eq:rec} at $b_k$ gives
$$B_{n+1}(b_k)=\bigl(U(b_k)-nW'(b_k)\bigr)B_n(b_k)=-(\ell_k+n)W'(b_k)\,B_n(b_k).$$
Since $b_k\notin \Z(P)$ by $\gcd(P,Q)=1$ and $b_k\notin \Z(T)$, one has $B_0(b_k)=P_\sharp(b_k)\ne0$. The nonvanishing follows by induction. Since $E$ is holomorphic at $b_k$, \eqref{eq:fnrep} yields~\eqref{eq:localpoleorder} after absorbing $e^E$ into the holomorphic unit.
\end{proof}

Since $T$ is nonconstant, \(\Z(T)\ne\varnothing\). The first part of Proposition~\ref{prop:local} therefore also shows that \(B_n\not\equiv0\) for every \(n\ge0\), so the degrees and leading coefficients used below are well defined.

\subsection{Degree growth and leading coefficients of \texorpdfstring{$B_n$}{B\_n}}\label{sec:deg}

To state the result in the case $h=0$, we need an auxiliary function. Since $H$ is constant when $h=0$, the function $f$ equals $e^H$ times the function $G$ obtained by removing $H$ from the exponent:
\begin{equation}\label{eq:Gdef}
G(z):=\frac{P(z)}{Q(z)}\exp\!\left(\frac{M(z)}{T(z)}\right),
\end{equation}
where $M$ is the remainder in~\eqref{eq:Edef}. Then $G$ has a Laurent expansion at infinity
$$G(z)=\sum_{\nu=-\infty}^{p-q}G_\nu z^\nu,\qquad G_{p-q}\ne0.$$
When $p\ge q$, let $J\ge1$ be the smallest integer such that $G_{-J}\ne0$. Such $J$ exists: if $G_{-J}=0$ for all $J\ge1$, then $G$ agrees near $\infty$ with a polynomial $R$. Hence
$$e^{M/T}=\frac{RQ}{P}$$
on a nonempty open subset of $\C\setminus\bigl(\Z(P)\cup\Z(Q)\cup \Z(T)\bigr)$. By the identity theorem on this connected domain, $e^{M/T}$ would be rational. This is impossible because $\gcd(M,T)=1$, so $M/T$ has a pole at every zero of $T$, and $e^{M/T}$ has an essential singularity at each such point.

\begin{prop}[degree growth]\label{prop:deg}
Let $\gamma_n:=\lc(B_n)$.
\begin{enumerate}
\item\label{it:deg-h} If $h>0$ and $H(z)=\tau_h z^h+\cdots$, then
$$\deg B_n=\deg B_0+n\kappa,\qquad\gamma_n=(h\tau_h)^n\gamma_0.$$
\item\label{it:deg-plq} If $h=0$ and $p<q$, then for every $n\ge0$,
$$\deg B_n=\deg B_0+n(d-1),\qquad\gamma_n=\gamma_0(-1)^n\frac{\Gamma(n+q-p)}{\Gamma(q-p)}.$$
Equivalently, $\gamma_{n+1}=(p-q-n)\gamma_n$ for all $n\ge0$.
\item\label{it:deg-pgq} If $h=0$ and $p\ge q$, then for $0\le n\le p-q$,
$$\deg B_n=\deg B_0+n(d-1),\qquad\gamma_n=\gamma_0\frac{(p-q)!}{(p-q-n)!},$$
whereas for $n\ge p-q+1$,
$$\deg B_n=q-\deg P_T-J+n(d-1),\qquad\gamma_n=G_{-J}(-1)^n\frac{\Gamma(n+J)}{\Gamma(J)}.$$
\end{enumerate}
In particular,
\[
\deg B_n=\kappa n+O(1),\qquad
\log|\gamma_n|=
\begin{cases}
n\sigma+O(1), & h>0,\\
\log n!+O(\log n), & h=0.
\end{cases}
\]
\end{prop}

\begin{proof}
If $h>0$, then $E'(z)=H'(z)+O(z^{-2})$, so the highest-degree term of $U=W\Lambda$ comes from $WH'$. Therefore
$$\deg U=d+h-1=\kappa,\qquad\lc(U)=h\tau_h.$$
In~\eqref{eq:rec} the term $UB_n$ has strictly larger degree than $WB_n'$ and $-nW'B_n$, so
$$\deg B_{n+1}=\deg B_n+\kappa,\qquad\gamma_{n+1}=(h\tau_h)\gamma_n.$$
This proves part~\ref{it:deg-h}.

Assume $h=0$. Since $H$ is constant, $f=e^H G$, and~\eqref{eq:fnrep} gives
$$B_n(z)=\frac{Q(z)W(z)^n}{P_T(z)}\,e^{-M(z)/T(z)}\,G^{(n)}(z).$$
As $z\to\infty$,
$$\frac{Q(z)W(z)^n}{P_T(z)}=z^{q+nd-\deg P_T}\bigl(1+O(z^{-1})\bigr),$$
while
$$e^{-M(z)/T(z)}=1+O(z^{-1}).$$
The leading Laurent exponent of the displayed expression for $B_n$ is obtained by adding $q+nd-\deg P_T$ to the leading Laurent exponent of $G^{(n)}$. Since $B_n$ is a polynomial by Proposition~\ref{prop:rec}, its degree is this exponent. Multiplication by $z^{q+nd-\deg P_T}(1+O(z^{-1}))$ and by $e^{-M/T}=1+O(z^{-1})$ cannot cancel the first nonzero Laurent term of $G^{(n)}$, so the leading coefficient is unchanged. Since $e^{M/T}=1+O(z^{-1})$ and $Q$ is monic, one has $G_{p-q}=\lc(P)=\lc(P_\sharp)=\gamma_0$.

If $p<q$, the leading term of $G$ is $G_{p-q}\, z^{p-q}$, so
$$G^{(n)}(z)=G_{p-q}(-1)^n\frac{\Gamma(n+q-p)}{\Gamma(q-p)}\,z^{p-q-n}+O(z^{p-q-n-1}).$$
This yields part~\ref{it:deg-plq}.

Assume now that $p\ge q$. For $0\le n\le p-q$, the leading term of $G^{(n)}$ comes from the polynomial part:
$$G^{(n)}(z)=G_{p-q}\frac{(p-q)!}{(p-q-n)!}\,z^{p-q-n}+O(z^{p-q-n-1}),$$
which gives the first formula in part~\ref{it:deg-pgq}. For $n\ge p-q+1$, every nonnegative-power term of $G$ is annihilated by $n$ derivatives, so the leading term comes from $G_{-J}z^{-J}$:
$$G^{(n)}(z)=G_{-J}(-1)^n\frac{\Gamma(n+J)}{\Gamma(J)}\,z^{-J-n}+O(z^{-J-n-1}).$$
This completes part~\ref{it:deg-pgq}. The final estimates follow from the displayed formulas and Stirling's formula.
\end{proof}

\section{Coefficient asymptotics on open Voronoi cells}\label{sec:coef}

\subsection{Translation generating function}

We start with the following identity.

\begin{prop}\label{prop:gen}
For $A_n=f^{(n)}/f$, $n=0,1,\dots$, and every $z\in\C\setminus(\Sigma\cup \Z(P_\sharp))$, one has
\begin{equation}\label{eq:gen}
\sum_{n\ge0}A_n(z)\frac{\xi^n}{n!}=\frac{P(z+\xi)Q(z)}{P(z)Q(z+\xi)}\,\exp\!\bigl(E(z+\xi)-E(z)\bigr).
\end{equation}
The series on the left-hand side, called the \emph{translation generating function}, has radius of convergence $\rho(z)=\operatorname{dist}(z,\Sigma)$.

 Consequently,
\begin{equation}\label{eq:limsup}
\limsup_{n\to\infty}\left|\frac{A_n(z)}{n!}\right|^{1/n}=\rho(z)^{-1}.
\end{equation}
\end{prop}

\begin{proof}
Since $z\notin\Sigma$, $E$ is holomorphic at $z$ and $Q(z)\ne0$; since $P=P_TP_\sharp$, $\Z(P_T)\subseteq\Sigma$, and $z\notin\Z(P_\sharp)$, one has $P(z)\ne0$. Thus $f(z)$ is finite and nonzero. For $|\xi|<\rho(z)$ one has $z+\xi\notin\Sigma$, hence Taylor's formula gives
$$f(z+\xi)=\sum_{n\ge0}f^{(n)}(z)\frac{\xi^n}{n!}=f(z)\sum_{n\ge0}A_n(z)\frac{\xi^n}{n!},$$
and therefore
$$\frac{f(z+\xi)}{f(z)}=\frac{P(z+\xi)Q(z)}{P(z)Q(z+\xi)}\,\exp\!\bigl(E(z+\xi)-E(z)\bigr).$$
This proves~\eqref{eq:gen} on $|\xi|<\rho(z)$. As a function of $\xi$, the right-hand side is holomorphic for $z+\xi\notin\Sigma$, and each point with $z+\xi\in\Sigma$ is an actual singularity. Indeed, if $z+\xi=b_k\in \Z(Q)\setminus \Z(T)$, then $Q(z+\xi)^{-1}$ has a pole and $P(z+\xi)\ne0$ by $\gcd(P,Q)=1$. If $z+\xi=c_j\in \Z(T)$, then $E(z+\xi)$ has a pole, so $\exp\!\bigl(E(z+\xi)-E(z)\bigr)$ has an essential singularity; multiplication by the rational factor $P(z+\xi)/Q(z+\xi)$, which has only finite order at $c_j$, cannot remove it. Thus the singularities are exactly the points with $z+\xi\in\Sigma$, the radius of convergence is $\rho(z)$, and \eqref{eq:limsup} follows from the Cauchy--Hadamard formula.
\end{proof}

\subsection{Generating function and local data at singular sites}
Set $C_n(z):=\frac{B_n(z)}{W(z)^n}$ for $n=0,1,\dots$. 
By \eqref{eq:Anrep} and Proposition~\ref{prop:gen}, for $z\in\C\setminus(\Sigma\cup \Z(P_\sharp))$ and $|\xi|<\rho(z)$,
\begin{equation}\label{eq:Cexplicit}
\mathcal C(z,\xi):=\sum_{n\ge0}C_n(z)\frac{\xi^n}{n!}
=\frac{Q(z)}{P_T(z)}\,\frac{P(z+\xi)}{Q(z+\xi)}\,\exp\!\bigl(E(z+\xi)-E(z)\bigr).
\end{equation}
The right-hand side is holomorphic on the domain $\{(z,\xi)\in(\C\setminus\Sigma)\times\C:\ |\xi|<\rho(z)\}$. Since $\C\setminus(\Sigma\cup\Z(P_\sharp))$ is dense in $\C\setminus\Sigma$ and each $C_n=B_n/W^n$ is holomorphic on $\C\setminus\Sigma$, the Taylor coefficients at $\xi=0$ of the right-hand side agree with $C_n(z)/n!$ on a dense set of $z$ and hence on all of $\C\setminus\Sigma$. Thus \eqref{eq:Cexplicit} holds for every $z\in\C\setminus\Sigma$ and $|\xi|<\rho(z)$. For fixed $z$, the same right-hand side is holomorphic on $\C\setminus(\Sigma-z)$ and furnishes the analytic continuation of $\mathcal C(z,\cdot)$ away from the translated singular set.

\medskip
For a fixed site $a_i\in\Sigma$, set
\begin{equation}\label{eq:local-orders-unified}
m_i:=\ord_{a_i}T,\qquad p_i:=\ord_{a_i}P_T,\qquad r_i:=\ord_{a_i}Q,\qquad \beta_i:=p_i-r_i.
\end{equation}
Here $m_i,p_i,r_i,\beta_i$ refer to the fixed site $a_i\in\Sigma$; they are the same local orders as before, now indexed by the unified enumeration $a_1,\dots,a_N$ of $\Sigma$. Thus $m_i=0$ exactly at the ordinary-pole sites $\Z(Q)\setminus\Z(T)$, where $p_i=0$ and $r_i\ge1$, while $m_i\ge1$ exactly at the essential sites $\Z(T)$, even if $Q$ also vanishes there. Also write
\begin{equation}\label{eq:local-factors-unified}
P_T(w)=(w-a_i)^{p_i}\widetilde P_i(w),\qquad
Q(w)=(w-a_i)^{r_i}\widetilde Q_i(w),\qquad
\widetilde P_i(a_i)\widetilde Q_i(a_i)\ne0.
\end{equation}
If $m_i\ge1$, write
\begin{equation}\label{eq:E-local-unified}
E(w)=\sum_{s=1}^{m_i}\lambda_{i,s}(w-a_i)^{-s}+E_i^{\mathrm{reg}}(w),\qquad \lambda_{i,m_i}\ne0,
\end{equation}
whereas if $m_i=0$ we set $E_i^{\mathrm{reg}}:=E$ and interpret sums over $1\le s\le m_i$ as empty.

\subsection{Wright-type expansions via Picard--Lefschetz analysis}
For a domain $D\subset\C$, write $\mathcal O(D)$ for the algebra of holomorphic functions on $D$, and for a function $G(\zeta)=\sum_{n\ge0}g_n\zeta^n$ holomorphic near $0$ write $[\zeta^n]G(\zeta):=g_n$. Relative chains below are singular chains with integer coefficients, represented by piecewise $C^1$ chains; a relative one-cycle in $(X,B)$ has boundary in $B$, and two such cycles are homologous if they differ by the boundary of a two-chain in $X$ plus a one-chain in $B$. For a holomorphic phase $\Phi$, downward and upward flows are the negative and positive gradient flows of $\Re\Phi$; an exit arc is a boundary arc along which the downward field points out of the annulus; a separatrix is a maximal such trajectory issuing from a critical point; and a thimble is the compact relative one-cycle obtained by joining the two separatrices through a simple saddle.

For the relative homology groups used below, fix a triangulation of each bordered surface for which $A$, $B$, and $A\cap B$ are subcomplexes. The intersection pairing
\[
H_1(X,B;\mathbb Z)\times H_1(X,A;\mathbb Z)\longrightarrow\mathbb Z
\]
is computed after a small collar perturbation making the representatives transverse in $X^\circ$. Poincar\'e--Lefschetz duality for the triad $(X;A,B)$ makes this pairing perfect in the present surfaces. For an annulus whose exit part $B$ consists of $m+1$ intervals and whose entrance part $A$ also consists of $m+1$ intervals, both relative groups are free of rank $m+1$: the long exact sequence of the pair gives
\[
0\to H_1(X)\to H_1(X,B)\to \widetilde H_0(B)\to0,
\]
and the same formula holds with $A$ in place of $B$. Thus a transverse intersection matrix with determinant $\pm1$ proves that the corresponding thimbles are integral bases.

The finite-annulus Picard--Lefschetz input used later is contained in Lemmas~\ref{lem:wright-exit-arcs}--\ref{lem:filtered-wright}. These lemmas are not quoted from the literature in the finite-annulus, moving-boundary form needed here; the proofs below construct the annulus, the boundary split, the thimble orientations, the coefficient-contour class, and the filtered representative. The subsequent setup defines the local Wright phase and its leading Stokes and anti-Stokes sets. On compact subsets of one connected component of the complement of the leading Stokes set, the lemmas give a parameter-uniform exit-boundary decomposition, dual integral thimble bases, locally constant coefficient-contour coordinates, and a filtered representative with a uniform height gap below every saddle whose thimble has nonzero coefficient in the coefficient-contour decomposition. Intersections with the coefficient-contour class are homological. When a geometric sign is needed, the core class means the homology class of a positively oriented circle $|\tau|=\text{const}$ in the annulus, represented transversely. The arguments exclude boundary tangencies, critical collisions, and saddle connections only after a compact subchamber has been fixed. The finite perturbation parameter records the lower principal terms and the logarithmic perturbation; on each compact subchamber, and for sufficiently small parameter depending on that subchamber, the perturbed family is isotopic to the limiting phase. No uniform finite-parameter isotopy is asserted on a whole chamber. The expansion is used on compact subsets of a component of the complement of the leading Stokes set, including compact sets that meet the anti-Stokes set. No lower bound for an active saddle sum is asserted. Wright's coefficient method, Olver's saddle estimates, and Picard--Lefschetz terminology are used as background; see \cite{W1,W2}, \cite[Chs.~4 and~9]{Olver}, and \cite{Pham}.

The later dependence on this Picard--Lefschetz material is separated as follows.
\begin{enumerate}
\item The fixed-scale law uses the ordinary-pole asymptotics in Theorem~\ref{prop:local-coef-unified}\textup{(1)} and, on essential cells, Corollary~\ref{cor:essential-l1}. For the latter, the Picard--Lefschetz input is used only to obtain the Wright upper bound, one-saddle lower points on a dense open set, and the resulting local \(L^1\)-control; no zero-freeness of an active saddle sum or explicit visibility formula is used in the fixed-scale proof.
\item The microscopic laws in Section~\ref{sec:microscopic} use the compact deformation and perturbation results for the local microscopic phase, especially Lemmas~\ref{lem:micro-compact-pl}, \ref{lem:sublinear-cycle-perturb}, \ref{lem:micro-coalescing-upper}, \ref{lem:full-local-perturb}, and~\ref{lem:full-local-no-origin}. They are logically independent of the visible-chain criterion in Proposition~\ref{prop:active-wright-coefficients}.
\item The sublinear chamber theorem and the final-set corollary use the full chamber expansion \eqref{eq:Bn-local-unified-essential} and the explicit active set and signs from Proposition~\ref{prop:active-wright-coefficients}.
\end{enumerate}

Throughout this subsection, fix $m\ge1$, $\beta\in\mathbb Z$, a simply connected domain $D\subset\C$, functions $\lambda_1,\dots,\lambda_m\in\mathcal O(D)$ with $\lambda_m$ nonvanishing on $D$, and a function $R$ holomorphic on a neighborhood of $D\times\{1\}$ with $R(z,1)\ne0$ for all $z\in D$. For $z\in D$, set
$$F_z(\zeta):=(1-\zeta)^\beta \exp\!\Bigl(\sum_{s=1}^m \lambda_s(z)(1-\zeta)^{-s}\Bigr)R(z,\zeta),$$
and assume that $F_z$ extends holomorphically to a neighborhood of $\{|\zeta|\le1\}\setminus\{1\}$, with $\zeta=1$ as its unique singularity on $|\zeta|=1$. Assume moreover that for every compact $L\Subset D$ there exists $\varrho_L>1$ such that $(z,\zeta)\mapsto F_z(\zeta)$ is holomorphic on a neighborhood of $L\times\bigl(\{|\zeta|\le\varrho_L\}\setminus\{1\}\bigr)$. Choose a holomorphic branch $\eta$ on $D$ with
$$\eta(z)^{m+1}=m\lambda_m(z),$$
choose a holomorphic square root $\eta^{1/2}$ on $D$, let $\omega_\nu:=e^{2\pi i\nu/(m+1)}$, $0\le\nu\le m$, and fix square roots $\omega_\nu^{1/2}$ once and for all. Put
\begin{equation}\label{eq:wright-stokes-sets}
\mathfrak S:=\bigcup_{0\le\nu<\mu\le m}\{z\in D:\Im((\omega_\nu-\omega_\mu)\eta(z))=0\},\qquad
\mathfrak A:=\bigcup_{0\le\nu<\mu\le m}\{z\in D:\Re((\omega_\nu-\omega_\mu)\eta(z))=0\}.
\end{equation}
We call $\mathfrak S$ the \emph{leading Stokes set}, $\mathfrak A$ the \emph{leading anti-Stokes set}, and the connected components of $D\setminus\mathfrak S$ the (leading) \emph{Stokes chambers}. The set $\mathfrak S$ is the limiting critical-value Stokes locus, i.e. the locus where two limiting saddle values have equal imaginary parts; $\mathfrak A$ is where limiting saddle heights are equal. If $D\setminus\mathfrak S=\varnothing$, the chamber statements below are void. For $z\in D$ and $n\ge1$, define
$$\Phi_{z,n}(t):=\sum_{s=1}^m \lambda_s(z)t^{-s}-(n+1)\log(1-t),$$
where $\log(1-t)$ denotes the principal branch near $t=0$.

\begin{lem}[landing and continuation of level-gradient separatrices]\label{lem:gradient-landing}
Let $X$ be a compact bordered Riemann surface with piecewise $C^1$ boundary, let $\Phi=h+ig$ be holomorphic on a neighborhood of $X$, and assume that all critical points of $\Phi$ in $X^\circ$ are simple and have distinct $g$-values. Let $A\cup B=\partial X$ be a decomposition into finitely many closed boundary arcs, meeting only at endpoints, and let $I\subset\mathbb R$ be a compact interval which contains the $g$-values of the critical points and is disjoint from the $g$-values of the endpoints $A\cap B$. Suppose that, on $\partial X\cap g^{-1}(I)$, the field $-\nabla h$ has normal component bounded away from zero, points outward on the relative interiors of the $B$-arcs, and points inward on the relative interiors of the $A$-arcs. Then every downward separatrix from a critical point hits the relative interior of $B$ in finite time and transversely, and no downward saddle connection occurs. The corresponding upward statement holds with $A$ and $B$ interchanged.

If the data depend $C^1$-continuously on a parameter in a compact set, the boundary decomposition is locally trivial, and the preceding hypotheses hold with uniform margins, then the landed separatrix arcs continue by ambient isotopy under parameter variation. In particular their relative homology classes, and the corresponding intersection numbers with a fixed relative cycle transported by the same boundary trivialization, are locally constant.
\end{lem}

\begin{proof}
The Cauchy--Riemann equations give $(\pm\nabla h)g=0$, so a gradient trajectory remains on one $g$-level. Consider a downward separatrix. Its $g$-level belongs to $I$. In a collar of $\partial X\cap g^{-1}(I)$ the assumed normal sign prevents asymptotic approach to an $A$-arc, while entry into the collar of a $B$-arc forces finite-time transverse exit because the outward normal component is bounded below. If the separatrix does not hit $B$, its omega-limit is a nonempty compact connected subset of $X^\circ$. The function $h$ decreases strictly away from critical points and has a finite limit along the trajectory. A noncritical point in the omega-limit would have a flow box in which each sufficiently close visit decreases $h$ by a fixed positive amount, impossible for infinitely many visits. Thus the omega-limit is contained in the finite critical set, hence is a single critical point. Constancy of $g$ would then give equality of the $g$-values of two critical points; convergence back to the initial saddle is also impossible after $h$ has decreased along the outgoing branch. Therefore the separatrix lands on $B$, and the same $g$-separation excludes downward saddle connections. The proof for upward separatrices is identical with the signs reversed.

For the parameter statement, the implicit function theorem gives the moving endpoints of $A\cap B$ and hence a local trivialization of the boundary triads. Critical points and local separatrix germs vary smoothly. The preceding argument is uniform, so a landed branch can fail to continue only by hitting a corner, encountering a boundary tangency, colliding with another critical point, or forming a saddle connection. These alternatives are excluded by the uniform hypotheses and distinct critical $g$-values. The isotopy extension theorem then gives ambient isotopies of the arcs, and homological intersection numbers are invariant under these isotopies.
\end{proof}

\begin{lem}[uniform annular exit and flow data]\label{lem:wright-exit-arcs}
Let $V$ be a connected component of $D\setminus\mathfrak S$, and let $U_0$ be connected with $\overline{U_0}\Subset V$. Put
$$\tau_\nu(z):=\omega_\nu\eta(z),\qquad \psi_z(\tau):=\lambda_m(z)\tau^{-m}+\tau.$$
Choose $\Lambda>0$ such that $|\tau_\nu(z)|<\Lambda/4$ on $\overline{U_0}$ for every $\nu$. Then there are constants $C_0>0$, $0<\delta<\Lambda<\Lambda_1$, $q_0>0$, and $b>0$, with $q_0\Lambda_1<1$, for which the following assertions hold. Set
$$X:=\{\delta\le |\tau|\le\Lambda_1\},\qquad \mathcal P:=\overline{U_0}\times[0,q_0].$$
For $0<q\le q_0$ define
$$\widehat\Phi_{z,q}(\tau):=\sum_{s=1}^m q^{m-s}\lambda_s(z)\tau^{-s}-(q^{-1}+q^m)\log(1-q\tau),$$
using the Taylor branch of the logarithm on $X$, and set $\widehat\Phi_{z,0}:=\psi_z$. Write
$$h_{z,q}:=\Re\widehat\Phi_{z,q},\qquad g_{z,q}:=\Im\widehat\Phi_{z,q},$$
and let $\mathbf n$ be the outward unit normal of the annulus $X$. For every $(z,q)\in\mathcal P$ the set
$$
B_{z,q}:=\overline{\{\tau\in\partial X:\ \partial_{\mathbf n}h_{z,q}(\tau)<0,\ |g_{z,q}(\tau)|\le C_0+1\}}
$$
has exactly $m+1$ connected components: $m$ compact intervals on $|\tau|=\delta$ and one compact interval on $|\tau|=\Lambda_1$. Put
$$A_{z,q}:=\overline{\partial X\setminus B_{z,q}}.$$
Then $A_{z,q}\cup B_{z,q}=\partial X$, $A_{z,q}\cap B_{z,q}$ is the finite set of endpoints of these intervals, and $A_{z,q}$ also has $m+1$ interval components. The endpoints are transverse intersections with the levels $g_{z,q}=\pm(C_0+1)$, the pairs $(A_{z,q},B_{z,q})$ vary locally trivially in $(z,q)$ as boundary subcomplexes, and no boundary tangency of $\pm\nabla h_{z,q}$ occurs on $\partial X\cap\{|g_{z,q}|\le C_0+1\}$. On $\partial X\cap\{|g_{z,q}|\le C_0\}$, the downward field $-\nabla h_{z,q}$ points strictly outward exactly on the relative interiors of $B_{z,q}$ and strictly inward on the relative interiors of $A_{z,q}$; the upward field has the opposite signs.

For every $(z,q)\in\mathcal P$, the phase $\widehat\Phi_{z,q}$ has exactly $m+1$ simple critical points in $X$, labelled by
$$\tau_\nu(z,q)=\tau_\nu(z)+O_{U_0}(q),\qquad 0\le\nu\le m,$$
with $\tau_\nu(z,0)=\tau_\nu(z)$. Their critical imaginary parts are separated by a positive constant independent of $(z,q)$, and $|g_{z,q}(\tau_\nu(z,q))|\le C_0$. Every maximal downward separatrix from a critical point hits the relative interior of $B_{z,q}$ in finite time and transversely, and no downward trajectory joins two critical points. Every maximal upward separatrix from a critical point hits the relative interior of $A_{z,q}$ in finite time and transversely, and no upward trajectory joins two critical points. Moreover,
$$\sup_{\tau\in B_{z,q}}\Re\widehat\Phi_{z,q}(\tau)
\le \min_{0\le\nu\le m}\Re\widehat\Phi_{z,q}(\tau_\nu(z,q))-b$$
uniformly on $\mathcal P$.
\end{lem}

\begin{proof}
Since $\overline{U_0}\Subset V$, there is $a_0>0$ such that
$$\left|\Im\left(\frac{m+1}{m}(\omega_\nu-\omega_\mu)\eta(z)\right)\right|\ge a_0,
\qquad z\in\overline{U_0},\quad \nu\ne\mu.$$
For $q=0$, the critical equation is $\tau^{m+1}=m\lambda_m(z)$, so the critical points are $\tau_\nu(z)$ and their critical values are $((m+1)/m)\omega_\nu\eta(z)$. Choose $C_0$ larger than the absolute values of all these limiting critical imaginary parts on $\overline{U_0}$ plus $4$.

Because $\eta$ is nonvanishing on $\overline{U_0}$, choose $\delta>0$ so small that all $\tau_\nu(z)$ lie in $\{4\delta<|\tau|<\Lambda/4\}$. On the inner circle $\tau=\delta e^{i\theta}$, with outward normal $-\partial/\partial r$,
\begin{align*}
g_{z,0}(\delta e^{i\theta})
&=\delta^{-m}\Im(\lambda_m(z)e^{-im\theta})+O_{U_0}(\delta^{-m+1}),\\
\partial_{\mathbf n}h_{z,0}(\delta e^{i\theta})
&=m\delta^{-m-1}\Re(\lambda_m(z)e^{-im\theta})+O_{U_0}(\delta^{-m}).
\end{align*}
Uniformly in $z$, the conditions $|g_{z,0}|\le C_0+2$ and $\partial_{\mathbf n}h_{z,0}<0$ therefore hold on exactly $m$ intervals, centered at the directions where $\lambda_m(z)e^{-im\theta}$ is negative real. Their endpoints are the transverse intersections $g_{z,0}=\pm(C_0+2)$, and after replacing $C_0+2$ by $C_0+1$ the same transversality persists. At every inner-boundary tangency $\partial_{\mathbf n}h_{z,0}=0$, one has $|g_{z,0}|\ge c\delta^{-m}$ for a constant $c>0$ independent of $z$; decreasing $\delta$ puts all such tangencies outside $|g_{z,0}|\le C_0+2$.

On the outer circle $\tau=\Lambda_1e^{i\theta}$,
\begin{align*}
g_{z,0}(\Lambda_1 e^{i\theta})&=\Lambda_1\sin\theta+O_{U_0}(\Lambda_1^{-m}),\\
\partial_{\mathbf n}h_{z,0}(\Lambda_1 e^{i\theta})&=\cos\theta+O_{U_0}(\Lambda_1^{-m-1}),\\
h_{z,0}(\Lambda_1 e^{i\theta})&=\Lambda_1\cos\theta+O_{U_0}(\Lambda_1^{-m}).
\end{align*}
For $\Lambda_1$ large, the only outer exit component with $|g_{z,0}|\le C_0+2$ is the interval centered at the negative real direction. All outer tangencies satisfy $|g_{z,0}|\ge \Lambda_1/2$, hence lie outside the same strip. Increasing $\Lambda_1$ and, if necessary, decreasing $\delta$, the $m$ inner exit intervals and the one outer exit interval satisfy
$$h_{z,0}\le \min_\nu h_{z,0}(\tau_\nu(z))-4b_0$$
for some $b_0>0$, uniformly for $z\in\overline{U_0}$.

After decreasing $q_0$, the point $1/q$ lies outside $X$ and the functions $\widehat\Phi_{z,q}$ converge to $\psi_z$ in $C^2(X)$ uniformly for $z\in\overline{U_0}$. Hence the boundary component counts, endpoint transversality, absence of tangencies in $|g|\le C_0+1$, inward/outward sign assertions, and exit sublevel estimates persist, with a fixed margin, for all $(z,q)\in\mathcal P$. The endpoint equations $g_{z,q}=\pm(C_0+1)$ and the sign condition on $\partial_{\mathbf n}h_{z,q}$ are transverse; the implicit function theorem gives smooth endpoint motion, and the isotopy extension theorem gives local trivializations of the triples $(X,A_{z,q},B_{z,q})$.

For the critical points,
$$\partial_\tau\widehat\Phi_{z,q}(\tau)=1-m\lambda_m(z)\tau^{-m-1}+O_{U_0}(q)$$
in $C^1(X)$. Since the graphs $z\mapsto\tau_\nu(z)$ are pairwise separated over $\overline{U_0}$, choose pairwise disjoint neighborhoods $N_\nu\subset\overline{U_0}\times X$ of these graphs and write $N_{\nu,z}:=\{\tau:(z,\tau)\in N_\nu\}$. On the compact fiberwise complement $X\setminus\bigcup_\nu N_{\nu,z}$, the quantity $|\partial_\tau\psi_z|$ has a uniform positive lower bound. Rouch\'e's theorem, applied fiberwise on the boundaries of the sets $N_{\nu,z}$, gives one simple critical point of $\widehat\Phi_{z,q}$ in each $N_{\nu,z}$ and none on the complement. The labels satisfy $\tau_\nu(z,q)=\tau_\nu(z)+O_{U_0}(q)$, and
$$\widehat\Phi_{z,q}(\tau_\nu(z,q))=\frac{m+1}{m}\omega_\nu\eta(z)+O_{U_0}(q),$$
so, after decreasing $q_0$ again, the critical imaginary parts are separated by at least $a_0/2$ and satisfy $|g_{z,q}(\tau_\nu(z,q))|\le C_0$. Thus two distinct saddles never lie on one $g_{z,q}$-level; since $g_{z,q}$ is constant along gradient trajectories, this directly excludes saddle connections in the landing argument below. Since the endpoints have $|g|=C_0+1$, any trajectory with $|g|\le C_0$ remains a positive distance, uniform over compact parameter sets, from the corner set $A_{z,q}\cap B_{z,q}$. The closeness of the perturbed critical heights to the limiting heights and the persisted exit sublevel estimate give the displayed boundary estimate after decreasing the margin and calling it $b>0$.

It remains to prove the landing assertions. Apply Lemma~\ref{lem:gradient-landing} to $X$, $\Phi=\widehat\Phi_{z,q}$, $I=[-C_0,C_0]$, and the boundary split $A_{z,q}\cup B_{z,q}=\partial X$. The hypotheses have been verified above uniformly on $\mathcal P$: the critical $g_{z,q}$-values lie in $I$ and are separated, the endpoints of $A_{z,q}\cap B_{z,q}$ have $|g_{z,q}|=C_0+1$, and the normal component of $-\nabla h_{z,q}$ on $\partial X\cap\{|g_{z,q}|\le C_0\}$ has the asserted signs with a uniform margin. Hence all downward separatrices land on the relative interior of $B_{z,q}$ in finite time and no downward saddle connection occurs. The upward assertions follow from the same lemma with the signs reversed.
\end{proof}

\begin{lem}[relative thimbles and contour coefficients]\label{lem:wright-thimble-basis}
Assume the notation and conclusions of Lemma~\ref{lem:wright-exit-arcs}. For each $(z,q)\in\mathcal P$, let $\Delta_\nu(z,q)$ be the downward thimble through $\tau_\nu(z,q)$, viewed as a relative one-cycle in $H_1(X,B_{z,q};\mathbb Z)$, and let $\nabla_\nu(z,q)$ be the upward thimble through the same saddle, viewed in $H_1(X,A_{z,q};\mathbb Z)$. The relative groups are those of the boundary triad
$$A_{z,q}\cup B_{z,q}=\partial X,\qquad A_{z,q}\cap B_{z,q}\text{ finite}.$$
Choose an orientation of each downward thimble and orient the corresponding upward thimble so that the local intersection at the common saddle is $+1$; for the thimble-thimble pairing, perturb boundary endpoints only inside boundary collars when needed. Then
$$\Delta_\nu(z,q)\cdot\nabla_\mu(z,q)=\delta_{\nu\mu}.$$
The classes $[\Delta_\nu(z,q)]$ form an integral basis of $H_1(X,B_{z,q};\mathbb Z)$, and $[\nabla_\nu(z,q)]$ form the dual integral basis of $H_1(X,A_{z,q};\mathbb Z)$. Hence, for the positively oriented circle $\Gamma_\Lambda:=\{|\tau|=\Lambda\}$,
$$[\Gamma_\Lambda]=\sum_{\nu=0}^m\sigma_\nu[\Delta_\nu(z,q)]
\quad\text{in }H_1(X,B_{z,q};\mathbb Z),$$
where
$$\sigma_\nu=[\Gamma_\Lambda]\cdot[\nabla_\nu(z,q)]\in\mathbb Z$$
is the homological intersection pairing. If this number is evaluated by crossing counts, use any sufficiently small transverse perturbation of the core representative inside $X$.
After these thimble orientations are transported through the locally trivial family of boundary triads, the integers $\sigma_\nu$ are independent of $(z,q)\in\mathcal P$, and they are not all zero.
\end{lem}

\begin{proof}
Fix $(z,q)$. The boundary decomposition in Lemma~\ref{lem:wright-exit-arcs} is a compact one-dimensional CW decomposition of $\partial X$: $A_{z,q}$ and $B_{z,q}$ are unions of closed intervals and meet only in their endpoints. Relative chains are taken with endpoints in the prescribed boundary part. Boundary endpoints are moved by an arbitrarily small collar perturbation away from $A_{z,q}\cap B_{z,q}$. Interior intersections used to evaluate a homological pairing are taken after an arbitrarily small transverse perturbation of one representative, supported away from $\partial X$ when the endpoints are fixed. These perturbations preserve relative homology classes and homological intersection numbers.

The downward thimble $\Delta_\nu$ is the union of the two downward separatrices from $\tau_\nu(z,q)$ to $B_{z,q}$, and the upward thimble $\nabla_\nu$ is the union of the two upward separatrices from the same saddle to $A_{z,q}$. Lemma~\ref{lem:wright-exit-arcs} gives finite-time transverse landing, so these are compact relative one-cycles.

Let $p_\nu=\tau_\nu(z,q)$. In a holomorphic Morse coordinate $u$ centered at $p_\nu$,
$$\widehat\Phi_{z,q}=\widehat\Phi_{z,q}(p_\nu)-u^2/2.$$
The local downward branches are the real $u$-axis and the local upward branches are the imaginary $u$-axis. The chosen orientation of $\Delta_\nu$ fixes the sign of the local real branch; orient $\nabla_\nu$ so that the local intersection number at $p_\nu$ is $+1$. If $\Delta_\nu$ and $\nabla_\mu$ meet away from the boundary, then $g_{z,q}$ is constant on both thimbles and
$$g_{z,q}(p_\nu)=g_{z,q}(p_\mu).$$
The uniform separation of critical imaginary parts gives $\nu=\mu$. For $\nu=\mu$, away from $p_\nu$ the downward thimble has $h_{z,q}<h_{z,q}(p_\nu)$ and the upward thimble has $h_{z,q}>h_{z,q}(p_\nu)$, so there is no second interior intersection. Boundary intersections have been removed by the collar perturbation. This proves the displayed intersection matrix.

The long exact sequence for $(X,B_{z,q})$ gives
$$0\to H_1(X;\mathbb Z)\to H_1(X,B_{z,q};\mathbb Z)
\to\widetilde H_0(B_{z,q};\mathbb Z)\to0,$$
because $H_1(B_{z,q};\mathbb Z)=0$ and $X$ is connected. Since $X$ is an annulus and $B_{z,q}$ has $m+1$ interval components, $H_1(X,B_{z,q};\mathbb Z)$ is free of rank $m+1$. The same argument gives $\operatorname{rank}H_1(X,A_{z,q};\mathbb Z)=m+1$.

Choose a triangulation of $X$ for which $A_{z,q}$, $B_{z,q}$, and $A_{z,q}\cap B_{z,q}$ are subcomplexes. Poincar\'e--Lefschetz duality for the oriented surface with boundary split into the complementary subcomplexes $A_{z,q}$ and $B_{z,q}$ gives a perfect integral pairing
$$H_1(X,B_{z,q};\mathbb Z)\times H_1(X,A_{z,q};\mathbb Z)\longrightarrow\mathbb Z.$$
In arbitrary integral bases of these free groups, let $J$ be the unimodular matrix of this pairing, and let $M_\Delta$ and $M_\nabla$ be the matrices whose columns are the downward and upward thimble classes. The intersection calculation gives
$$M_\Delta^{T}JM_\nabla=I_{m+1}.$$
All matrices have integral entries and $\det J=\pm1$, hence $\det M_\Delta=\det M_\nabla=\pm1$. Thus the downward thimbles form an integral basis and the upward thimbles form the dual basis.

It remains to justify parameter constancy. The transverse endpoint equations give, locally in $\mathcal P$, a $C^1$ trivialization of the boundary triads $(X,A_{z,q},B_{z,q})$, and the isotopy extension theorem realizes it by an ambient isotopy of $X$. Lemma~\ref{lem:gradient-landing}, applied with the uniform margins established in Lemma~\ref{lem:wright-exit-arcs}, gives isotopic continuation of the landed upward and downward separatrix arcs throughout each small parameter ball: a possible corner hit, boundary tangency in the strip, critical collision, or saddle connection is excluded respectively by $|g|=C_0+1$ at the corners, the boundary sign margin, the critical-point separation, and the distinct critical imaginary values. Thus the separatrix representatives, their orientations, and the upward-thimble classes vary by isotopy. The relative class of $\Gamma_\Lambda$ is fixed under the same trivialization, so the homological intersection $[\Gamma_\Lambda]\cdot[\nabla_\nu(z,q)]$ is a locally constant integer-valued function of $(z,q)$; since $\mathcal P$ is connected, it is constant on $\mathcal P$.

Finally, $H_1(B_{z,q};\mathbb Z)=0$, so $H_1(X;\mathbb Z)\to H_1(X,B_{z,q};\mathbb Z)$ is injective. The circle $\Gamma_\Lambda$ generates $H_1(X;\mathbb Z)$, hence its relative class is nonzero. Since the downward thimbles are a basis, the coordinates $\sigma_\nu$ of this class cannot all vanish.
\end{proof}

\begin{lem}[filtered deformation for the scaled Wright phase]\label{lem:filtered-wright}
Let $V$ be a connected component of $D\setminus\mathfrak S$, let $U_0$ be connected with $\overline{U_0}\Subset V$, and let $K_0\Subset U_0$. Put
$$\tau_\nu(z):=\omega_\nu\eta(z),\qquad \psi_z(\tau):=\lambda_m(z)\tau^{-m}+\tau.$$
Choose $\Lambda>0$ such that $|\tau_\nu(z)|<\Lambda/4$ on $\overline{U_0}$ for every $\nu$. Then there are $0<\delta<\Lambda<\Lambda_1$ and $q_0>0$, with $q_0\Lambda_1<1$, such that Lemmas~\ref{lem:wright-exit-arcs} and~\ref{lem:wright-thimble-basis} apply to $\widehat\Phi_{z,q}$ on $X=\{\delta\le |\tau|\le\Lambda_1\}$. Let
$$I:=\{\nu:\sigma_\nu\ne0\},\qquad
\mathcal H_{z,q}:=\max_{\nu\in I}\Re\widehat\Phi_{z,q}(\tau_\nu(z,q)).$$
There are pairwise disjoint neighborhoods $\mathcal N_\mu$ of the critical graphs $\{(z,q,\tau_\mu(z,q)):(z,q)\in K_0\times[0,q_0]\}$, each carrying a parameter-dependent holomorphic Morse coordinate for $\widehat\Phi_{z,q}$, and a number $c>0$ such that
$$C_{z,q}:=\sum_{\nu\in I}\sigma_\nu\Delta_\nu(z,q)$$
represents $[\Gamma_\Lambda]$ in $H_1(X,B_{z,q};\mathbb Z)$, contains no zero-coefficient thimble at chain level, has uniformly bounded piecewise $C^1$ length and multiplicity, and satisfies
$$\operatorname{supp}C_{z,q}\setminus\bigcup_{\nu\in I}\mathcal N_\nu(z,q)
\subset\{\tau\in X:\Re\widehat\Phi_{z,q}(\tau)\le\mathcal H_{z,q}-c\}.$$
Moreover one may choose singular two-chains $\mathcal S_{z,q}$ in $X$ and one-chains $E_{z,q}$ in $B_{z,q}$ such that
$$\Gamma_\Lambda-C_{z,q}=\partial\mathcal S_{z,q}+E_{z,q},\qquad
\partial E_{z,q}=-\partial C_{z,q},$$
$$\operatorname{supp}E_{z,q}\subset\{\tau\in B_{z,q}:\Re\widehat\Phi_{z,q}(\tau)\le\mathcal H_{z,q}-c\},$$
and $E_{z,q}$ has uniformly bounded piecewise $C^1$ length and multiplicity. Consequently, for every holomorphic one-form $\Omega$ on a neighborhood of $X$,
$$\int_{\Gamma_\Lambda}\Omega=\int_{C_{z,q}}\Omega+\int_{E_{z,q}}\Omega.$$
\end{lem}

\begin{proof}
Lemmas~\ref{lem:wright-exit-arcs} and~\ref{lem:wright-thimble-basis} give the boundary partition, absence of saddle connections, thimble bases, and parameter-constant integers $\sigma_\nu$; in particular $I\ne\varnothing$.

All estimates below are over the compact parameter set $K_0\times[0,q_0]$. The critical graphs are disjoint and their critical imaginary values are uniformly separated. The holomorphic Morse lemma with parameters gives pairwise disjoint neighborhoods $\mathcal N_\mu$ and coordinates $u$ in which
$$\widehat\Phi_{z,q}(\tau)=\widehat\Phi_{z,q}(\tau_\mu(z,q))-u^2/2.$$
Shrinking once, uniformly over the compact parameter set, we arrange that the $g_{z,q}$-ranges of the neighborhoods are pairwise disjoint, their boundaries meet the four local separatrix germs transversely, and every local downward branch exits its own $\mathcal N_\nu(z,q)$ through points satisfying
$$h_{z,q}\le h_{z,q}(\tau_\nu(z,q))-2\varepsilon$$
with one $\varepsilon>0$ independent of $(z,q)$ and $\nu$.

A thimble has constant $g_{z,q}$. Since the critical neighborhoods have disjoint $g_{z,q}$-ranges, an active thimble cannot enter a critical neighborhood other than its own. In its own Morse coordinate, the level of $g_{z,q}$ through the saddle is the union of the coordinate axes, and the downward branches are the real-axis branches pointing out of the neighborhood. Each such branch exits through points with
$$h_{z,q}\le h_{z,q}(\tau_\nu(z,q))-2\varepsilon,$$
and $h_{z,q}$ then decreases strictly along the branch until it reaches $B_{z,q}$. Hence every exterior point of an active thimble satisfies
$$\operatorname{supp}\Delta_\nu(z,q)\setminus\mathcal N_\nu(z,q)
\subset\{h_{z,q}\le h_{z,q}(\tau_\nu(z,q))-\varepsilon\}
\subset\{h_{z,q}\le\mathcal H_{z,q}-\varepsilon\},\qquad \nu\in I.$$
Lemma~\ref{lem:wright-exit-arcs} gives on $B_{z,q}$
$$h_{z,q}\le\min_\mu h_{z,q}(\tau_\mu(z,q))-b\le\mathcal H_{z,q}-b.$$
The support estimates follow with $c\le\min(\varepsilon,b)$.

By Lemma~\ref{lem:wright-thimble-basis},
$$[\Gamma_\Lambda]=\sum_{\nu=0}^m\sigma_\nu[\Delta_\nu(z,q)]
=\sum_{\nu\in I}\sigma_\nu[\Delta_\nu(z,q)]
=[C_{z,q}]$$
in $H_1(X,B_{z,q};\mathbb Z)$. The omitted thimbles have coefficient zero and are not present in the chain. The relative equality means that there are a singular two-chain $\mathcal S^0_{z,q}$ in $X$ and a one-chain $E^0_{z,q}$ in $B_{z,q}$ with
$$\Gamma_\Lambda-C_{z,q}=\partial\mathcal S^0_{z,q}+E^0_{z,q}.$$
No support control is asserted for this preliminary $E^0_{z,q}$.

The length and multiplicity of $C_{z,q}$ are uniformly bounded. Inside the Morse neighborhoods this follows from the uniform coordinate radii. Near the boundary it follows from the transverse landing flow boxes in Lemma~\ref{lem:wright-exit-arcs}. On the remaining compact part of $X$, the gradient has a positive lower bound. Parametrizing a branch by arclength gives $|dh_{z,q}/ds|=|\nabla h_{z,q}|$, so the length there is bounded by the uniform oscillation of $h_{z,q}$ divided by this lower bound. There are finitely many branches and the integers $\sigma_\nu$ are fixed.

We now construct the controlled boundary chain. Since $[C_{z,q}]=[\Gamma_\Lambda]$ lies in the image of $H_1(X;\mathbb Z)$, the connecting homomorphism sends $[C_{z,q}]$ to zero in $H_0(B_{z,q};\mathbb Z)$. Equivalently, the total coefficient of $\partial C_{z,q}$ on each interval component $J$ of $B_{z,q}$ is zero. Orient $J$, list the support of $\partial C_{z,q}|_J$ as $p_1,\dots,p_s$ in order, and write
$$\partial C_{z,q}|_J=\sum_{r=1}^s c_r p_r,\qquad \sum_{r=1}^s c_r=0.$$
If $s=0$, set $E_J=0$. If $s>0$, put
$$E_J:=\sum_{r=1}^{s-1}\left(\sum_{u=1}^r c_u\right)[p_r,p_{r+1}],$$
where $[p_r,p_{r+1}]$ is the oriented subinterval of $J$. Then $\partial E_J=-\partial C_{z,q}|_J$. Summing over the interval components of $B_{z,q}$ gives a one-chain $E_{z,q}\subset B_{z,q}$ with
$$\partial E_{z,q}=-\partial C_{z,q}.$$
The cumulative coefficients are bounded by $\sum_{\nu\in I}|\sigma_\nu|$, and the boundary intervals have uniformly bounded length by local triviality of the boundary partition. Thus $E_{z,q}$ has uniformly bounded piecewise $C^1$ length and multiplicity. Its support is contained in $B_{z,q}$, so Lemma~\ref{lem:wright-exit-arcs} gives
$$\operatorname{supp}E_{z,q}\subset\{h_{z,q}\le\mathcal H_{z,q}-c\}$$
after decreasing $c$.

The chains $E_{z,q}$ and $E^0_{z,q}$ have the same boundary. Since $H_1(B_{z,q};\mathbb Z)=0$, their difference bounds a singular two-chain in $B_{z,q}$. Absorbing that two-chain into $\mathcal S^0_{z,q}$ gives
$$\Gamma_\Lambda-C_{z,q}=\partial\mathcal S_{z,q}+E_{z,q}.$$
If $\Omega$ is holomorphic on a neighborhood of $X$, then $d\Omega=0$ and Stokes' theorem gives $\int_{\partial\mathcal S_{z,q}}\Omega=0$. Integrating the chain identity proves the formula.
\end{proof}

\begin{lem}[parameter-uniform Wright upper bound]\label{lem:wright-upper}
Let $K\Subset D$ be compact and set $\alpha:=m/(m+1)$ and
$$\theta:=-\frac{2\beta+m+2}{2(m+1)}.$$
There is a constant $C_K$ such that, for all $n\ge2$,
\begin{equation}\label{eq:wright-upper}
\sup_{z\in K}\log^+\left(n^{-\theta}\left|[\zeta^n]F_z(\zeta)\right|\right)\le C_K n^\alpha.
\end{equation}
\end{lem}

\begin{proof}
Choose $L$ with $K\Subset L\Subset D$ and use the annulus hypothesis to fix $\varrho>1$ for $L$. Since $R$ is holomorphic near $K\times\{1\}$, it is uniformly bounded for $z\in K$ and $|1-\zeta|$ sufficiently small. Put $\varepsilon_n:=n^{-1/(m+1)}$ and, for large $n$, $r_n:=1-2\varepsilon_n$; then $|1-\zeta|>\varepsilon_n$ on $|\zeta|=r_n$ and the Taylor branch of $\log(1-t)$ is valid on $|t|=\varepsilon_n$.

Cauchy's formula on $|\zeta|=r_n$, followed by Cauchy's theorem in
$$\{r_n\le |\zeta|\le\varrho\}\setminus\{|1-\zeta|<\varepsilon_n\},$$
gives an outer-circle integral $O_K(\varrho^{-n})$ and an integral over $|t|=\varepsilon_n$, $t=1-\zeta$, of
$$t^\beta R(z,1-t)\exp\!\left(\sum_{s=1}^m\lambda_s(z)t^{-s}-(n+1)\log(1-t)\right).$$
On this small circle,
$$|t^\beta|\le n^{|\beta|/(m+1)},\qquad
\Re\left(\sum_{s=1}^m\lambda_s(z)t^{-s}-(n+1)\log(1-t)\right)
\le C_K n^{m/(m+1)},$$
because $|t|^{-s}\le n^{m/(m+1)}$ for $s\le m$ and $(n+1)|\log(1-t)|=O(n^{m/(m+1)})$. The small circle has length $O(n^{-1/(m+1)})$, and the factor $\zeta^{-n-1}=(1-t)^{-n-1}$ is precisely the logarithmic part of the displayed phase. Hence
$$|[\zeta^n]F_z(\zeta)|\le C_K n^{A}\exp(C_K n^\alpha)$$
for some fixed $A$, uniformly on $K$ and all large $n$. Multiplication by $n^{-\theta}$ changes only the polynomial factor; taking $\log^+$ and increasing $C_K$ to absorb the finitely many remaining $n\ge2$ proves \eqref{eq:wright-upper}.
\end{proof}

\begin{lem}[sectorial parameter-uniform multi-saddle asymptotics]\label{lem:wright-multi}
Let $V$ be a connected component of $D\setminus\mathfrak S$, and let $K\Subset V$ be compact. Then there exists a connected open set $U$ with $K\Subset U\Subset V$ such that, for each $0\le\nu\le m$ and all sufficiently large $n$, there is a holomorphic function $t_{n,\nu}:U\to\C$ satisfying
$$\Phi_{z,n}'(t_{n,\nu}(z))=0,\qquad
t_{n,\nu}(z)=\omega_\nu\eta(z)\,n^{-1/(m+1)}+O_U\!\left(n^{-2/(m+1)}\right).$$
Set
\begin{equation}\label{eq:wright-multi-data}
\theta:=-\frac{2\beta+m+2}{2(m+1)},\qquad
A_\nu(z):=\frac{(\omega_\nu\eta(z))^{\beta}\,(\omega_\nu^{1/2}\eta^{1/2}(z))}{\sqrt{2\pi(m+1)}}\,R(z,1),
\qquad
\Xi_\nu(z;n):=\Phi_{z,n}(t_{n,\nu}(z)).
\end{equation}
Then each $A_\nu$ is holomorphic and nonvanishing on $U$, and each $\Xi_\nu(z;n)$ admits, by iterating the saddle equation, an asymptotic expansion to arbitrary finite order in descending powers of $n^{1/(m+1)}$, uniformly on $K$, with
\begin{equation}\label{eq:wright-multi-phase}
\Xi_\nu(z;n)=\frac{m+1}{m}\,\omega_\nu\eta(z)\,n^{m/(m+1)}
+O_K\!\left(n^{(m-1)/(m+1)}\right)
\end{equation}
uniformly for $z\in K$.

Write $\tau_\nu(z):=\omega_\nu\eta(z)$ and use the square roots in \eqref{eq:wright-multi-data} to set $\tau_\nu(z)^{1/2}=\omega_\nu^{1/2}\eta^{1/2}(z)$. Orient each local downward thimble at $t_{n,\nu}(z)$ so that, in the Morse coordinate $u$ normalized by
$$\Phi_{z,n}(t)=\Xi_\nu(z;n)-u^2/2,
\qquad
\left.\frac{\partial t}{\partial u}\right|_{u=0}=\frac{i}{\sqrt{\Phi_{z,n}''(t_{n,\nu}(z))}},$$
the real $u$-axis has its standard orientation; the square root of $\Phi_{z,n}''$ is the continuation whose leading term is
$$\sqrt{\frac{m+1}{\tau_\nu(z)}}\,n^{(m+2)/(2(m+1))}.$$
With these local thimble orientations, Lemma~\ref{lem:filtered-wright}, applied on $U\Subset V$ to the scaled coefficient circle, gives integers $\sigma_{\nu,U}$, not all zero. Because the branch of $\eta$ has been fixed on $D$, the limiting saddle labels $\tau_\nu(z)=\omega_\nu\eta(z)$ are single-valued on all of $V$; for large $n$ the finite-$q$ saddles are labelled by the perturbation of these branches. If $U_1,U_2\Subset V$ are two admissible choices, choose a compact connected set $L\Subset V$ containing $U_1\cup U_2$ and cover $L$ by finitely many subdomains to which Lemmas~\ref{lem:wright-exit-arcs}--\ref{lem:filtered-wright} apply with overlapping closures. On each overlap the locally trivial boundary triads identify the labelled thimbles, and Lemma~\ref{lem:gradient-landing} excludes corner hits, boundary tangencies, critical collisions and saddle connections during the continuation. The intersection numbers with the transported coefficient-core class therefore agree from one member of the finite chain to the next. Thus the integers are independent of the auxiliary compact chart and are denoted by $\sigma_{\nu,V}$; this notation records homological continuation on compact subsets of $V$ and does not assert a uniform finite-$q$ isotopy on all of $V$ at once. Put
$$I_V:=\{\nu:\sigma_{\nu,V}\ne0\}.$$
Then $I_V$ is nonempty. Proposition~\ref{prop:active-wright-coefficients} below computes the vanishing pattern of these integers explicitly; the present lemma uses only their homological definition and nontriviality. If
$$M_V(z;n):=\max_{\nu\in I_V}\Re\Xi_\nu(z;n),$$
then there are functions $\varepsilon_{n,\nu}:K\to\C$, $\nu\in I_V$, satisfying
$$\max_{\nu\in I_V}\sup_{z\in K}|\varepsilon_{n,\nu}(z)|=O_K\!\left(n^{-1/(m+1)}\right),$$
and, uniformly for $z\in K$,
\begin{equation}\label{eq:wright-multi}
[\zeta^n]F_z(\zeta)
=
n^\theta\sum_{\nu\in I_V}\sigma_{\nu,V}A_\nu(z)\exp\!\bigl(\Xi_\nu(z;n)\bigr)\bigl(1+\varepsilon_{n,\nu}(z)\bigr)
+\mathcal R_{n,V}(z),
\end{equation}
with
\begin{equation}\label{eq:wright-multi-rem}
\mathcal R_{n,V}(z)=
O_K\!\left(
n^\theta
\exp\!\bigl(M_V(z;n)-c_K n^{m/(m+1)}\bigr)
\right)
\end{equation}
for some $c_K>0$. The estimate is uniform on compact sets crossing leading anti-Stokes arcs. Changing one square-root branch in \eqref{eq:wright-multi-data} changes the displayed $A_\nu$ by a sign; changing the corresponding thimble orientation changes $\sigma_{\nu,V}$ by the same sign, so the products $\sigma_{\nu,V}A_\nu$ are intrinsic.
\end{lem}

\begin{proof}
Fix $K\Subset V$ and choose a connected open set $U$ with $K\Subset U\Subset V$. Put $\tau_\nu(z):=\omega_\nu\eta(z)$. The scaled saddle equation is
\begin{align*}
0
&=n^{-1}\Phi_{z,n}'\!\bigl(n^{-1/(m+1)}\tau\bigr)\\
&=1-\tau_\nu(z)^{m+1}\tau^{-m-1}+O_U\!\left(n^{-1/(m+1)}\right),
\end{align*}
uniformly for $\tau$ on compact subsets of $\C^\times$. Since
$$\partial_\tau\!\left(1-\tau_\nu(z)^{m+1}\tau^{-m-1}\right)\Big|_{\tau=\tau_\nu(z)}=\frac{m+1}{\tau_\nu(z)}\ne0,$$
the holomorphic implicit function theorem gives holomorphic saddles $t_{n,\nu}=n^{-1/(m+1)}\tau_{n,\nu}$ on $U$, with
$$\tau_{n,\nu}(z)=\tau_\nu(z)+O_U\!\left(n^{-1/(m+1)}\right),
\qquad
 t_{n,\nu}(z)=\tau_\nu(z)n^{-1/(m+1)}+O_U\!\left(n^{-2/(m+1)}\right).$$
Repeated substitution in the implicit equation gives the full expansion of $\tau_{n,\nu}$ and then of $\Xi_\nu$. Its leading term is
$$\lambda_m(z)\tau_\nu(z)^{-m}n^{m/(m+1)}+\tau_\nu(z)n^{m/(m+1)}
=\frac{m+1}{m}\,\tau_\nu(z)n^{m/(m+1)},$$
which is \eqref{eq:wright-multi-phase}.

By the annulus hypothesis applied to $\overline U$, choose $\varrho>1$ such that $(z,\zeta)\mapsto F_z(\zeta)$ is holomorphic on a neighborhood of $\overline U\times(\{|\zeta|\le\varrho\}\setminus\{1\})$. Choose $\Lambda>0$ so that every $\tau_\nu(z)$ lies in $|\tau|<\Lambda/4$ for $z\in\overline U$, and put $\varepsilon_n:=\Lambda n^{-1/(m+1)}$. For $0<r<1$ with $|1-\zeta|>\varepsilon_n$ on $|\zeta|=r$, Cauchy's formula gives
$$[\zeta^n]F_z=(2\pi i)^{-1}\int_{|\zeta|=r}F_z(\zeta)\zeta^{-n-1}\,\dd\zeta.$$
Cauchy's theorem in
$$\{r\le |\zeta|\le\varrho\}\setminus\{|1-\zeta|<\varepsilon_n\}$$
gives, with the boundary orientations induced by this punctured annulus,
\begin{equation}\label{eq:wright-contour-split}
[\zeta^n]F_z(\zeta)=\frac{1}{2\pi i}\int_{|\zeta|=\varrho}F_z(\zeta)\zeta^{-n-1}\,\dd\zeta
+\frac{1}{2\pi i}\int_{\Gamma_{\varepsilon_n}} t^\beta R(z,1-t)e^{\Phi_{z,n}(t)}\,\dd t,
\end{equation}
where $t=1-\zeta$ and $\Gamma_{\varepsilon_n}$ is the positively oriented circle $|t|=\varepsilon_n$. The sign is positive because the small boundary is clockwise in the $\zeta$-plane and $\dd\zeta=-\dd t$. The outer integral is $O_U(\varrho^{-n})$.

Set $q_n:=n^{-1/(m+1)}$ and
$$\widehat\Phi_{z,n}(\tau):=n^{-m/(m+1)}\Phi_{z,n}\!\bigl(q_n\tau\bigr).$$
Apply Lemma~\ref{lem:filtered-wright} to $U_0=U$, $K_0=K$, and this $\Lambda$. For all large $n$, $q_n\le q_0$ and
$$\widehat\Phi_{z,q_n}(\tau)=\widehat\Phi_{z,n}(\tau),$$
so the labelled critical points in the scaled annulus are $\tau_{n,\nu}(z)=q_n^{-1}t_{n,\nu}(z)$. The scaled image of $\Gamma_{\varepsilon_n}$ is the positively oriented circle $\Gamma_\Lambda:=\{|\tau|=\Lambda\}$. Lemma~\ref{lem:wright-thimble-basis} gives
$$[\Gamma_\Lambda]=\sum_{\nu=0}^m\sigma_{\nu,V}[\Delta_\nu(z,q_n)]
\quad\text{in }H_1(X,B_{z,q_n};\mathbb Z),$$
where the integers are locally constant on the compact parameter family $\overline U\times[0,q_0]$. More explicitly, for fixed $U\Subset V$ the path $q\in[0,q_0]$ stays inside the family of Lemma~\ref{lem:wright-exit-arcs}: the boundary triad is locally trivial, critical points neither collide nor acquire equal imaginary values, and Lemma~\ref{lem:gradient-landing} excludes creation of saddle connections along the path. Therefore the relative homology local system over $\overline U\times[0,q_0]$ is trivialized by the transported thimbles, and the intersection number $[\Gamma_\Lambda]\cdot[\nabla_\nu(z,q)]$ is independent of $q$. Thus the coefficients used at $q_n=n^{-1/(m+1)}$ for all sufficiently large $n$ are the same intersection numbers as for the limiting Wright phase $\lambda_m(z)\tau^{-m}+\tau$ on $U$. The notation $\sigma_{\nu,V}$ records the continuation of these limiting coefficients in the chamber; no uniform assertion is used for $q$-dependent points approaching $\partial V$. Hence $I_V:=\{\nu:\sigma_{\nu,V}\ne0\}$ is nonempty.

Let
$$H_{z,n}:=\max_{\nu\in I_V}\Re\widehat\Phi_{z,n}(\tau_{n,\nu}(z))=n^{-m/(m+1)}M_V(z;n).$$
Lemma~\ref{lem:filtered-wright} gives a filtered chain $C_{z,q_n}$ and a boundary chain $E_{z,n}$ such that
$$\int_{\Gamma_\Lambda}\Omega_{z,n}=\int_{C_{z,q_n}}\Omega_{z,n}+\int_{E_{z,n}}\Omega_{z,n},$$
where
$$\Omega_{z,n}:=q_n^{\beta+1}\tau^\beta R(z,1-q_n\tau)
\exp\!\bigl(n^{m/(m+1)}\widehat\Phi_{z,n}(\tau)\bigr)\,\dd\tau.$$
This one-form is holomorphic on a neighborhood of $X$ for all large $n$: $\beta\in\mathbb Z$, $0\notin X$, $q_n\Lambda_1<1$ keeps the logarithmic singularity outside the annulus and fixes the Taylor branch, and $R(z,1-q_n\tau)$ is holomorphic there. Since $q_n\Lambda_1\to0$, the factor $R(z,1-q_n\tau)$ and all derivatives needed below are bounded uniformly for $z\in K$ and $\tau\in X$. Moreover,
$$\operatorname{supp}C_{z,q_n}\setminus\bigcup_{\nu\in I_V}\mathcal N_\nu(z,q_n)
\subset\{\Re\widehat\Phi_{z,n}\le H_{z,n}-c\},
\qquad
\operatorname{supp}E_{z,n}\subset\{\Re\widehat\Phi_{z,n}\le H_{z,n}-c\}$$
uniformly for $z\in K$ and all large $n$. Thus inactive thimbles are absent from the chain representative, and the boundary chain is included in the integral identity.

On the nonlocal pieces of $C_{z,q_n}$ and on all of $E_{z,n}$,
$$\Re\Phi_{z,n}\le M_V(z;n)-c n^{m/(m+1)}.$$
The length and multiplicity of these chains are uniformly bounded. The amplitude $q_n^{\beta+1}\tau^\beta R(z,1-q_n\tau)$ is bounded by a fixed power of $n$ on $X$. Decreasing $c$ absorbs this polynomial factor and gives the remainder estimate \eqref{eq:wright-multi-rem} for all nonlocal and boundary pieces. Since $n^{-m/(m+1)}\Xi_\nu(z;n)$ is uniformly bounded on $K$, $M_V(z;n)\ge -C_Kn^{m/(m+1)}$; the outer integral $O_U(\varrho^{-n})$ in \eqref{eq:wright-contour-split} is then also absorbed by \eqref{eq:wright-multi-rem}, after decreasing $c$ once more.

It remains to evaluate the active saddle neighborhoods. Fix $\nu\in I_V$, put $\alpha=m/(m+1)$, and write $\tau_s=\tau_{n,\nu}(z)$ and $t_s=q_n\tau_s$. The parametric holomorphic Morse lemma, applied on the compact critical graph over $K$, gives one radius $r>0$ and coordinates $v$ with uniformly bounded inverse maps and derivatives such that
$$\widehat\Phi_{z,n}(\tau)=\widehat\Phi_{z,n}(\tau_s)-v^2/2,
\qquad
\left.\frac{\partial\tau}{\partial v}\right|_{v=0}=\frac{i}{\sqrt{\widehat\Phi_{z,n}''(\tau_s)}},$$
and the oriented local downward thimble is the real $v$-axis. The square root is chosen so that
$$\widehat\Phi_{z,n}''(\tau_s)=\frac{m+1}{\tau_\nu(z)}\left(1+O_U\!\left(n^{-1/(m+1)}\right)\right),$$
which is equivalent to the convention
$$\sqrt{\Phi_{z,n}''(t_s)}=
\sqrt{\frac{m+1}{\tau_\nu(z)}}\,n^{(m+2)/(2(m+1))}
\left(1+O_U\!\left(n^{-1/(m+1)}\right)\right).$$
Define
$$a_{z,n}(v):=\tau(v)^\beta R(z,1-q_n\tau(v))\frac{\partial\tau}{\partial v}(v).$$
The functions $a_{z,n}$ have uniformly bounded first two derivatives for $|v|\le r$. The local contribution of the unit oriented thimble is
\begin{align*}
\frac{1}{2\pi i}\int q_n^{\beta+1}\tau^\beta R(z,1-q_n\tau)
 e^{n^\alpha\widehat\Phi_{z,n}(\tau)}\,\dd\tau
&=\frac{e^{\Xi_\nu(z;n)}q_n^{\beta+1}}{2\pi i}
\int_{-r}^r a_{z,n}(v)e^{-n^\alpha v^2/2}\,\dd v\\
&\quad +O_K\!\left(n^A e^{\Re\Xi_\nu(z;n)-c n^\alpha}\right)
\end{align*}
for a fixed $A$. The error includes the part of the thimble outside $|v|<r$ and the Gaussian tail, both of which lie a fixed amount below the saddle height in the scaled phase. Taylor expansion at $v=0$ gives
$$\int_{-r}^r a_{z,n}(v)e^{-n^\alpha v^2/2}\,\dd v
=a_{z,n}(0)\sqrt{2\pi}\,n^{-\alpha/2}
\left(1+O_K\!\left(n^{-1/(m+1)}\right)\right),$$
because the odd term integrates to zero and the next even term is $O(n^{-\alpha})$, which is $O(n^{-1/(m+1)})$ for $m\ge1$.

Since $a_{z,n}(0)=\tau_s^\beta R(z,1-t_s)i/\sqrt{\widehat\Phi_{z,n}''(\tau_s)}$ and
$$\Phi_{z,n}''(t_s)=n^\alpha q_n^{-2}\widehat\Phi_{z,n}''(\tau_s),$$
the preceding display becomes
$$\exp\!\bigl(\Xi_\nu(z;n)\bigr)
\frac{t_s^\beta R(z,1-t_s)}{\sqrt{2\pi\Phi_{z,n}''(t_s)}}
\left(1+O_K\!\left(n^{-1/(m+1)}\right)\right).$$
Finally,
$$t_s^\beta R(z,1-t_s)=q_n^\beta\tau_\nu(z)^\beta R(z,1)
\left(1+O_K\!\left(n^{-1/(m+1)}\right)\right),$$
and the square-root convention gives
$$\frac{1}{\sqrt{\Phi_{z,n}''(t_s)}}=
\frac{\tau_\nu(z)^{1/2}}{\sqrt{m+1}}\,n^{-(m+2)/(2(m+1))}
\left(1+O_K\!\left(n^{-1/(m+1)}\right)\right).$$
Thus
$$\frac{t_s^\beta R(z,1-t_s)}{\sqrt{2\pi\Phi_{z,n}''(t_s)}}
=n^\theta A_\nu(z)\left(1+O_K\!\left(n^{-1/(m+1)}\right)\right),$$
with the $A_\nu$ and $\theta$ of \eqref{eq:wright-multi-data}. Multiplying these unit-thimble contributions by the fixed integers $\sigma_{\nu,V}$, summing over $\nu\in I_V$, and adding the estimated remainder proves \eqref{eq:wright-multi}.
\end{proof}

\begin{prop}[active Wright coefficients]\label{prop:active-wright-coefficients}
Let $V$ be a leading Stokes chamber in the notation of this subsection, and fix $z_\ast\in V$. Put
\[
\psi_\ast(\tau):=\lambda_m(z_\ast)\tau^{-m}+\tau,
\qquad
\tau_{\nu,\ast}:=\omega_\nu\eta(z_\ast),\quad 0\le\nu\le m.
\]
Set $v_\nu:=\omega_\nu\eta(z_\ast)$ and $P_\ast:=\operatorname{conv}\{v_0,\ldots,v_m\}$. The limiting critical values of $\psi_\ast$ at the saddles $\tau_{\nu,\ast}$ are $\frac{m+1}{m}v_\nu$, so, up to this common factor, $P_\ast$ is the regular $(m+1)$-gon of limiting critical values for the Wright phase $\psi_\ast$. Since $z_\ast\notin\mathfrak S$, no two $v_\nu$ have the same imaginary part. Let $\mathcal R_V$ be the set of labels for the vertices on the boundary chain of $P_\ast$ visible from $+\infty$ in the real direction. Equivalently, with $\operatorname{Arg}\in(-\pi,\pi]$,
\begin{equation}\label{eq:active-right-boundary}
\mathcal R_V=
\left\{0\le\nu\le m:
\left|\operatorname{Arg} v_\nu\right|<
\frac{\pi}{2}+\frac{\pi}{m+1}\right\}.
\end{equation}
Equality in \eqref{eq:active-right-boundary} is equivalent to equality of the imaginary parts of two adjacent limiting critical values. These adjacent equalities are among the components of \(\mathfrak S\), so equality cannot occur for \(z_\ast\in V\). Hence \(\mathcal R_V\) is independent of the base point in \(V\). Non-adjacent components of \(\mathfrak S\) may further subdivide the parameter domain, but they do not change the visible-chain criterion. The coefficient-contour coordinates in Lemmas~\ref{lem:wright-thimble-basis} and~\ref{lem:wright-multi} satisfy
\[
|\sigma_{\nu,V}|=
\begin{cases}
1, & \nu\in\mathcal R_V,\\
0, & \nu\notin\mathcal R_V.
\end{cases}
\]
Thus
\[
I_V=\{\nu:\sigma_{\nu,V}\ne0\}=\mathcal R_V.
\]
The signed value for $\nu\in\mathcal R_V$ is the oriented crossing sign of the corresponding upward thimble with the positively oriented core circle, with the Morse orientation convention of Lemma~\ref{lem:wright-multi}. These signs are independent of the admissible annulus, core representative, and base point in $V$, after continuation of the saddle labels and Morse orientations.

For an essential site $a_i$, the same rule is applied with
\[
v_\nu=\omega_\nu\eta_D(z_\ast),
\qquad
\eta_D(z_\ast)^{m_i+1}=m_i\lambda_{i,m_i}(z_\ast-a_i)^{-m_i}.
\]
Changing the branch of $\eta_D$ only relabels the output.
\end{prop}

Figure~\ref{fig:wright-active-visibility} illustrates the visible-chain criterion and its angular form for $m=5$.

\begin{figure}[t]
\centering
\begin{tikzpicture}[
    x=1cm,
    y=1cm,
    font=\small,
    >=stealth,
    line cap=round,
    line join=round
]

% =========================================================
% Left panel: visible boundary chain
% =========================================================
\begin{scope}[shift={(0,0)}]

\node[font=\normalsize] at (0,3.35)
    {Visible boundary chain in $P_\ast$};

% axes
\draw[->,thin,gray!65] (-3.15,0) -- (3.48,0) node[below right] {$\Re$};
\draw[->,thin,gray!65] (0,-2.42) -- (0,2.85) node[above left] {$\Im$};

% hexagon vertices for m=5
\coordinate (v0) at ( 2.55, 0.00);
\coordinate (v1) at ( 1.275, 2.208);
\coordinate (v2) at (-1.275, 2.208);
\coordinate (v3) at (-2.55, 0.00);
\coordinate (v4) at (-1.275,-2.208);
\coordinate (v5) at ( 1.275,-2.208);

% polygon
\fill[gray!8] (v0)--(v1)--(v2)--(v3)--(v4)--(v5)--cycle;
\draw[thick,gray!60] (v0)--(v1)--(v2)--(v3)--(v4)--(v5)--cycle;

% active visible chain
\draw[very thick,blue!70!black] (v5)--(v0)--(v1);

% vertices
\fill[blue!70!black] (v5) circle (2.3pt);
\fill[blue!70!black] (v0) circle (2.3pt);
\fill[blue!70!black] (v1) circle (2.3pt);

\filldraw[fill=white,draw=gray!70] (v2) circle (2.1pt);
\filldraw[fill=white,draw=gray!70] (v3) circle (2.1pt);
\filldraw[fill=white,draw=gray!70] (v4) circle (2.1pt);

% labels
\node[anchor=north west] at ( 2.66,-0.08) {$v_0$};
\node[anchor=south west] at ( 1.39, 2.28) {$v_1$};
\node[anchor=south east] at (-1.39, 2.28) {$v_2$};
\node[anchor=east]       at (-2.70, 0.16) {$v_3$};
\node[anchor=north east] at (-1.39,-2.28) {$v_4$};
\node[anchor=north west] at ( 1.39,-2.28) {$v_5$};

\node at (-0.15,0.50) {$P_\ast$};

% viewing arrow from +infty: points toward the visible boundary chain, not a vertex
\draw[->,thick,orange!85!black] (3.92,0.86) -- (2.56,0.52);
\node[orange!85!black,anchor=west] at (4.00,0.96) {$+\infty$};

% note below panel, with more separation from the axis
\node[align=center,text width=5.8cm] at (0,-3.35)
    {Active vertices are precisely the boundary chain visible\\
     from the positive real direction.};

\end{scope}

% =========================================================
% Right panel: angular criterion
% =========================================================
\begin{scope}[shift={(9.05,0)}]

\node[font=\normalsize] at (0,3.35)
    {Equivalent angular criterion};

% axes
\draw[->,thin,gray!65] (-2.90,0) -- (3.00,0) node[below right] {$\Re$};
\draw[->,thin,gray!65] (0,-2.42) -- (0,2.80) node[above left] {$\Im$};

% active sector
\fill[blue!10]
    (0,0) -- (120:2.12)
    arc[start angle=120,end angle=-120,radius=2.12]
    -- cycle;

% guide circle
\draw[thick,gray!60] (0,0) circle (2.12);

% boundary rays
\draw[dashed,gray!80] (0,0) -- (120:2.48);
\draw[dashed,gray!80] (0,0) -- (-120:2.48);

% points
\coordinate (u0) at (   0:2.12);
\coordinate (u1) at (  60:2.12);
\coordinate (u2) at ( 120:2.12);
\coordinate (u3) at ( 180:2.12);
\coordinate (u4) at (-120:2.12);
\coordinate (u5) at ( -60:2.12);

\fill[blue!70!black] (u5) circle (2.3pt);
\fill[blue!70!black] (u0) circle (2.3pt);
\fill[blue!70!black] (u1) circle (2.3pt);

\filldraw[fill=white,draw=gray!70] (u2) circle (2.1pt);
\filldraw[fill=white,draw=gray!70] (u3) circle (2.1pt);
\filldraw[fill=white,draw=gray!70] (u4) circle (2.1pt);

% point labels
\node[anchor=north west] at ( 2.22,-0.08) {$v_0$};
\node[anchor=west]       at ( 1.10, 2.18) {$v_1$};
\node[anchor=east]       at (-1.12, 2.18) {$v_2$};
\node[anchor=east]       at (-2.30, 0.16) {$v_3$};
\node[anchor=east]       at (-1.12,-2.18) {$v_4$};
\node[anchor=west]       at ( 1.10,-2.18) {$v_5$};

% boundary-ray labels: slightly closer to v_2 and v_4
\node[anchor=south east] at (-1.28,2.38)
    {$\operatorname{Arg} v=\frac{2\pi}{3}$};

\node[anchor=north east] at (-1.28,-2.38)
    {$\operatorname{Arg} v=-\frac{2\pi}{3}$};

% criterion text: shifted farther right and slightly lower
\node[align=center,text width=5.3cm] at (0.82,-3.44)
    {Active iff
    $\bigl|\operatorname{Arg} v\bigr|
      < \dfrac{\pi}{2}+\dfrac{\pi}{m+1}
      = \dfrac{2\pi}{3}$
    \quad $(m=5)$.};

\end{scope}

\end{tikzpicture}

\caption{The active-saddle visibility rule of Proposition~\ref{prop:active-wright-coefficients}, shown for $m=5$ in a schematic boundary orientation. Up to the common factor $(m+1)/m$, the limiting Wright critical values form the regular polygon $P_\ast=\operatorname{conv}\{v_0,\ldots,v_5\}$. The active vertices are the boundary chain visible from $+\infty$ in the real direction, equivalently those satisfying $\bigl|\operatorname{Arg} v_\nu\bigr|<\pi/2+\pi/(m+1)$. A generic chamber picture is obtained by a small rotation; in the displayed cutoff orientation the boundary vertices $v_2$ and $v_4$ are not active because the inequality is strict. Filled vertices are active and open vertices are inactive. The dashed rays in the angular picture mark the boundary Stokes directions.}
\label{fig:wright-active-visibility}
\end{figure}

\begin{proof}
The homological definition of the signed coefficients is the duality statement in Lemma~\ref{lem:wright-thimble-basis}. The same lemma, together with Lemma~\ref{lem:wright-exit-arcs}, gives local constancy on compact subchambers and identifies, for sufficiently small $q$, the finite-$q$ thimbles with the limiting $q=0$ thimbles by continuation in the locally trivial boundary triads. Hence it remains to identify which limiting upward thimbles meet the exterior side of the coefficient core.

Scale $\tau=\eta(z_\ast)u$. Since $\eta(z_\ast)^{m+1}=m\lambda_m(z_\ast)$,
\[
\psi_\ast(\eta(z_\ast)u)=\eta(z_\ast)F(u),
\qquad
F(u):=u+\frac1m u^{-m}.
\]
The critical points of $F$ are $\omega_0,\ldots,\omega_m$, and their critical values are $((m+1)/m)\omega_\nu$. On the exterior domain $E:=\{|u|>1\}$ the map $F$ is univalent: if $F(u)=F(v)$ with $u,v\in E$ and $u\ne v$, then
\[
1=\frac1m\sum_{r=0}^{m-1}u^{-r-1}v^{-(m-r)},
\]
whose right-hand side has modulus strictly smaller than $1$. Thus $F$ is the conformal exterior branch normalized by $F(u)=u+O(u^{-m})$ at infinity; the boundary curve $F(|u|=1)$ is the scaled $(m+1)$-cusped hypocycloid, with cusps at the critical values.

Let $c=(m+1)/m$ and, for a direction $\alpha$, put
\[
L_\alpha:=c+[0,\infty)e^{i\alpha}.
\]
Assume first that $L_\alpha$ contains no critical value other than its initial point $c$. Along any component of the complement of the finitely many directions from $c$ to the other critical values, the exterior inverse branch of $F$ along $L_\alpha$ varies by analytic continuation and cannot be created or destroyed in $E$: away from the cusps the lifted ray is a smooth one-manifold, and near infinity the normalization $F(u)=u+O(u^{-m})$ gives the unique unbounded exterior end. For $\alpha=0$ the lift is the real branch $u>1$, and it crosses every sufficiently large positively oriented core circle once. The component containing $0$ is bounded by the directions from $c$ to the two adjacent critical values, namely
\[
\arg\bigl(e^{\pm2\pi i/(m+1)}-1\bigr)
=\pm\left(\frac\pi2+\frac\pi{m+1}\right)\pmod{2\pi};
\]
the directions to non-adjacent critical values lie outside this open interval in $(-\pi,\pi]$. Hence the exterior branch lands at the cusp $u=1$ exactly when
\[
|\alpha|<\frac\pi2+\frac\pi{m+1}.
\]
Equality is the adjacent Stokes event where the ray from one critical value passes through an adjacent one.

For the saddle labelled by $\nu$, write $u=\omega_\nu w$. Since $F(\omega_\nu w)=\omega_\nu F(w)$, the phase is $v_\nu F(w)$ with $v_\nu=\omega_\nu\eta(z_\ast)$. The upward flow has constant imaginary part and increasing real part in the phase plane; after division by $v_\nu$, its two local branches are the lifts of $L_{-\operatorname{Arg}v_\nu}$. Therefore an exterior branch occurs exactly under the inequality \eqref{eq:active-right-boundary}. This is precisely the condition that the vertex $v_\nu$ of the regular polygon $P_\ast$ lie on the boundary chain visible from $+\infty$ in the real direction.

If one branch is exterior, the upward thimble has one outer and one inner end in any admissible annulus. A transverse positive core representative meets the exterior branch once and misses the inner branch after a collar perturbation near $A$, so the absolute homological intersection is one. If no branch is exterior, both ends are inner; the upward thimble is then homologous in $H_1(X,A;\mathbb Z)$ to a chain supported on the inner side of the annulus and has zero intersection with the core class. The sign in the active case is the oriented transverse intersection sign fixed by the Morse orientation convention. Non-adjacent components of $\mathfrak S$ only subdivide the parameter domain on which the signed thimble basis is continued; they do not change the visibility criterion. The essential-site and branch-relabeling assertions follow by substituting the displayed local branch of $\eta_D$.
\end{proof}

\subsection{Main result on open Voronoi cells}

\begin{theorem}[uniform coefficient extraction on Voronoi cells]\label{prop:local-coef-unified}
Fix $i$ and let $D$ be an open set with \(\overline D\Subset\mathcal V_i^\circ\setminus\Sigma\). Set
$$d_i(z):=a_i-z,\qquad w:=z+d_i(z)\zeta,\qquad
G_{i,z}(\zeta):=\mathcal C\bigl(z,d_i(z)\zeta\bigr)=\sum_{n\ge0}\frac{C_n(z)}{n!}\,d_i(z)^n\zeta^n.$$
Then $\zeta=1$ is the unique singularity of $G_{i,z}$ on $|\zeta|=1$. Consequently the coefficient asymptotics split into exactly two cases: an ordinary pole site $m_i=0$, where the pole at $\zeta=1$ gives elementary coefficient asymptotics, and an essential singularity site $m_i\ge1$, where a sectorial Wright expansion with contour-determined saddle multipliers applies. In both cases,
\begin{equation}\label{eq:Gi-local-unified}
G_{i,z}(\zeta)
=
(1-\zeta)^{\beta_i}
\exp\!\Bigl(\sum_{s=1}^{m_i}\widetilde\lambda_{i,s}(z)(1-\zeta)^{-s}\Bigr)
R_i(z,\zeta),
\end{equation}
where
$$\widetilde\lambda_{i,s}(z):=\lambda_{i,s}(z-a_i)^{-s}\qquad(1\le s\le m_i),$$
and
\begin{equation}\label{eq:Ri-local-unified}
R_i(z,\zeta):=
(-d_i(z))^{\beta_i}\frac{Q(z)}{P_T(z)}
\frac{\widetilde P_i(w)\,P_\sharp(w)}{\widetilde Q_i(w)}
\exp\!\bigl(E_i^{\mathrm{reg}}(w)-E(z)\bigr).
\end{equation}
The function $R_i$ is holomorphic on a neighborhood of $D\times\{1\}$ and satisfies
\begin{equation}\label{eq:Ri-local-unified-at1}
R_i(z,1)=
\frac{\widetilde Q_i(z)\,\widetilde P_i(a_i)\,P_\sharp(a_i)}{\widetilde P_i(z)\,\widetilde Q_i(a_i)}
e^{E_i^{\mathrm{reg}}(a_i)-E(z)}\ne0.
\end{equation}
Let $K\Subset D$ be compact.
\begin{enumerate}
\item If $m_i=0$, then $\beta_i=-r_i\le-1$, and uniformly for $z\in K$,
\begin{align}
\frac{C_n(z)}{n!}
&=d_i(z)^{-n}n^{-\beta_i-1}\frac{R_i(z,1)}{\Gamma(-\beta_i)}\bigl(1+O_K(n^{-1})\bigr),\label{eq:Cn-local-unified-pole}\\
\frac{B_n(z)}{n!}
&=W(z)^n d_i(z)^{-n}n^{-\beta_i-1}\frac{R_i(z,1)}{\Gamma(-\beta_i)}\bigl(1+O_K(n^{-1})\bigr).\label{eq:Bn-local-unified-pole}
\end{align}
\item Suppose $m_i\ge1$ and $D$ is simply connected. Choose a holomorphic branch $\eta_D$ on $D$ with
$$\eta_D(z)^{m_i+1}=m_i\lambda_{i,m_i}(z-a_i)^{-m_i},$$
choose a holomorphic square root $\eta_D^{1/2}$ on $D$, let $\omega_\nu:=e^{2\pi i\nu/(m_i+1)}$, $0\le\nu\le m_i$, and fix square roots $\omega_\nu^{1/2}$ once and for all. Set
$$\Phi_{i,z,n}(t):=\sum_{s=1}^{m_i}\widetilde\lambda_{i,s}(z)t^{-s}-(n+1)\log(1-t),\qquad
\theta_i:=-\frac{2\beta_i+m_i+2}{2(m_i+1)},$$
\begin{equation}\label{eq:local-stokes-sets}
\mathfrak S_{i,D}:=\bigcup_{0\le\nu<\mu\le m_i}
\{z\in D:\Im((\omega_\nu-\omega_\mu)\eta_D(z))=0\},
\end{equation}
and
\begin{equation}\label{eq:local-antistokes-sets}
\mathfrak A_{i,D}:=\bigcup_{0\le\nu<\mu\le m_i}
\{z\in D:\Re((\omega_\nu-\omega_\mu)\eta_D(z))=0\}.
\end{equation}
For every compact $K\Subset D$ there is $C_K$ such that
\begin{equation}\label{eq:essential-normalized-upper}
\sup_{z\in K}\log^+\left|
\frac{B_n(z)}{n!}\left(\frac{d_i(z)}{W(z)}\right)^n n^{-\theta_i}
\right|\le C_K n^{m_i/(m_i+1)}
\end{equation}
for all $n\ge2$.

Let $V$ be a connected component of $D\setminus\mathfrak S_{i,D}$, and let $K\Subset V$ be compact. Then there exists an open set $U$ with $K\Subset U\Subset V$ such that, for each $0\le\nu\le m_i$ and all sufficiently large $n$, there is a holomorphic function $t_{i,n,\nu}:U\to\C$ satisfying
$$\Phi_{i,z,n}'(t_{i,n,\nu}(z))=0,\qquad
t_{i,n,\nu}(z)=\omega_\nu\eta_D(z)\,n^{-1/(m_i+1)}+O_U\!\left(n^{-2/(m_i+1)}\right).$$
Define
\begin{equation}\label{eq:Ai-local-unified}
\mathcal A_{i,\nu}(z):=
\frac{(\omega_\nu\eta_D(z))^{\beta_i}\,(\omega_\nu^{1/2}\eta_D^{1/2}(z))}{\sqrt{2\pi(m_i+1)}}\,R_i(z,1),
\end{equation}
and
\begin{equation}\label{eq:Xi-local-unified}
\Xi_{i,\nu}(z;n):=\Phi_{i,z,n}(t_{i,n,\nu}(z)).
\end{equation}
Then each $\mathcal A_{i,\nu}$ is holomorphic and nonvanishing on $U$, each $\Xi_{i,\nu}(z;n)$ admits an asymptotic expansion to arbitrary finite order in descending powers of $n^{1/(m_i+1)}$, uniformly on $K$, and
\begin{equation}\label{eq:Xi-local-unified-leading}
\Xi_{i,\nu}(z;n)=\frac{m_i+1}{m_i}\,\omega_\nu\eta_D(z)\,n^{m_i/(m_i+1)}
+O_K\!\left(n^{(m_i-1)/(m_i+1)}\right)
\end{equation}
uniformly for $z\in K$.

Orient the local downward thimbles compatibly with the square-root choices in \eqref{eq:Ai-local-unified}. Lemma~\ref{lem:filtered-wright}, applied as in Lemma~\ref{lem:wright-multi}, gives integers $\sigma_{i,\nu,V}$, equal to the intersection numbers of the scaled coefficient-contour class with the dual upward thimbles. Proposition~\ref{prop:active-wright-coefficients} gives their vanishing pattern: for any, equivalently every, $z_\ast\in V$, set $v_\nu=\omega_\nu\eta_D(z_\ast)$. Then
\[
I_{i,V}:=\{\nu:\sigma_{i,\nu,V}\ne0\}
=\left\{\nu:
\left|\operatorname{Arg} v_\nu\right|<
\frac\pi2+\frac\pi{m_i+1}\right\},
\]
equivalently the labels on the boundary chain of $\operatorname{conv}\{v_0,
\ldots,v_{m_i}\}$ visible from $+\infty$ in the real direction. Moreover $|\sigma_{i,\nu,V}|=1$ on $I_{i,V}$ and $\sigma_{i,\nu,V}=0$ off it; the signs are the homological intersection signs with the transverse core representative specified in Proposition~\ref{prop:active-wright-coefficients}. If
$$M_{i,V}(z;n):=\max_{\nu\in I_{i,V}}\Re\Xi_{i,\nu}(z;n),$$
then there are functions $\varepsilon_{i,n,\nu}:K\to\C$, $\nu\in I_{i,V}$, satisfying
$$\max_{\nu\in I_{i,V}}\sup_{z\in K}|\varepsilon_{i,n,\nu}(z)|=O_K\!\left(n^{-1/(m_i+1)}\right),$$
and, uniformly for $z\in K$,
\begin{align}
\frac{C_n(z)}{n!}
&=d_i(z)^{-n}n^{\theta_i}
\left(
\sum_{\nu\in I_{i,V}}\sigma_{i,\nu,V}\mathcal A_{i,\nu}(z)\exp\!\bigl(\Xi_{i,\nu}(z;n)\bigr)\bigl(1+\varepsilon_{i,n,\nu}(z)\bigr)+\mathcal E_{i,n,V}(z)
\right),\label{eq:Cn-local-unified-essential}\\
\frac{B_n(z)}{n!}
&=W(z)^n d_i(z)^{-n}n^{\theta_i}
\left(
\sum_{\nu\in I_{i,V}}\sigma_{i,\nu,V}\mathcal A_{i,\nu}(z)\exp\!\bigl(\Xi_{i,\nu}(z;n)\bigr)\bigl(1+\varepsilon_{i,n,\nu}(z)\bigr)+\mathcal E_{i,n,V}(z)
\right),\label{eq:Bn-local-unified-essential}
\end{align}
where
\begin{equation}\label{eq:R-local-unified-essential}
\mathcal E_{i,n,V}(z)=
O_K\!\left(
\exp\!\bigl(M_{i,V}(z;n)-c_K n^{m_i/(m_i+1)}\bigr)
\right)
\end{equation}
for some $c_K>0$. The expansion is uniform on compact subsets of a Stokes chamber $V$, including compact subsets that meet the leading anti-Stokes set $\mathfrak A_{i,D}\cap V$. Changing one square-root branch changes the displayed $\mathcal A_{i,\nu}$ by a sign; changing the corresponding thimble orientation changes $\sigma_{i,\nu,V}$ by the same sign, so the products $\sigma_{i,\nu,V}\mathcal A_{i,\nu}$ are intrinsic.
\end{enumerate}
\end{theorem}

\begin{proof}
For each $z\in D$, the possible singularities of $G_{i,z}$ occur at
$$\zeta=\frac{a_j-z}{a_i-z},\qquad 1\le j\le N.$$
They are actual singularities: at a site in $\Z(Q)\setminus\Z(T)$ the pole of $Q^{-1}$ is not cancelled by $P$, and at a site in $\Z(T)$ the exponential factor has an essential singularity which a rational prefactor cannot remove. Thus these are precisely the singularities of $G_{i,z}$. Since $z\in\mathcal V_i^\circ$, one has
$$\left|\frac{a_j-z}{a_i-z}\right|>1\qquad(j\ne i),$$
so $\zeta=1$ is the unique singularity on $|\zeta|=1$.

Using $P=P_TP_\sharp$, the local factorizations \eqref{eq:local-factors-unified}, and
$$w-a_i=-d_i(z)(1-\zeta),$$
we obtain from \eqref{eq:Cexplicit}
\begin{align*}
G_{i,z}(\zeta)
&=(-d_i(z))^{\beta_i}(1-\zeta)^{\beta_i}
\frac{Q(z)}{P_T(z)}
\frac{\widetilde P_i(w)\,P_\sharp(w)}{\widetilde Q_i(w)}
\exp\!\bigl(E(w)-E(z)\bigr)\\
&=(1-\zeta)^{\beta_i}
\exp\!\Bigl(\sum_{s=1}^{m_i}\widetilde\lambda_{i,s}(z)(1-\zeta)^{-s}\Bigr)
R_i(z,\zeta),
\end{align*}
which is \eqref{eq:Gi-local-unified}--\eqref{eq:Ri-local-unified}. Evaluating at $\zeta=1$ and using
$$(-d_i(z))^{\beta_i}\frac{Q(z)}{P_T(z)}=\frac{\widetilde Q_i(z)}{\widetilde P_i(z)}$$
gives \eqref{eq:Ri-local-unified-at1}.

Let
$$\varrho_K:=\inf_{z\in K}\min_{j\ne i}\left|\frac{a_j-z}{a_i-z}\right|\in(1,\infty],$$
where the empty minimum is interpreted as $+\infty$. Fix $\varrho$ with $1<\varrho<\varrho_K$ (or any $\varrho>1$ if $\varrho_K=\infty$). Then, for every $z\in K$, all singularities of $G_{i,z}$ other than $\zeta=1$ lie outside $|\zeta|\le\varrho$. Since \eqref{eq:Cexplicit} is jointly holomorphic in $(z,\zeta)$ away from the moving singular set, it follows that $(z,\zeta)\mapsto G_{i,z}(\zeta)$ is holomorphic on a neighborhood of $K\times\bigl(\{|\zeta|\le\varrho\}\setminus\{1\}\bigr)$. The same argument applies to every compact $L\Subset D$, so the annulus hypothesis required in Lemmas~\ref{lem:wright-upper} and~\ref{lem:wright-multi} holds on $D$.

Assume first that $m_i=0$. Then $r_i=-\beta_i\in\mathbb N$ and
$$G_{i,z}(\zeta)=(1-\zeta)^{-r_i}R_i(z,\zeta).$$
Expand $R_i(z,\zeta)$ at $\zeta=1$ to order $r_i-1$:
$$R_i(z,\zeta)=\sum_{u=0}^{r_i-1}c_{i,u}(z)(1-\zeta)^u+(1-\zeta)^{r_i}\widehat H_i(z,\zeta),$$
where the coefficients $c_{i,u}$ are holomorphic on a neighborhood of $K$, $c_{i,0}(z)=R_i(z,1)$, and $\widehat H_i$ is holomorphic near $K\times\{1\}$. Therefore
$$G_{i,z}(\zeta)=\sum_{u=0}^{r_i-1}c_{i,u}(z)(1-\zeta)^{-r_i+u}+\widehat H_i(z,\zeta),$$
and the right-hand side has no singularity at $\zeta=1$ except in the displayed principal part. Hence $\widehat H_i$ extends holomorphically to a neighborhood of $K\times\{|\zeta|\le1+\delta\}$ for every fixed $0<\delta<\varrho_K-1$. For integers $s\ge1$,
$$[\zeta^n](1-\zeta)^{-s}=\binom{n+s-1}{s-1}
=\frac{n^{s-1}}{\Gamma(s)}\bigl(1+O(n^{-1})\bigr),$$
while Cauchy's estimate gives
$$[\zeta^n]\widehat H_i(z,\zeta)=O_K\!\bigl((1+\delta/2)^{-n}\bigr).$$
Since $R_i(\cdot,1)$ is holomorphic and nonvanishing, it is bounded away from zero on $K$; hence the $u=0$ term dominates uniformly on $K$. This yields \eqref{eq:Cn-local-unified-pole}, and \eqref{eq:Bn-local-unified-pole} follows by multiplying by $W(z)^n$.

Assume now that $m_i\ge1$ and $D$ is simply connected. Then $\widetilde\lambda_{i,m_i}(z)=\lambda_{i,m_i}(z-a_i)^{-m_i}$ is holomorphic and nonvanishing on $D$, so the required branches $\eta_D$ and $\eta_D^{1/2}$ exist. Lemma~\ref{lem:wright-upper}, applied to $F_z=G_{i,z}$ with $m=m_i$, $\beta=\beta_i$, $\lambda_s=\widetilde\lambda_{i,s}$, and $R=R_i$, gives \eqref{eq:essential-normalized-upper}. If $V$ is a component of $D\setminus\mathfrak S_{i,D}$ and $K\Subset V$, Lemma~\ref{lem:wright-multi} applied with the same data gives \eqref{eq:Ai-local-unified}--\eqref{eq:R-local-unified-essential}; the identities for $C_n/n!$ and $B_n/n!$ follow from
$$[\zeta^n]G_{i,z}(\zeta)=\frac{C_n(z)}{n!}d_i(z)^n,
\qquad B_n(z)=W(z)^nC_n(z).$$
\end{proof}

\begin{cor}[one-saddle consequence in an essential cell]\label{cor:essential-one-saddle}
Assume the essential case of Theorem~\ref{prop:local-coef-unified}\textup{(2)}, let $V$ be a connected component of $D\setminus\mathfrak S_{i,D}$, and let $I_{i,V}$ be the corresponding active set. If $K\Subset V$ is compact and there are an index $\nu_\ast\in I_{i,V}$ and a number $\delta>0$ such that
$$\Re\bigl((\omega_{\nu_\ast}-\omega_\nu)\eta_D(z)\bigr)\ge\delta
\qquad(z\in K,\ \nu\in I_{i,V}\setminus\{\nu_\ast\}),$$
then, uniformly for $z\in K$,
\begin{equation}\label{eq:Bn-essential-one-saddle}
\frac{B_n(z)}{n!}=W(z)^n d_i(z)^{-n}n^{\theta_i}\sigma_{i,\nu_\ast,V}\mathcal A_{i,\nu_\ast}(z)
\exp\!\bigl(\Xi_{i,\nu_\ast}(z;n)\bigr)
\left(1+O_K\!\left(n^{-1/(m_i+1)}\right)\right).
\end{equation}
In particular $B_n$ has no zeros on $K$ for all sufficiently large $n$.
\end{cor}

\begin{proof}
By Theorem~\ref{prop:local-coef-unified}\textup{(2)}, the expansion \eqref{eq:Bn-local-unified-essential} holds uniformly on $K$. The hypothesis and \eqref{eq:Xi-local-unified-leading} give, for every $\nu\in I_{i,V}\setminus\{\nu_\ast\}$ and all sufficiently large $n$,
$$
\Re\bigl(\Xi_{i,\nu_\ast}(z;n)-\Xi_{i,\nu}(z;n)\bigr)
\ge \frac{m_i+1}{2m_i}\delta n^{m_i/(m_i+1)}
$$
uniformly on $K$. Hence $M_{i,V}(z;n)=\Re\Xi_{i,\nu_\ast}(z;n)$ on $K$ for all large $n$, and the remainder \eqref{eq:R-local-unified-essential} is exponentially smaller than the $\nu_\ast$ term. All other active saddle terms have the same exponential smallness. Since $\sigma_{i,\nu_\ast,V}\mathcal A_{i,\nu_\ast}$ is nonvanishing on $K$, the only remaining relative error is the local saddle error $O_K(n^{-1/(m_i+1)})$, proving \eqref{eq:Bn-essential-one-saddle}. The same formula implies zero-freeness on $K$ for all sufficiently large $n$.
\end{proof}

For the next corollary, with $\gamma_n=\lc(B_n)$ as in Proposition~\ref{prop:deg}, set
$$s_n:=
\begin{cases}
0, & \text{if } h=0,\\
\log n!, & \text{if } h>0,
\end{cases}$$
and, whenever $\deg B_n>0$,
$$\widetilde L_n(z):=\frac{1}{\deg B_n}\Bigl(\log|B_n(z)|-\log|\gamma_n|-s_n\Bigr),\qquad
\Psi_i(z):=\frac{1}{\kappa}\bigl(\log|W(z)|-\log|z-a_i|-\sigma\bigr).$$

\begin{lem}[Jensen--Poisson $L^1$ criterion]\label{lem:jensen-poisson-l1}
Let $D\subset\C$ be a domain, let $\alpha>0$, and let $F_n$ be holomorphic and not identically zero on $D$ for all sufficiently large $n$. Assume that for every compact $L\Subset D$ there is $C_L$ such that
$$\sup_L\log|F_n|\le C_L n^\alpha$$
for all sufficiently large $n$, and that there is a dense set $\mathcal G\subset D$ with the following property: for each $a\in\mathcal G$ there are $C_a$ and $n_a$ such that $F_n(a)\ne0$ and
$$|\log|F_n(a)||\le C_a n^\alpha\qquad(n\ge n_a).$$
Then, for every compact $K\Subset D$,
$$\|\log|F_n|\|_{L^1(K)}=O_K(n^\alpha).$$
\end{lem}

\begin{proof}
Discard finitely many initial indices, so that every $F_n$ under consideration is not identically zero. For each $x\in K$, choose $a_x\in\mathcal G$ so close to $x$ that the interval
$$\left(|x-a_x|,\frac14\operatorname{dist}(a_x,\partial D)\right)$$
is nonempty. The set of radii $r$ in this interval for which either $\partial D(a_x,2r)$ or $\partial D(a_x,4r)$ contains a zero of some $F_n$ is countable. Choose
$$r_x\in\left(|x-a_x|,\frac14\operatorname{dist}(a_x,\partial D)\right)$$
outside this set. Compactness gives finitely many disks $D(a_\ell,r_\ell)$ with
$$K\subset\bigcup_{\ell=1}^L D(a_\ell,r_\ell),\qquad
\overline{D(a_\ell,4r_\ell)}\Subset D,$$
and with neither $\partial D(a_\ell,2r_\ell)$ nor $\partial D(a_\ell,4r_\ell)$ containing a zero of any $F_n$.
After this finite subcover is fixed, discard further finitely many indices and choose constants $C_\ell^0$ so that, for each $\ell$,
$$
F_n(a_\ell)\ne0,
\qquad
|\log|F_n(a_\ell)||\le C_\ell^0 n^\alpha
$$
for all remaining $n$. This is allowed because $a_\ell\in\mathcal G$ for every $\ell$.

Fix $\ell$. Jensen's formula on $D(a_\ell,4r_\ell)$ gives
$$\sum_{|\zeta-a_\ell|<4r_\ell}\operatorname{mult}_\zeta(F_n)
\log\frac{4r_\ell}{|\zeta-a_\ell|}
=\frac1{2\pi}\int_0^{2\pi}\log|F_n(a_\ell+4r_\ell e^{it})|\,dt-\log|F_n(a_\ell)|.$$
The hypotheses bound the right-hand side by $O(n^\alpha)$. Every zero in $D(a_\ell,2r_\ell)$ contributes at least $\log2$, so the number of zeros in $D(a_\ell,2r_\ell)$, counted with multiplicity, is $O(n^\alpha)$.

Apply Poisson--Jensen in $D(a_\ell,2r_\ell)$:
$$\log|F_n|=h_{\ell,n}+G_{\ell,n},$$
where $h_{\ell,n}$ is harmonic and $G_{\ell,n}\le0$ is the Green potential of the zeros in the disk. Jensen's formula on $D(a_\ell,2r_\ell)$ and the pointwise bound at $a_\ell$ give $-G_{\ell,n}(a_\ell)=O(n^\alpha)$, hence $|h_{\ell,n}(a_\ell)|=O(n^\alpha)$. On the boundary, $h_{\ell,n}=\log|F_n|\le C_\ell n^\alpha$. Choose $C'_\ell>C_\ell$. The harmonic function $u_{\ell,n}:=C'_\ell n^\alpha-h_{\ell,n}$ is positive on $D(a_\ell,2r_\ell)$ and satisfies $u_{\ell,n}(a_\ell)=O(n^\alpha)$. Harnack's inequality gives $u_{\ell,n}=O(n^\alpha)$ on $D(a_\ell,r_\ell)$, while the maximum principle gives $h_{\ell,n}\le C'_\ell n^\alpha$. Hence
$$\|h_{\ell,n}\|_{L^\infty(D(a_\ell,r_\ell))}=O(n^\alpha).$$
The Green kernel of $D(a_\ell,2r_\ell)$ has $L^1(D(a_\ell,r_\ell))$-norm bounded uniformly in its pole, and there are $O(n^\alpha)$ poles. Hence
$$\|\log|F_n|\|_{L^1(D(a_\ell,r_\ell))}=O(n^\alpha).$$
Summing over $\ell$ proves the lemma.
\end{proof}

\begin{lem}[upper-envelope \(L^1\) criterion]\label{lem:upper-envelope-l1}
Let $D\subset\C$ be a domain, let $M$ be a continuous subharmonic function on $D$, and let $F_n\in\mathcal O(D)$ be not identically zero for all sufficiently large $n$. Put $u_n:=n^{-1}\log|F_n|$. Assume that
\[
\limsup_{n\to\infty}\sup_K(u_n-M)\le0
\]
for every $K\Subset D$, and that there is an open dense subset $G\subset D$ of full planar measure such that $u_n\to M$ uniformly on every compact subset of $G$. Then
\[
u_n\longrightarrow M\qquad\text{in }L^1_{\mathrm{loc}}(D).
\]
\end{lem}

\begin{proof}
The upper bound gives local upper bounds for the subharmonic functions $u_n$. By the compactness theorem for subharmonic functions, every subsequence has a further subsequence which either tends to $-\infty$ locally uniformly or converges in $L^1_{\mathrm{loc}}(D)$ to a subharmonic function $u$. The first alternative is impossible because $G$ is nonempty on every component and $u_n\to M$ uniformly on compact subsets of $G$. For an $L^1_{\mathrm{loc}}$ limit, the upper bound gives $u\le M$. A further subsequence converges pointwise almost everywhere; on $G$ the pointwise limit is $M$, and $G$ has full measure. Hence $u=M$ almost everywhere, and therefore as a subharmonic function. Thus every subsequence has a further subsequence converging to $M$, which proves the claim.
\end{proof}

\begin{cor}[local $L^1$-rate on essential Voronoi interiors]\label{cor:essential-l1}
Assume the notation of Theorem~\ref{prop:local-coef-unified}\textup{(2)}, and let $K\Subset D$ be compact. Then
\begin{equation}\label{eq:essential-L1-rate}
\bigl\|\widetilde L_n-\Psi_i\bigr\|_{L^1(K)}
=O_K\!\left(n^{-1/(m_i+1)}\right).
\end{equation}
Here \(L^1(K)\) is taken with respect to planar Lebesgue area. The same estimate holds on any compact subset of $\mathcal V_i^\circ\setminus\Sigma$ after a finite simply connected cover.
\end{cor}

\begin{proof}
Set
$$\alpha:=\frac{m_i}{m_i+1},\qquad
F_n(z):=\frac{B_n(z)}{n!}\left(\frac{a_i-z}{W(z)}\right)^n n^{-\theta_i}.$$
The function $F_n$ is holomorphic on $D$ and has the same zeros there as $B_n$. By Proposition~\ref{prop:deg}, $B_n$ is not identically zero for all sufficiently large $n$. By \eqref{eq:essential-normalized-upper}, for every compact $L\Subset D$ there is $C_L$ such that
$$\sup_L\log|F_n|\le C_L n^\alpha$$
for all sufficiently large $n$.

Let $\mathcal G_D$ be the set of points $z\in D\setminus\mathfrak S_{i,D}$ for which, in the chamber $V$ containing $z$, one active saddle is strictly dominant:
$$\Re\bigl((\omega_{\nu_\ast}-\omega_\nu)\eta_D(z)\bigr)>0
\qquad(\nu\in I_{i,V}\setminus\{\nu_\ast\})$$
for some $\nu_\ast\in I_{i,V}$. This set is open and dense in $D$. Indeed,
$$\eta_D'(z)=-\frac{m_i}{m_i+1}\frac{\eta_D(z)}{z-a_i}\ne0,$$
so, for every $\nu\ne\mu$, the harmonic functions
$$\Re((\omega_\nu-\omega_\mu)\eta_D(z)),\qquad
\Im((\omega_\nu-\omega_\mu)\eta_D(z))$$
are nonconstant and their zero sets have empty interior. Thus $D\setminus(\mathfrak S_{i,D}\cup\mathfrak A_{i,D})$ is dense; at each point of this dense set the real parts of all limiting saddle phases are distinct, so the nonempty active set in the relevant chamber has a unique largest real part. The strict inequalities persist locally, hence $\mathcal G_D$ is open.

If $z_0\in\mathcal G_D$, Corollary~\ref{cor:essential-one-saddle} applied to the one-point compact set $\{z_0\}$ gives, for all large $n$,
$$F_n(z_0)=\sigma_{i,\nu_\ast,V}\mathcal A_{i,\nu_\ast}(z_0)
\exp\!\bigl(\Xi_{i,\nu_\ast}(z_0;n)\bigr)
\left(1+O_{z_0}\!\left(n^{-1/(m_i+1)}\right)\right).$$
The prefactor is nonzero and \eqref{eq:Xi-local-unified-leading} gives $|\Re\Xi_{i,\nu_\ast}(z_0;n)|=O_{z_0}(n^\alpha)$; hence $F_n(z_0)\ne0$ and
$$|\log|F_n(z_0)||\le C_{z_0}n^\alpha$$
for all sufficiently large $n$. Lemma~\ref{lem:jensen-poisson-l1} therefore gives
$$\|\log|F_n|\|_{L^1(K)}=O_K(n^\alpha).$$

On $K$,
$$\log|B_n|=\log|F_n|+\theta_i\log n+\log n!+n\bigl(\log|W|-\log|a_i-z|\bigr).$$
Using Proposition~\ref{prop:deg}, $\deg B_n=\kappa n+O(1)$ and
$$\frac{\theta_i\log n+\log n!-\log|\gamma_n|-s_n}{\deg B_n}+\frac{\sigma}{\kappa}
=O\!\left(\frac{\log n}{n}\right).$$
Since $\log|W|-\log|a_i-z|$ is bounded on $K$, replacing $n/\deg B_n$ by $1/\kappa$ contributes $O_K(n^{-1})$. Hence
$$
\bigl\|\widetilde L_n-\Psi_i\bigr\|_{L^1(K)}
\le\frac{1}{\deg B_n}\|\log|F_n|\|_{L^1(K)}+O_K\!\left(\frac{\log n}{n}\right).
$$
Since $\alpha-1=-1/(m_i+1)$, \eqref{eq:essential-L1-rate} follows. A compact subset of $\mathcal V_i^\circ\setminus\Sigma$ is covered by finitely many simply connected open sets of the above kind, so the final statement follows by summing the estimates.
\end{proof}

\section{Fixed-scale global Voronoi law}\label{sec:global}

In this section we prove the global zero asymptotics for the sequence of derivatives of a hyperexponential function $f$. The edge part of the limit measure comes from the competition of nearest singularities on open Voronoi cells, while the atomic part at $\Z(T)$ is the fixed-scale shadow of clusters collapsing toward the finite essential singularities. The microscopic distribution of those clusters is treated later in Section~\ref{sec:microscopic}, and its connection with the fixed-scale atoms is recorded in Corollary~\ref{cor:atomic-local-bridge}. Recall from \eqref{eq:hdef} and \eqref{eq:kappadef} that $h$ is the degree of the polynomial part $H$ of $E=S/T$ when $H$ is nonconstant, and $h=0$ otherwise, while $\kappa=d+h-1$. To isolate only the zero distribution, write $\gamma_n:=\lc(B_n)$ and set
\begin{equation}\label{eq:sndef}
s_n:=
\begin{cases}
0, & \text{if } h=0,\\
\log n!, & \text{if } h>0.
\end{cases}
\end{equation}

For every $n$ with $\deg B_n>0$, define
\begin{equation}\label{eq:Ltildedef}
\widetilde L_n(z):=\frac{1}{\deg B_n}\Bigl(\log|B_n(z)|-\log|\gamma_n|-s_n\Bigr).
\end{equation}
The functions $\widetilde L_n$ are renormalized logarithmic potentials. The subtracted constants do not affect the Laplacian, and Poincar\'e--Lelong gives
\[
(2\pi)^{-1}\Delta\widetilde L_n
=\frac{1}{\deg B_n}\sum_{B_n(\zeta)=0}\operatorname{mult}_\zeta(B_n)\,\delta_\zeta
\]
whenever $\deg B_n>0$.

Define
\begin{equation}\label{eq:Psidef}
\Psi(z):=\frac{1}{\kappa}\bigl(\log|W(z)|-\log\rho(z)-\sigma\bigr),\qquad z\in\C\setminus\Sigma.
\end{equation}
Equivalently,
$$\Psi(z)=\max_{1\le i\le N}\Psi_i(z),\qquad
\Psi_i(z):=\frac{1}{\kappa}\bigl(\log|W(z)|-\log|z-a_i|-\sigma\bigr).$$
Since $\ord_{a_j}W\ge1$ for every $j$, each $\Psi_i$ has nonnegative logarithmic coefficient at every site: the coefficient is $\ord_{a_i}W-1$ at $a_i$ and $\ord_{a_j}W$ at $a_j\ne a_i$. Hence each $\Psi_i$, and therefore $\Psi=\max_i\Psi_i$, extends subharmonically across $\Sigma$; in what follows $\Psi$ denotes this subharmonic extension.
\begin{lem}[local upper bounds]\label{lem:locub}
The family $\{\widetilde L_n:\deg B_n>0\}$ is locally uniformly bounded from above on~$\C$.
\end{lem}

\begin{proof}
Fix a compact set $K\subset\C$ and put $F:=\Sigma\cup \Z(P_\sharp)$. Choose closed discs $D_1,\dots,D_L$ covering $K$ with $\partial D_\nu\cap F=\varnothing$ for every $\nu$, and set $\Gamma:=\bigcup_\nu\partial D_\nu$. Since $\Gamma\cap\Sigma=\varnothing$, fix $r:=\tfrac12\operatorname{dist}(\Gamma,\Sigma)>0$. Since $\Gamma\cap F=\varnothing$, for $z\in\Gamma$ and $|\xi|=r$ Proposition~\ref{prop:gen} applies and the function
$$\frac{P(z+\xi)Q(z)}{P(z)Q(z+\xi)}\,\exp\!\bigl(E(z+\xi)-E(z)\bigr)$$
is continuous on the compact set $\Gamma\times\{|\xi|=r\}$. Hence there is $\mathcal M>0$ such that
$$\left|\frac{P(z+\xi)Q(z)}{P(z)Q(z+\xi)}\,\exp\!\bigl(E(z+\xi)-E(z)\bigr)\right|\le\mathcal M\qquad(z\in\Gamma,\ |\xi|=r).$$
By Cauchy's estimate applied to~\eqref{eq:gen},
$$\left|\frac{A_n(z)}{n!}\right|\le\mathcal M\,r^{-n}\qquad(z\in\Gamma,\ n\ge0).$$
Using $B_n=P_\sharp W^n A_n$, one gets
\begin{align*}
\widetilde L_n(z)
&\le\frac{n}{\deg B_n}\bigl(\log\sup_\Gamma|W|-\log r\bigr)+\frac{\sup_\Gamma\log|P_\sharp|+\log\mathcal M}{\deg B_n}\\
&\qquad+\frac{\log n!-\log|\gamma_n|-s_n}{\deg B_n}.
\end{align*}
By Proposition~\ref{prop:deg}, the right-hand side is bounded above uniformly in~$n$. Hence there is a constant $C_\Gamma$ with $\widetilde L_n(z)\le C_\Gamma$ for all $z\in\Gamma$ and all $n$ for which $\widetilde L_n$ is defined. Since each $\widetilde L_n$ is subharmonic, the maximum principle on every $D_\nu$ gives $\sup_K\widetilde L_n\le C_\Gamma$.
\end{proof}

\begin{theorem}[fixed-scale Voronoi law]\label{thm:limit}
\mbox{}
\begin{enumerate}
\item \textbf{(Convergence of potentials.)}
For all sufficiently large $n$ one has $\deg B_n>0$, and the resulting sequence $\{\widetilde L_n\}$ satisfies
$$\widetilde L_n\longrightarrow\Psi\qquad\text{in }L^1_{\mathrm{loc}}(\C).$$
\item \textbf{(Limit measure on $\C$.)}
For those $n$, the normalized zero-counting measures
\[
\mu_n:=(2\pi)^{-1}\Delta\widetilde L_n
=\frac{1}{\deg B_n}\sum_{\zeta:\,B_n(\zeta)=0}\operatorname{mult}_\zeta(B_n)\,\delta_\zeta
\]
converge vaguely on $\C$ to
\begin{equation}\label{eq:measure}
\mu_{\mathrm{fix}}:=\frac{1}{2\pi}\Delta\Psi=\frac{1}{\kappa}\sum_{j=1}^{\check t} m_j\,\delta_{c_j}+\frac{1}{2\pi\kappa}\sum_{1\le i<j\le N}\frac{|a_i-a_j|}{|\cdot-a_i|\,|\cdot-a_j|}\,\dd\ell\big|_{E_{ij}},
\end{equation}
where an empty edge contributes zero; on a nonempty edge $E_{ij}$ the displayed summand has density $\zeta\mapsto |a_i-a_j|/(|\zeta-a_i|\,|\zeta-a_j|)$ with respect to arclength.
\item \textbf{(Mass.)}
The total finite-plane mass is
\begin{equation}\label{eq:mass}
\mu_{\mathrm{fix}}(\C)=\frac{d-1}{\kappa}=1-\frac{h}{\kappa}.
\end{equation}
In particular, if $h=0$ then $\mu_n$ converges weakly to $\mu_{\mathrm{fix}}$ on $\C$.
\item \textbf{(Convergence on $\widehat\C$.)}
Identify each $\mu_n$ from part~\textup{(2)} with the probability measure on $\widehat\C$ that assigns mass $0$ to $\{\infty\}$. Then
$$\mu_n\longrightarrow\mu_{\mathrm{fix}}+\frac{h}{\kappa}\,\delta_\infty$$
weakly on $\widehat\C$, where $\delta_\infty$ denotes the unit point mass at $\infty$.
\item \textbf{(Cauchy transform on Voronoi interiors.)}
Write $W(z)=\prod_{j=1}^N(z-a_j)^{\varpi_j}$, where $\varpi_j:=\ord_{a_j}W$. On each Voronoi interior $\mathcal V_i^\circ\setminus\Sigma$,
\begin{equation}\label{eq:cauchybranch}
C_{\mu_{\mathrm{fix}}}(z)=2\partial\Psi(z)=\frac{1}{\kappa}\left(\frac{W'(z)}{W(z)}-\frac{1}{z-a_i}\right)=\frac{1}{\kappa}\sum_{j=1}^N\frac{\varpi_j-\delta_{ij}}{z-a_j},
\end{equation}
where $C_{\mu_{\mathrm{fix}}}$ is the Cauchy transform defined in \S~\ref{sec:pottheory}, $\delta_{ij}$ is the Kronecker delta, and $2\partial=\partial_x-i\partial_y$.
\end{enumerate}
\end{theorem}

\begin{proof}[Proof of Theorem~\ref{thm:limit}, parts~\textup{(1)}--\textup{(4)}]
Only two local inputs enter the argument. On pole cells Theorem~\ref{prop:local-coef-unified}\textup{(1)} gives uniform asymptotics, and on essential-singularity cells Corollary~\ref{cor:essential-l1} gives the required local $L^1$ control.

Choose $n_0$ so that $\deg B_n>0$ for all $n\ge n_0$, and set
$$\Omega:=\bigcup_{i=1}^N\bigl(\mathcal V_i^\circ\setminus\Sigma\bigr).$$
Let $K\Subset\Omega$. Then $K$ is covered by finitely many compact sets $K_1,\dots,K_M$, each contained in a single Voronoi interior. If $K_\alpha\Subset\mathcal V_i^\circ\setminus\Sigma$ and $a_i\in \Z(Q)\setminus \Z(T)$, Theorem~\ref{prop:local-coef-unified}\textup{(1)} gives
\[
\log|B_n(z)|
=\log n!+n\bigl(\log|W(z)|-\log|a_i-z|\bigr)+O_{K_\alpha}(\log n)
\]
uniformly on $K_\alpha$. Proposition~\ref{prop:deg} gives
\[
\frac{\log n!-\log|\gamma_n|-s_n}{\deg B_n}
=-\frac{\sigma}{\kappa}+O\!\left(\frac{\log n}{n}\right),
\qquad
\frac{n}{\deg B_n}=\frac1\kappa+O\!\left(\frac1n\right).
\]
Hence
$$\widetilde L_n=\Psi_i+O_{K_\alpha}\!\left(\frac{\log n}{n}\right)$$
uniformly on $K_\alpha$. If $K_\alpha\Subset\mathcal V_i^\circ\setminus\Sigma$ and $m_i\ge1$, Corollary~\ref{cor:essential-l1} gives
$$\|\widetilde L_n-\Psi_i\|_{L^1(K_\alpha)}=O_{K_\alpha}\!\left(n^{-1/(m_i+1)}\right).$$
Summing over $\alpha$ yields
$$\widetilde L_n\to\Psi\qquad\text{in }L^1_{\mathrm{loc}}(\Omega).$$
More precisely, on compact subsets of pole cells the convergence is uniform with rate $O(\log n/n)$, while on compact subsets of an essential-singularity cell attached to $a_i$ it holds in $L^1$ with rate $O\!\left(n^{-1/(m_i+1)}\right)$.

Since $\C\setminus\Omega$ is the union of the Voronoi diagram and the finite set $\Sigma$, it has planar Lebesgue measure zero. Let $\{\widetilde L_{n_j}\}$ be any subsequence with $n_j\ge n_0$. By Lemma~\ref{lem:locub} and the compactness theorem for subharmonic functions \cite{Ra}, either $\widetilde L_{n_j}\to-\infty$ uniformly on compact subsets or a further subsequence converges in $L^1_{\mathrm{loc}}(\C)$ to a subharmonic function $u$. The first alternative is impossible, since some disk compactly contained in $\Omega$ has $L^1$-convergence to the finite harmonic branch of $\Psi$. The same subsequence converges to $u$ in $L^1_{\mathrm{loc}}(\Omega)$, while the whole sequence converges there to $\Psi$; hence $u=\Psi$ almost everywhere on $\Omega$, and therefore almost everywhere on $\C$. Both $u$ and $\Psi$ are subharmonic, hence the upper-semicontinuous regularizations of the same $L^1_{\mathrm{loc}}(\C)$-class. Therefore $u=\Psi$ everywhere. Thus every subsequence has a further subsequence converging to $\Psi$ in $L^1_{\mathrm{loc}}(\C)$, and the whole sequence converges.

Since $\widetilde L_n\to\Psi$ in $L^1_{\mathrm{loc}}(\C)$, one has $\mu_n=(2\pi)^{-1}\Delta\widetilde L_n\to(2\pi)^{-1}\Delta\Psi=\mu_{\mathrm{fix}}$ distributionally; because these are positive Radon measures, this is vague convergence on $\C$. Since
$$W(z)=z^d+O(z^{d-1}),\qquad \rho(z)=|z|+O(1)\qquad(|z|\to\infty),$$
one has
$$\Psi(z)=\frac{d-1}{\kappa}\log|z|+O(1).$$
Let \(R\) be outside the finite set of radii for which \(\partial B_R\) meets a Voronoi vertex or a point of \(\Sigma\). Since, locally away from \(\Sigma\), \(\Psi\) is the maximum of finitely many harmonic functions with bounded gradients, and the only logarithmic point masses occur at \(\Sigma\), the Riesz measure of \(\Psi\) gives no mass to \(\partial B_R\). Approximating \(\mathbf 1_{B_R}\) by smooth radial cutoffs in the distributional identity \(\Delta\Psi=2\pi\mu_{\mathrm{fix}}\) gives Green's formula
\[
\mu_{\mathrm{fix}}(B_R)=\frac1{2\pi}\int_{|z|=R}\partial_{\mathbf n}\Psi\,\dd\ell .
\]
On \(|z|=R\), one has \(\partial_{\mathbf n}\log|W|=d/R+O(R^{-2})\), and on each Voronoi arc of the circle,
\(\partial_{\mathbf n}\log\rho=1/R+O(R^{-2})\). Letting \(R\to\infty\) through such values yields
\[
\mu_{\mathrm{fix}}(\C)=\frac{d-1}{\kappa},
\]
which is~\eqref{eq:mass}. If $h=0$, then $\mu_{\mathrm{fix}}(\C)=1$. Vague convergence of probability measures to a probability measure on the locally compact space $\C$ implies tightness, hence weak convergence on $\C$.

For the spherical statement, let $\nu$ be a subsequential weak limit of $\mu_n$ on the compact space $\widehat\C$, say $\mu_{n_j}\to\nu$ weakly on $\widehat\C$. For every continuous compactly supported function $\varphi:\C\to\mathbb R$, extend $\varphi$ to $\widehat\C$ by setting $\varphi(\infty)=0$. Then weak convergence of $\mu_{n_j}$ on $\widehat\C$ and vague convergence of $\mu_n$ on $\C$ give
$$\int_{\widehat\C}\varphi\,\dd\nu=\lim_{j\to\infty}\int_{\C}\varphi\,\dd\mu_{n_j}=\int_{\C}\varphi\,\dd\mu_{\mathrm{fix}}.$$
Hence $\nu|_{\C}=\mu_{\mathrm{fix}}$. Since $\nu(\widehat\C)=1$ and $\mu_{\mathrm{fix}}(\C)=(d-1)/\kappa$, one gets $\nu(\{\infty\})=1-\mu_{\mathrm{fix}}(\C)=h/\kappa$. Therefore $\nu=\mu_{\mathrm{fix}}+(h/\kappa)\delta_\infty$. Since every subsequential limit is the same, the whole sequence converges.
\end{proof}

\begin{proof}[Proof of Theorem~\ref{thm:limit}, explicit form of $\mu_{\mathrm{fix}}$, and the Cauchy-transform formula]
Write $W(z)=\prod_{j=1}^N(z-a_j)^{\varpi_j}$. On $\mathcal V_i^\circ\setminus\Sigma$ one has
$$\Psi(z)=\frac{1}{\kappa}\left((\varpi_i-1)\log|z-a_i|+\sum_{j\ne i}\varpi_j\log|z-a_j|-\sigma\right).$$
Near $a_i$, the branch $\Psi_i$ dominates because $\log|z-a_i|\to-\infty$ and the coefficient of $\log|z-a_i|$ in $\Psi_i$ is $\varpi_i-1<\varpi_i$, whereas in $\Psi_j$ it is $\varpi_i$ for $j\ne i$. Hence $\Psi=\Psi_i$ in a punctured neighborhood of $a_i$, and the atomic coefficient there is $(\varpi_i-1)/\kappa$, namely $m_j/\kappa$ when $a_i=c_j\in \Z(T)$ and $0$ when $a_i\in \Z(Q)\setminus \Z(T)$. Away from $\Sigma$ and the Voronoi edges the function is harmonic. On the relative interior of each nonempty edge $E_{ij}$, $\Psi=\max(\Psi_i,\Psi_j)$ and
$$\Psi_i-\Psi_j=\frac{1}{\kappa}\bigl(\log|z-a_j|-\log|z-a_i|\bigr).$$
Let $n_{ij}$ be a unit normal to $E_{ij}$, chosen so that $\Psi_i>\Psi_j$ on one side of the edge. Near a point of the relative interior of $E_{ij}$ that meets no other Voronoi edge, the functions $\Psi_i$ and $\Psi_j$ are harmonic, all other branches are strictly smaller, and $\Psi=\max(\Psi_i,\Psi_j)$. The jump formula for the normal derivative therefore gives
$$\Delta\Psi=\bigl(\partial_{n_{ij}}\Psi_i-\partial_{n_{ij}}\Psi_j\bigr)\,\dd\ell\big|_{E_{ij}}=\bigl|\partial_{n_{ij}}(\Psi_i-\Psi_j)\bigr|\,\dd\ell\big|_{E_{ij}}.$$
On $E_{ij}$ one has $|z-a_i|=|z-a_j|=:r$, and
$$|\nabla(\Psi_i-\Psi_j)(z)|=\frac{1}{\kappa}\left|\frac{z-a_j}{|z-a_j|^2}-\frac{z-a_i}{|z-a_i|^2}\right|=\frac{|a_i-a_j|}{\kappa r^2}=\frac{|a_i-a_j|}{\kappa |(z-a_i)(z-a_j)|}.$$
Since $\Psi_i-\Psi_j$ is constant on $E_{ij}$, this gradient is normal to $E_{ij}$. Hence the density of $\Delta\Psi$ along $E_{ij}$ is
$$\frac{|a_i-a_j|}{\kappa |(z-a_i)(z-a_j)|}\,\dd\ell\big|_{E_{ij}}.$$
Dividing by $2\pi$ gives the edge term in~\eqref{eq:measure}. At a Voronoi vertex not belonging to $\Sigma$, the function $\Psi$ is the maximum of finitely many harmonic functions with bounded gradients; its Riesz measure therefore has no atom there. The edge density is $O(|\zeta|^{-2})$ on every unbounded edge. Hence
\[
\int_{\C}\log(2+|\zeta|)\,\dd\mu_{\mathrm{fix}}(\zeta)<\infty
\]
and, for $z\notin\operatorname{supp}\mu_{\mathrm{fix}}$, the integral defining $C_{\mu_{\mathrm{fix}}}(z)$ converges absolutely. The logarithmic potential
$$U_{\mu_{\mathrm{fix}}}(z):=\int\log|z-\zeta|\,\dd\mu_{\mathrm{fix}}(\zeta)$$
is a subharmonic function, finite off the atoms, and satisfies $U_{\mu_{\mathrm{fix}}}(z)=\mu_{\mathrm{fix}}(\C)\log|z|+O(1)$ at infinity. Distributionally, $\Delta(\Psi-U_{\mu_{\mathrm{fix}}})=0$ on $\C$; hence the difference has an entire harmonic representative. Since
$$\Psi(z)=\mu_{\mathrm{fix}}(\C)\log|z|+O(1),$$
that representative is bounded at infinity and therefore constant. Differentiating off the support gives $C_{\mu_{\mathrm{fix}}}(z)=2\partial\Psi(z)$ on $\C\setminus\operatorname{supp}\mu_{\mathrm{fix}}$. On $\mathcal V_i^\circ\setminus\Sigma$ one has $\Psi=\Psi_i$, which gives~\eqref{eq:cauchybranch}.
\end{proof}

\begin{cor}[fixed-scale law for zeros of \(f^{(n)}\)]\label{cor:fixed-scale-derivative-zeros}
For all sufficiently large $n$, define
\[
\nu_n:=\frac{1}{\deg B_n}
\sum_{\substack{\zeta\in\C\setminus\Sigma\\ f^{(n)}(\zeta)=0}}
\operatorname{mult}_\zeta(f^{(n)})\,\delta_\zeta .
\]
Then $\nu_n=\mu_n$. Consequently the zeros of $f^{(n)}$ away from the singular set satisfy the same finite-plane vague convergence, the same weak convergence on $\C$ when $h=0$, and the same spherical convergence to $\mu_{\mathrm{fix}}+(h/\kappa)\delta_\infty$ when $h>0$ as in Theorem~\ref{thm:limit}.
\end{cor}

\begin{proof}
On $\C\setminus\Sigma$, the factor $P_Te^E/(QW^n)$ in \eqref{eq:fnrep} is holomorphic and nonvanishing, so the zero divisors of $f^{(n)}$ and $B_n$ coincide there, including multiplicities. Proposition~\ref{prop:local} gives $B_n(a)\ne0$ for every site $a\in\Sigma$, so no zero of $B_n$ is lost by restricting the derivative count to $\C\setminus\Sigma$. Hence $\nu_n=\mu_n$, and the convergence statements follow from Theorem~\ref{thm:limit}.
\end{proof}

\section{Microscopic local laws at essential singularities}\label{sec:microscopic}

\subsection{Reduced local models at essential singularities}\label{sec:local-models}

We first identify the exactly solvable highest-term reduction at a finite essential singularity, obtained by retaining the algebraic factor and the highest-order singular term in the exponent. This reduced model supplies the explicit microscopic measure used in the full local theorem proved in Subsection~\ref{sec:full-local-factor}.

Let $c_j$ be a zero of $T$ of multiplicity $m_j$, set $\alpha_j:=p_j-\nu_j$, and write
$$E(z)=\sum_{s=1}^{m_j}\lambda_{j,s}(z-c_j)^{-s}+E_j^{\mathrm{reg}}(z),\qquad \lambda_{j,m_j}\ne0,$$
where $E_j^{\mathrm{reg}}$ is holomorphic near $c_j$. Then
$$f(z)=(z-c_j)^{\alpha_j}\exp\!\left(\sum_{s=1}^{m_j}\lambda_{j,s}(z-c_j)^{-s}\right)\phi(z),$$
with $\phi$ holomorphic and nonvanishing near $c_j$. The reduced model is
$$g_{\alpha,m}(z):=(z-a)^\alpha \exp\!\left(\frac{\lambda}{(z-a)^m}\right),\qquad \lambda\ne0,$$
studied for all $m\ge1$ and $\alpha\in\mathbb Z$; the choice of $a\in\C$ is irrelevant.

Recall that a polynomial sequence $\{P_n\}_{n\ge0}$ is a \emph{Sheffer sequence} if its exponential generating function has the form $A(t)e^{xB(t)}$ with $A(0)\ne0$, $B(0)=0$, and $B'(0)\ne0$. If $\mathbf u=(u_0,\dots,u_{m-1})$ is an $m$-tuple of linear functionals on $\C[x]$, we say that $\{P_n\}_{n\ge0}$ is \emph{$m$-orthogonal} with respect to $\mathbf u$ if
$$\langle u_j,x^\nu P_n\rangle=0\quad\text{for }n\ge m\nu+j+1,$$
and
$$\langle u_j,x^\nu P_{m\nu+j}\rangle\ne0\quad\text{for every }\nu\ge0\text{ and }0\le j\le m-1.$$
We use the standard type-II multiple-orthogonality terminology in the following sense: for weights $w_0,\dots,w_{m-1}$ on $[0,\infty)$ and a multi-index $\boldsymbol n=(n_0,\dots,n_{m-1})$, a type-II multiple orthogonal polynomial is a nonzero polynomial $P$ of degree at most $|\boldsymbol n|$ satisfying
$$\int_0^\infty x^\nu P(x)w_j(x)\,\dd x=0\qquad(0\le\nu<n_j,\ 0\le j\le m-1),$$
and the multi-index is normal if such a polynomial is unique up to a nonzero scalar and has degree $|\boldsymbol n|$.
The polynomial sequence $\Pi_n^{(\alpha,m)}$ defined in Proposition~\ref{prop:micro-higher} is the Varma--Ta\c{s}delen Laguerre-type Sheffer sequence; see \cite{VT}. For $\alpha<0$ its $m$-orthogonality is equivalent, after a triangular change of Laguerre moment functionals, to type-II multiple orthogonality for a staircase multi-index. The corresponding biorthogonal ensemble is the Laguerre Muttalib--Borodin ensemble with parameter $\theta=1/m$ \cite{Borodin98,FW}. For $\alpha\ge0$, writing $\alpha=qm+s$ with $q\ge0$ and $0\le s<m$, the polynomial $\Pi_n^{(\alpha,m)}$ has a zero at the origin of multiplicity $q+1$ for every $n\ge\alpha+1$; this fixed loss of roots is negligible in normalized zero counting, and the limiting zero law follows from the negative-exponent case by a finite connection formula. The case $m=1$ reduces to generalized Laguerre polynomials, and after the reciprocal change of variables the microscopic zero distribution is the Marchenko--Pastur law.

\begin{prop}[reduced local model and Sheffer sequence]\label{prop:micro-higher}
Let $m\ge1$, $\alpha\in\mathbb Z$, $a\in\C$, and $\lambda\in\C^\times$, and define
$$g_{\alpha,m}(z):=(z-a)^\alpha \exp\!\left(\frac{\lambda}{(z-a)^m}\right),\qquad x:=-\frac{\lambda}{(z-a)^m}.$$
Then the following hold.

\begin{enumerate}
\item For $z\ne a$, the derivatives satisfy
$$g_{\alpha,m}^{(n)}(z)=(-1)^n(z-a)^{\alpha-n}\Pi_n^{(\alpha,m)}(x)\exp\!\left(\frac{\lambda}{(z-a)^m}\right),\qquad n\ge0,$$
where $\{\Pi_n^{(\alpha,m)}\}_{n\ge0}$ is the unique polynomial Sheffer sequence satisfying 
$$\sum_{n\ge0}\Pi_n^{(\alpha,m)}(x)\frac{t^n}{n!}
=(1-t)^\alpha \exp\!\bigl(x(1-(1-t)^{-m})\bigr).$$

\item The sequence is characterized equivalently by $\Pi_0^{(\alpha,m)}=1$ and
$$\Pi_{n+1}^{(\alpha,m)}(x)=mx\bigl(\Pi_n^{(\alpha,m)}\bigr)'(x)+(n-\alpha-mx)\Pi_n^{(\alpha,m)}(x).$$
Explicitly,
$$\Pi_n^{(\alpha,m)}(x)=\sum_{k=0}^n\frac{x^k}{k!}\sum_{j=0}^k(-1)^j\binom{k}{j}(mj-\alpha)_n,$$
where $(a)_0:=1$ and $(a)_n:=a(a+1)\cdots(a+n-1)$ for $n\ge1$. With $\beta:=-1-\alpha/m$, one has
$$\Pi_n^{(\alpha,m)}(x)=P_n^{(\beta)}(x;m),$$
where $P_n^{(\beta)}(x;m)$ is the Varma--Ta\c{s}delen Laguerre-type family defined by
$$\sum_{n\ge0}P_n^{(\beta)}(x;m)\frac{t^n}{n!}
=(1-t)^{-(\beta+1)m}\exp\!\bigl(-x((1-t)^{-m}-1)\bigr),$$
see \cite[Eq.~(3)]{VT}.

For $m=1$ the same generating relation reduces to the classical Laguerre generating function, i.e., 
$$\Pi_n^{(\alpha,1)}(x)=n!L_n^{(-\alpha-1)}(x).$$ 
In particular, $\deg\Pi_n^{(\alpha,m)}=n$ and $\lc\!\bigl(\Pi_n^{(\alpha,m)}\bigr)=(-m)^n$ for every $n$.
\end{enumerate}
\end{prop}

\begin{proof}
Write $w:=z-a$ and $x=-\lambda w^{-m}$. Since $x'=-mx/w$, differentiating
$$(-1)^n w^{\alpha-n}\Pi_n(x)e^{\lambda w^{-m}}$$
gives
$$\Pi_{n+1}^{(\alpha,m)}(x)=mx\bigl(\Pi_n^{(\alpha,m)}\bigr)'(x)+(n-\alpha-mx)\Pi_n^{(\alpha,m)}(x),\qquad \Pi_0^{(\alpha,m)}=1.$$
Taylor's formula for $g_{\alpha,m}(z+\xi)$ with $t=-\xi/w$ gives
$$\sum_{n\ge0}\Pi_n^{(\alpha,m)}(x)\frac{t^n}{n!}
=(1-t)^\alpha \exp\!\bigl(x(1-(1-t)^{-m})\bigr).$$
Expanding that generating function gives
$$\Pi_n^{(\alpha,m)}(x)=n![t^n]\!\left((1-t)^\alpha \exp\!\bigl(x(1-(1-t)^{-m})\bigr)\right)
=\sum_{k=0}^n\frac{x^k}{k!}\sum_{j=0}^k(-1)^j\binom{k}{j}(mj-\alpha)_n.$$
Comparing with the defining generating function of $P_n^{(\beta)}(x;m)$ and using $\beta=-1-\alpha/m$ gives $\Pi_n^{(\alpha,m)}(x)=P_n^{(\beta)}(x;m)$. For $m=1$ this reduces to the classical Laguerre generating function, hence $\Pi_n^{(\alpha,1)}(x)=n!L_n^{(-\alpha-1)}(x)$. The highest-degree term in the recurrence is $-mx\,\Pi_n^{(\alpha,m)}(x)$, so $\deg\Pi_n^{(\alpha,m)}=n$ and $\lc\!\bigl(\Pi_n^{(\alpha,m)}\bigr)=(-m)^n$.
\end{proof}

\begin{prop}[Laguerre-type $m$-orthogonality for $\alpha<0$]\label{prop:micro-dorth}
Assume $m\ge1$ and $\alpha\in\mathbb Z$ with $\alpha<0$, and let $\{\Pi_n^{(\alpha,m)}\}_{n\ge0}$ be as in Proposition~\ref{prop:micro-higher}. For $0\le r\le m-1$ set
$$\lambda_r:=\frac{r-\alpha}{m}>0.$$
For $0\le j\le m-1$, define a linear functional on $\C[x]$ by
$$\langle u_j,\varphi\rangle:=\frac1{j!}\sum_{r=0}^j(-1)^r\binom{j}{r}\frac1{\Gamma(\lambda_r)}
\int_0^\infty \varphi(x)x^{\lambda_r-1}e^{-x}\,\dd x.$$
Then
$$\langle u_j,x^\nu\Pi_n^{(\alpha,m)}\rangle=0\quad\text{for }n\ge m\nu+j+1,$$
and
$$\langle u_j,x^\nu\Pi_{m\nu+j}^{(\alpha,m)}\rangle\ne0\quad\text{for every }\nu\ge0.$$
Hence $\{\Pi_n^{(\alpha,m)}\}_{n\ge0}$ is $m$-orthogonal with respect to $\mathbf{u}=(u_0,\dots,u_{m-1})$. Equivalently, after the parameter change $\beta=-1-\alpha/m$, this is the Varma--Ta\c{s}delen Laguerre-type family from~\cite{VT}.
\end{prop}

\begin{proof}
Write $\Pi_n:=\Pi_n^{(\alpha,m)}$. Fix $\nu\ge0$. Choose $\varepsilon>0$ so small that $\Re((1-t)^{-m})\ge c>0$ for $|t|\le\varepsilon$. On this disc the generating function is dominated, after multiplication by $y^{\nu+\lambda_r-1}e^{-y}$, by an integrable multiple of $y^{\nu+\lambda_r-1}e^{-cy}$; termwise integration below is therefore justified by dominated convergence. Proposition~\ref{prop:micro-higher} gives
\begin{align*}
\sum_{n\ge0}\langle u_j,x^\nu\Pi_n\rangle\frac{t^n}{n!}
&=\frac1{j!}\sum_{r=0}^j(-1)^r\binom{j}{r}\frac1{\Gamma(\lambda_r)}
\int_0^\infty y^\nu\left(\sum_{n\ge0}\Pi_n(y)\frac{t^n}{n!}\right)y^{\lambda_r-1}e^{-y}\,\dd y\\
&=\frac1{j!}\sum_{r=0}^j(-1)^r\binom{j}{r}\frac1{\Gamma(\lambda_r)}
\int_0^\infty y^{\nu+\lambda_r-1}(1-t)^\alpha e^{-y(1-t)^{-m}}\,\dd y\\
&=\frac1{j!}\sum_{r=0}^j(-1)^r\binom{j}{r}(\lambda_r)_\nu
(1-t)^{\alpha+m(\nu+\lambda_r)}\\
&=\frac{(1-t)^{m\nu}}{j!}\sum_{r=0}^j(-1)^r\binom{j}{r}(\lambda_r)_\nu(1-t)^r,
\end{align*}
because $m\lambda_r+\alpha=r$. The right-hand side is a polynomial in $t$ of degree $m\nu+j$. Its leading coefficient is
$$\frac{(-1)^{m\nu}}{j!}(\lambda_j)_\nu\ne0,$$
because $\lambda_j>0$. Therefore the coefficient of $t^n$ vanishes for $n>m\nu+j$ and is nonzero for $n=m\nu+j$, which is exactly the stated $m$-orthogonality. The final identification with the Varma--Ta\c{s}delen family is the parameter change already recorded in Proposition~\ref{prop:micro-higher}.
\end{proof}

For the Laguerre weights used in Lemma~\ref{lem:nva-staircase}, write $\alpha_{\rm loc}<0$ for the local exponent and set
\[
\chi_r:=\frac{r-\alpha_{\rm loc}}{m}-1,\qquad 0\le r\le m-1.
\]
Then $\chi_r>-1$ and $\chi_r-\chi_s\notin\mathbb Z$ for $r\ne s$. We use the standard algebraic-Chebyshev (AT) property for these Laguerre weights: for every multi-index $\boldsymbol n$, the functions $x^\nu x^{\chi_r}e^{-x}$, $0\le \nu<n_r$, $0\le r\le m-1$, form a Chebyshev system on $(0,\infty)$. Consequently, the relevant multi-indices are normal, their zeros are positive and simple, and nearest-neighbor type-II polynomials interlace.

\begin{prop}[staircase multiple Laguerre and Laguerre Muttalib--Borodin ensemble]\label{prop:micro-mop-ensemble}
Assume $m\ge1$ and $\alpha\in\mathbb Z$ with $\alpha<0$. For $0\le r\le m-1$ set
\[
w_r(x):=\frac{x^{(r-\alpha)/m-1}e^{-x}}{\Gamma((r-\alpha)/m)},\qquad x>0.
\]
If $n=am+s$, $0\le s<m$, put
\[
\boldsymbol n(n):=(\underbrace{a+1,\dots,a+1}_{s\text{ times}},\underbrace{a,\dots,a}_{m-s\text{ times}}).
\]
Then
\[
P_n(x):=(-m)^{-n}\Pi_n^{(\alpha,m)}(x)
\]
is the monic type-II multiple Laguerre polynomial of the first kind for the weights $w_0,\dots,w_{m-1}$ and multi-index $\boldsymbol n(n)$.

Moreover, with
\[
g_{k+1}(x):=\frac{x^{(k-\alpha)/m-1}e^{-x}}{\Gamma((k-\alpha)/m)},\qquad 0\le k\le n-1,
\]
the biorthogonal density
\[
\frac1{Z_n}\det[x_k^{\ell-1}]_{\ell,k=1}^n\det[g_\ell(x_k)]_{\ell,k=1}^n\prod_{k=1}^n\dd x_k
\]
is equivalently the Laguerre Muttalib--Borodin ensemble with parameter $\theta=1/m$ and Laguerre exponent $-\alpha/m-1$:
\[
\frac1{\widetilde Z_n}
\prod_{1\le p<q\le n}(x_q-x_p)(x_q^{1/m}-x_p^{1/m})
\prod_{p=1}^n x_p^{-\alpha/m-1}e^{-x_p}\,\dd x_p .
\]
Its average characteristic polynomial is
\[
\mathbb E\prod_{p=1}^n(z-x_p)=P_n(z).
\]
\end{prop}

\begin{proof}
Let $\mu_r$ be the Laguerre moment functional with density $w_r$. The functionals in Proposition~\ref{prop:micro-dorth} satisfy
\[
u_j=\frac1{j!}\sum_{r=0}^j(-1)^r\binom jr\mu_r .
\]
The change-of-basis matrix from $(\mu_0,\dots,\mu_{m-1})$ to $(u_0,\dots,u_{m-1})$ is lower triangular with nonzero diagonal, and its inverse is lower triangular. The multi-index $\boldsymbol n(n)$ is nonincreasing. If the type-II conditions for the $\mu_r$ hold, then for each $j$ and $0\le\nu<n_j(n)$ the $u_j$-condition is a linear combination of $\mu_r$-conditions with $r\le j$, all available because $n_r(n)\ge n_j(n)$. Conversely, if the $u_j$-conditions hold, then for each $r$ and $0\le\nu<n_r(n)$ the $\mu_r$-condition is a linear combination of $u_j$-conditions with $j\le r$, all available because $n_j(n)\ge n_r(n)$. Thus the two systems of orthogonality conditions are equivalent. Proposition~\ref{prop:micro-dorth} gives the corresponding $u_j$-conditions for $\Pi_n^{(\alpha,m)}$. The leading coefficient from Proposition~\ref{prop:micro-higher} makes $P_n=(-m)^{-n}\Pi_n^{(\alpha,m)}$ monic; normality follows from the AT-system property.

The functions $x^\nu w_r(x)$ with $0\le\nu<n_r(n)$ are, up to positive constants and in the order $k=m\nu+r$, the functions $g_{k+1}$, $0\le k<n$. Hence
\[
\det[g_\ell(x_k)]_{\ell,k=1}^n
=\text{const}\cdot\prod_{p=1}^n x_p^{-\alpha/m-1}e^{-x_p}
\det[x_p^{(\ell-1)/m}]_{\ell,p=1}^n,
\]
which gives the displayed Muttalib--Borodin density. Heine's identity for biorthogonal ensembles, obtained by inserting one extra Vandermonde row and applying Andreief's formula, shows that the average characteristic polynomial is the monic degree-$n$ polynomial orthogonal to $g_1,\dots,g_n$; by normality this polynomial is $P_n$.
\end{proof}

Thus, for $\alpha<0$, the reduced microscopic polynomial is a known finite-$n$ integrable-probability object: the average characteristic polynomial of a Laguerre Muttalib--Borodin, equivalently Borodin biorthogonal, ensemble \cite{Borodin98,FW}. Product-Ginibre ensembles give related multiple-orthogonal-polynomial and Meijer-$G$ kernels \cite{AIK,KZ}, but their finite-$n$ weights are not the powers $g_k$ above. The zero limit used here is the Neuschel--Van Assche multiple-Laguerre limit \cite{NVA}.

We use exactly two external asymptotic inputs. First, Neuschel--Van Assche \cite[Theorem~1.2 and \S\S~5--6]{NVA} give the zero distribution for the diagonal type-II multiple Laguerre polynomials of the first kind in total-degree scaling. In their notation the diagonal multi-index is $a\mathbf 1=(a,\ldots,a)$, the total degree is $N=ma$, and the zeros of $P_{a\mathbf 1}(N\zeta)$ have the limiting Cauchy transform displayed in Lemma~\ref{lem:nva-staircase}. The bounded-staircase extension used below is not quoted as a separate theorem from \cite{NVA}; it is obtained in the proof of Lemma~\ref{lem:nva-staircase} from the diagonal theorem, interlacing, and finitely many nearest-neighbor ratios. In this normalization the endpoint becomes
\[
\frac1m\frac{(m+1)^{m+1}}{m^m}=\left(\frac{m+1}{m}\right)^{m+1}=c_m
\]
after division by the total degree. Second, for the nearest-neighbor ratios in the diagonal first-kind multiple-Laguerre regime we use the ratio asymptotics obtained in Neuschel--Van Assche \cite[\S\S~5--6]{NVA} from Van Assche's theorem \cite[Theorem~1.2]{VAratio} by first working in distinct directions and then taking the diagonal limit. In the form needed here: for every fixed word in the steps $-\mathbf e_j$ and every compact $K\Subset\widehat\C\setminus[0,c_m]$, the ratios
\[
N\frac{P_{\boldsymbol k-\mathbf e_j}(N\zeta)}{P_{\boldsymbol k}(N\zeta)}
\]
which occur along the word are defined on $K$ for all large $N$ and converge locally uniformly along the asymptotic ray. The passage from diagonal indices to the bounded staircase indices is carried out in the proof below by interlacing and by this finite product of nearest-neighbor ratios; no further multiple-Laguerre asymptotic input is used.

\begin{lem}[staircase multiple-Laguerre input]\label{lem:nva-staircase}
Fix $m\ge1$ and $\alpha_0\in\mathbb Z$ with $\alpha_0<0$, and let $\Pi_n^{(\alpha_0,m)}$ be as in Proposition~\ref{prop:micro-higher}. Put
$$c_m:=\left(\frac{m+1}{m}\right)^{m+1},\qquad
\Omega:=\widehat\C\setminus[0,c_m].$$
Then the normalized zero-counting measures of $\Pi_n^{(\alpha_0,m)}(n\zeta)$ converge to a probability measure $\mu_m$ supported on $[0,c_m]$ whose Cauchy transform $C_m=C_{\mu_m}$ is the branch
$$C_m(\zeta)=1-v(\zeta)^{-m},\qquad v^{m+1}=m\zeta(v-1),\qquad
v(\zeta)=1+\frac{1}{m\zeta}+O(\zeta^{-2}).$$
Moreover, for every fixed $L$ and every subsequence of $n$ there is a further subsequence on which, for $0\le\ell\le L$,
$$
S_{n,\ell}(\zeta):=
\frac{n!}{(n-\ell)!}\,
\frac{\Pi_{n-\ell}^{(\alpha_0,m)}(n\zeta)}
{\Pi_n^{(\alpha_0,m)}(n\zeta)},\qquad S_{n,0}\equiv1,
$$
converges locally uniformly on $\Omega$ to a holomorphic function $H_\ell$; on each compact $K\Subset\Omega$ the denominators are zero-free for all sufficiently large $n$ in that subsequence. The limits satisfy $H_0\equiv1$ and $H_\ell(\infty)=0$ for $\ell\ge1$.
\end{lem}

\begin{proof}
Let $P_{\boldsymbol k}$ denote the monic type-II multiple Laguerre polynomial of the first kind for the weights $x^{(r-\alpha_0)/m-1}e^{-x}$, $0\le r\le m-1$. If $n=am+s$, $0\le s<m$, put
$$\boldsymbol n(n):=(\underbrace{a+1,\dots,a+1}_{s\text{ times}},\underbrace{a,\dots,a}_{m-s\text{ times}}).$$
By Proposition~\ref{prop:micro-mop-ensemble}, $(-m)^{-n}\Pi_n^{(\alpha_0,m)}=P_{\boldsymbol n(n)}$.

Neuschel--Van Assche prove the zero distribution for the diagonal indices $a\mathbf 1$, $\mathbf 1=(1,\ldots,1)$, with the argument scaled by the total degree $ma$. Their endpoint, after division by $m$ in the total-degree scaling, is
\[
\frac1m\,\frac{(m+1)^{m+1}}{m^m}
=\left(\frac{m+1}{m}\right)^{m+1}=c_m.
\]
The staircase $\boldsymbol n(n)$ is obtained from $a\mathbf 1$ by at most $m-1$ nearest-neighbor increments. At one increment the AT-system interlacing theorem gives interlacing of the two real simple zero sets. With a common spatial scale $N\asymp n$ and normalization by total degree, the corresponding distribution functions differ by $O(1/n)$; hence their integrals against any compactly supported $C^1$ test function $\chi$ differ by $O(\|\chi'\|_{L^1}/n)$. Iterating over the bounded number of increments gives the same estimate between the diagonal measure at degree $ma$ and the staircase measure when both are read at the common scale $ma$. Replacing $ma$ by $n=ma+s$ changes the integral of $\chi$ by $O_\chi(1/n)$, since $s$ is bounded. The diagonal theorem gives tightness in the one-point compactification of $[0,\infty)$; approximating continuous tests first by compactly supported $C^1$ tests and then using the preceding estimates gives the same zero distribution for $\boldsymbol n(n)$ with the scale $x=n\zeta$. This proves the first assertion once the Cauchy transform is identified.

For the ratio assertion, fix $L$ and first restrict to a residue class of $n$ modulo $m$. For each $0\le\ell\le L$, the multi-index $\boldsymbol n(n-\ell)$ is obtained from $\boldsymbol n(n)$ by a fixed word of $\ell$ nearest-neighbor decrements, depending only on the residue class and on $\ell$. Let $K\Subset\Omega$. Since $L$ is fixed, for all large $n$ the sets $(n/(n-r))K$, $0\le r\le L$, are contained in one compact $K'\Subset\Omega$. The nearest-neighbor ratio asymptotics quoted above, applied at total degree $N=n-r$, gives locally uniform limits on $K'$ for
\[
N\frac{P_{\boldsymbol k-\mathbf e_j}(N\xi)}{P_{\boldsymbol k}(N\xi)}.
\]
Putting $\xi=(n/N)\zeta$ gives
\[
n\frac{P_{\boldsymbol k-\mathbf e_j}(n\zeta)}{P_{\boldsymbol k}(n\zeta)}
=\frac{n}{N}\left[
N\frac{P_{\boldsymbol k-\mathbf e_j}(N\xi)}{P_{\boldsymbol k}(N\xi)}
\right]_{\xi=(n/N)\zeta},
\]
so replacing the scale $n-r$ by $n$ changes each step by $o(1)$ locally uniformly on $K$. The same zero-free denominator statement holds on $K$ for all sufficiently large indices. Multiplying the finitely many ratio limits and restoring the leading coefficients $(-m)^n$ and the factorial factor in $S_{n,\ell}$ gives a locally bounded normal family on $\Omega$. Every subsequence therefore has a further subsequence on which all $S_{n,\ell}$, $0\le\ell\le L$, converge locally uniformly. The limit for $\ell=0$ is $1$, and for $\ell\ge1$ the expansion at infinity
\[
S_{n,\ell}(\zeta)=(-m)^{-\ell}\zeta^{-\ell}+O(\zeta^{-\ell-1})
\]
uniformly for large $|\zeta|$ gives $H_\ell(\infty)=0$.

The limiting Cauchy transform is the total-degree-scaled Stieltjes transform obtained in \cite[\S6]{NVA}. In the present normalization it is
$$C_m(\zeta)=1-v(\zeta)^{-m},\qquad v^{m+1}=m\zeta(v-1),$$
with the branch determined by
$$v(\zeta)=1+\frac1{m\zeta}+O(\zeta^{-2})$$
at infinity. The branch points of $\zeta(v)=v^{m+1}/(m(v-1))$ occur at $v=0$ and $v=(m+1)/m$; the latter gives
$$c_m=\zeta\!\left(\frac{m+1}{m}\right)=\left(\frac{m+1}{m}\right)^{m+1}.$$
Thus the limiting cut is $[0,c_m]$.
\end{proof}

\begin{prop}[microscopic limit for the reduced local sequence]\label{prop:micro-higher-limit}
Let $m\ge1$ and $\alpha\in\mathbb Z$, and let $\Pi_n^{(\alpha,m)}$ be as in Proposition~\ref{prop:micro-higher}. For $n\ge1$, set
$$R_n(\zeta):=\Pi_n^{(\alpha,m)}(n\zeta),$$
and let $\nu_n$ be the normalized zero-counting measure of $R_n$. Then $\nu_n$ converges weakly on $\C$ to $\mu_m$, where $c_m:=\left(\frac{m+1}{m}\right)^{m+1}$ and $\mu_m$ is the probability measure supported on
$$[0,c_m]$$
whose Cauchy transform is
$$C_m(\zeta)=1-v(\zeta)^{-m},$$
with $v$ the branch, holomorphic on $\C\setminus[0,c_m]$, of
$$v^{m+1}=m\zeta(v-1)$$
satisfying $v(\zeta)=1+\frac{1}{m\zeta}+O(\zeta^{-2})$ as $\zeta\to\infty$. Equivalently, $\mu_m$ is absolutely continuous on $(0,c_m)$, and for
$$\zeta(\phi):=\frac{1}{m}\frac{\sin^{m+1}((m+1)\phi)}{\sin\phi\,\sin^m(m\phi)},\qquad 0<\phi<\frac{\pi}{m+1},$$
its density with respect to Lebesgue measure on $(0,c_m)$ satisfies
$$\frac{d\mu_m}{dx}\bigl(\zeta(\phi)\bigr)=\frac{\sin^{m+1}(m\phi)}{\pi\,\sin^m((m+1)\phi)}.$$
Its moments are
$$\int_0^{c_m} x^k\,\dd\mu_m(x)=\frac{1}{m^k(k+1)}\binom{(m+1)k}{k},\qquad k\ge0,$$
and $\mu_m(\{0\})=0$.
\end{prop}

\begin{proof}
First assume $\alpha<0$. By Proposition~\ref{prop:micro-mop-ensemble}, $(-m)^{-n}\Pi_n^{(\alpha,m)}$ is the monic staircase type-II multiple Laguerre polynomial, equivalently the average characteristic polynomial of the Laguerre Muttalib--Borodin ensemble with parameter $1/m$ and Laguerre exponent $-\alpha/m-1$. Lemma~\ref{lem:nva-staircase}, applied with $\alpha_0=\alpha$, gives the desired convergence when $\alpha<0$. In the present normalization the endpoint and parametrization are
$$c_m=\left(\frac{m+1}{m}\right)^{m+1},\qquad
\zeta(\phi)=\frac{1}{m}\frac{\sin^{m+1}((m+1)\phi)}{\sin\phi\,\sin^m(m\phi)},$$
and the Cauchy transform is the branch $C_m(\zeta)=1-v(\zeta)^{-m}$ determined by $v^{m+1}=m\zeta(v-1)$ and $C_m(\zeta)=\zeta^{-1}+O(\zeta^{-2})$ at infinity. The shifted-ratio compactness and zero-free denominator property used below are the second part of Lemma~\ref{lem:nva-staircase}.

It remains to remove the restriction $\alpha<0$. Choose $L\in\mathbb N$ with $\alpha_0:=\alpha-L<0$. Comparing generating functions gives, for fixed $L$,
$$
\Pi_n^{(\alpha,m)}(x)=
\sum_{\ell=0}^{L}(-1)^\ell\binom{L}{\ell}
\frac{n!}{(n-\ell)!}\,\Pi_{n-\ell}^{(\alpha_0,m)}(x)
\qquad(n\ge L).
$$
Set $R_n(\zeta):=\Pi_n^{(\alpha_0,m)}(n\zeta)$, $Q_n(\zeta):=\Pi_n^{(\alpha,m)}(n\zeta)$, and $\Omega:=\widehat\C\setminus[0,c_m]$. Let $n_j$ be an arbitrary subsequence. By Lemma~\ref{lem:nva-staircase} applied with this $\alpha_0$, after passing to a further subsequence and relabelling, $S_{n_j,\ell}\to H_\ell$ locally uniformly on $\Omega$ for $0\le\ell\le L$, and $R_{n_j}$ has no zeros on compact subsets of $\Omega$ for all large $j$. Hence
$$
\frac{Q_{n_j}}{R_{n_j}}\longrightarrow
F:=\sum_{\ell=0}^L(-1)^\ell\binom{L}{\ell}H_\ell
$$
locally uniformly on $\Omega$. Since $F(\infty)=1$, $F$ is not identically zero.

Let $K\Subset\Omega$. Cover the finitely many zeros of $F$ in $K$ by disjoint discs whose boundaries contain no zero of $F$. Since $R_{n_j}$ is zero-free on $K$ for all large $j$ and $Q_{n_j}/R_{n_j}\to F$ locally uniformly, Hurwitz's theorem gives no zeros of $Q_{n_j}$ on the complement of these discs and only $O_K(1)$ zeros inside them. Applying this with $K=\widehat\C\setminus U$, where $U$ is any open neighborhood of $[0,c_m]$, shows that the number of zeros of $Q_{n_j}$ in $\widehat\C\setminus U$ is $O_U(1)$. Thus every subsequential normalized zero-counting limit of $Q_{n_j}$ is supported on $[0,c_m]$.

On $\Omega\setminus F^{-1}(0)$,
$$
\frac1{n_j}\frac{Q_{n_j}'}{Q_{n_j}}-\frac1{n_j}\frac{R_{n_j}'}{R_{n_j}}
=\frac1{n_j}\frac{(Q_{n_j}/R_{n_j})'}{Q_{n_j}/R_{n_j}}\longrightarrow0
$$
locally uniformly. The same convergence holds distributionally on all of $\Omega$: on a compact set, the zeros of $Q_{n_j}/R_{n_j}$ lying over the zeros of $F$ have total multiplicity $O_K(1)$, so their normalized divisor contributes $O_K(n_j^{-1})$ to every compactly supported test. Thus every subsequential zero-measure limit of $Q_{n_j}$ has the same Cauchy transform on $\Omega$ as the limit for $R_{n_j}$, namely $C_m$. Since such a limit is a probability measure supported on $[0,c_m]$, it is determined by its Cauchy transform on $\Omega$. Every subsequence has a further subsequence with this limit, so the normalized zero-counting measures for $Q_n$ converge to the same limit as in the case $\alpha_0<0$.

The branch $v$ is fixed by $v(\zeta)=1+\frac1{m\zeta}+O(\zeta^{-2})$, which gives
$$C_m(\zeta)=1-v(\zeta)^{-m}=\frac1\zeta+O(\zeta^{-2})$$
at infinity and hence total mass one. To compute the moments, put $y=\zeta^{-1}$ and $w=v-1$. Then
$$w=\frac{y}{m}(1+w)^{m+1},\qquad C_m(1/y)=1-(1+w)^{-m}.$$
Lagrange inversion gives, for $k\ge0$,
$$[y^{k+1}]C_m(1/y)=\frac{1}{k+1}[u^k]m(1+u)^{-m-1}\left(\frac{(1+u)^{m+1}}{m}\right)^{k+1}
=\frac{1}{(k+1)m^k}\binom{(m+1)k}{k},$$
which proves the moment formula. The boundary parametrization
$$v_+(\phi)=\frac{\sin((m+1)\phi)}{\sin(m\phi)}e^{-i\phi},\qquad 0<\phi<\frac{\pi}{m+1},$$
gives $\zeta=\zeta(\phi)$ on the two-sided cut and
$$
\Im C_{m,+}(\zeta(\phi))
=\Im\bigl(1-v_+(\phi)^{-m}\bigr)
=-\frac{\sin^{m+1}(m\phi)}{\sin^m((m+1)\phi)}.
$$
Since $C_m(z)=\int(z-x)^{-1}\,\dd\mu_m(x)$, Stieltjes inversion gives the displayed density on $(0,c_m)$. As $\phi\uparrow\pi/(m+1)$ one has $\zeta(\phi)\asymp(\pi/(m+1)-\phi)^{m+1}$, so the density is $O(\zeta^{-m/(m+1)})$ and is integrable at the origin. To exclude an atom at $0$, the algebraic relation gives $|v(-y)|\asymp y^{1/(m+1)}$ as $y\downarrow0$, so $C_m(-y)=O(y^{-m/(m+1)})$. For any finite measure $\nu$ on $[0,c_m]$,
$$
-yC_\nu(-y)=\int_{[0,c_m]}\frac{y}{y+x}\,\dd\nu(x)\longrightarrow\nu(\{0\})
$$
by dominated convergence. Applying this to $\mu_m$ gives $\mu_m(\{0\})=0$.
\end{proof}

\begin{remark}[Marchenko--Pastur law at $m=1$]\label{rem:mp-from-sheffer}
At $m=1$, the algebraic relation $v^2=\zeta(v-1)$ has critical value $c_1=4$, and $\mu_1$ has Lebesgue density
$$\frac{d\mu_1}{dx}(x)=\frac{1}{2\pi x}\sqrt{x(4-x)},\qquad 0<x<4,$$
the standard Marchenko--Pastur law \cite{MP}. Since $\Pi_n^{(\alpha,1)}(x)=n!\,L_n^{(-\alpha-1)}(x)$, the $m=1$ case of Proposition~\ref{prop:micro-higher-limit} recovers the classical Laguerre zero distribution.
\end{remark}

\begin{cor}[microscopic zero law for the higher-order local model]\label{cor:micro-higher-zero}
Assume $m\ge2$ and keep the notation of Propositions~\ref{prop:micro-higher} and~\ref{prop:micro-higher-limit}. Let $N_n$ be the number of nonzero zeros of $\Pi_n^{(\alpha,m)}$, counted with multiplicity. If $\alpha<0$, then $N_n=n$ for every $n\ge0$. If $\alpha\ge0$ and $\alpha=qm+s$ with $q\in\mathbb Z_{\ge0}$ and $0\le s<m$, then
$$\ord_0\Pi_n^{(\alpha,m)}=q+1,\qquad N_n=n-q-1,\qquad n\ge\alpha+1.$$
In particular $N_n/n\to1$, and $N_n>0$ for all sufficiently large $n$. Fix $\eta$ with $\eta^m=-\lambda$, and for $0\le\nu\le m-1$ define
$$\phi_\nu:[0,c_m]\to\widehat\C,\qquad
\phi_\nu(0)=\infty,\qquad \phi_\nu(x)=\varepsilon_\nu\eta\,x^{-1/m}\ (x>0),\qquad \varepsilon_\nu:=e^{2\pi i\nu/m},$$
where $x^{1/m}$ denotes the positive real root for $x>0$. Set
$$\widehat\mu_{m,\lambda}:=\frac1m\sum_{\nu=0}^{m-1}(\phi_\nu)_\#\mu_m.$$
This measure is independent of the chosen root $\eta$. Each nonzero zero of $\Pi_n^{(\alpha,m)}$ gives exactly $m$ finite zeros of $g_{\alpha,m}^{(n)}$, counted with the same multiplicity. Thus, for all sufficiently large $n$,
$$\widehat\nu_n:=\frac{1}{mN_n}\sum_{\substack{g_{\alpha,m}^{(n)}(z)=0\\ z\ne a}}
\operatorname{mult}_z\!\bigl(g_{\alpha,m}^{(n)}\bigr)\,
\delta_{\,n^{1/m}(z-a)}$$
is a probability measure. Then
$$\widehat\nu_n\xrightarrow{\;w\;}\widehat\mu_{m,\lambda}$$
weakly on $\widehat\C$. Since $\widehat\mu_{m,\lambda}(\{\infty\})=0$, the sequence is tight in $\C$; hence the same convergence holds weakly on $\C$.
\end{cor}

\begin{proof}
If $\alpha<0$, then the constant term in the explicit formula of Proposition~\ref{prop:micro-higher} is $(-\alpha)_n\ne0$. Hence $N_n=n$ for all $n$.

Assume $\alpha\ge0$ and write $\alpha=mq+s$ with $q\ge0$ and $0\le s<m$. Fix $n\ge\alpha+1$. In the explicit formula
$$\Pi_n^{(\alpha,m)}(x)=\sum_{k=0}^n\frac{x^k}{k!}\sum_{j=0}^k(-1)^j\binom{k}{j}(mj-\alpha)_n,$$
the factor $(mj-\alpha)_n$ vanishes for every $j=0,\dots,q$, because $\alpha-mj\in\{0,\dots,\alpha\}\subset\{0,\dots,n-1\}$. Therefore the coefficients of $x^k$ vanish for $0\le k\le q$. For $k=q+1$, all terms with $j\le q$ vanish and the remaining term is
$$\frac{(-1)^{q+1}}{(q+1)!}(m-s)_n\ne0.$$
Hence $\ord_0\Pi_n^{(\alpha,m)}=q+1$ and $N_n=n-q-1$ for every $n\ge\alpha+1$. In particular $N_n/n\to1$. Changing $\eta$ to $\varepsilon_j\eta$ merely cyclically permutes the maps $\phi_\nu$, so $\widehat\mu_{m,\lambda}$ is independent of the chosen root.

Let \(\nu_n\) denote the normalized zero-counting measure of \(\Pi_n^{(\alpha,m)}(n\cdot)\). For every continuous function $\varphi:\widehat\C\to\mathbb C$, define a continuous function $F_\varphi:\widehat\C\to\mathbb C$ by
$$F_\varphi(0):=\varphi(\infty),\qquad
F_\varphi(\infty):=\varphi(0),\qquad
F_\varphi(x):=\frac1m\sum_{w^m=-\lambda/x}\varphi(w)\quad(x\in\C^\times).$$
The sum over the $m$ roots of $w^m=-\lambda/x$ is symmetric, hence independent of their ordering; continuity at $0$ and $\infty$ is immediate because those roots tend respectively to $\infty$ and $0$. By the definition of $\widehat\mu_{m,\lambda}$,
$$\int F_\varphi\,\dd\mu_m=\int\varphi\,\dd\widehat\mu_{m,\lambda}.$$
Moreover,
$$\int F_\varphi\,\dd\nu_n
=\frac{N_n}{n}\int\varphi\,\dd\widehat\nu_n+\frac{n-N_n}{n}\varphi(\infty).$$
Since $\nu_n\xrightarrow{\;w\;}\mu_m$, $N_n/n\to1$, and $(n-N_n)/n\to0$, it follows that $\widehat\nu_n\xrightarrow{\;w\;}\widehat\mu_{m,\lambda}$ on $\widehat\C$. Because $\mu_m(\{0\})=0$, the limit assigns zero mass to $\{\infty\}$. The weak convergence on $\widehat\C$ therefore gives tightness in $\C$; combined with convergence against compactly supported continuous test functions, this is weak convergence on $\C$.
\end{proof}

Figure~\ref{fig:microscopic-root-pushforward} shows the reciprocal-root map from the scaled polynomial zero variable to the microscopic $u$-plane.

\begin{center}
\begin{minipage}{\linewidth}
\centering
\begin{tikzpicture}[
    x=1cm,
    y=1cm,
    line cap=round,
    line join=round,
    every node/.style={font=\small}
]
% Left panel: scaled zero variable.
\begin{scope}[shift={(-3.80,0)}]
    \draw[->,black!55] (-0.25,0) -- (4.25,0) node[right] {$\zeta$};
    \draw[line width=2.6pt,black!22] (0,0) -- (3.35,0);
    \draw[line width=1.15pt,black!75] (0,0) -- (3.35,0);
    \foreach \x/\lab in {0/0,3.35/c_m} {
        \draw[black!70] (\x,0.08) -- (\x,-0.08);
        \node[below] at (\x,-0.08) {$\lab$};
    }
    \fill[black!75] (0,0) circle (1.15pt);
    \fill[black!75] (3.35,0) circle (1.15pt);
    \node[align=center] at (1.68,0.68) {$\operatorname{supp}\mu_m=[0,c_m]$};
    \node[align=center,black!65] at (1.68,-0.72)
        {scaled zero variable $\zeta=X/n$};
\end{scope}

% The reciprocal-root map.
\draw[->,line width=0.85pt,black!55] (0.50,0.38)
    .. controls (1.24,1.08) and (2.24,1.08) .. (2.98,0.38);
\node[align=center] at (1.74,1.50)
    {$\zeta=-\lambda/u^m$\\[-1pt]
     $u=\varepsilon_\nu\eta\,\zeta^{-1/m},\quad \eta^m=-\lambda$};
\node[align=center,black!65] at (1.74,-0.52)
    {$\zeta\downarrow0\Rightarrow |u|\to\infty$};

% Right panel: image in the microscopic u-plane.
\begin{scope}[shift={(4.70,0)}]
    \draw[->,black!38] (-2.10,0) -- (2.25,0) node[below right] {$\Re u$};
    \draw[->,black!38] (0,-1.83) -- (0,1.93) node[above left] {$\Im u$};

    % Guide circle: endpoint \zeta=c_m maps to this radius.
    \draw[black!25,densely dashed] (0,0) circle (0.68);
    \node[anchor=west,black!60] at (0.78,-0.78)
        {$r_0=|\eta|c_m^{-1/m}$};

    % Three rays are drawn for the illustrative case m=3, after rotation so eta>0.
    \foreach \ang in {0,120,240} {
        \draw[line width=4.2pt,blue!35,opacity=0.65] (\ang:0.68) -- (\ang:1.72);
        \draw[->,line width=1.0pt,blue!70!black] (\ang:0.68) -- (\ang:1.98);
        \fill[blue!70!black] (\ang:0.68) circle (1.3pt);
    }

    \node[blue!70!black,anchor=west] at (2.02,0.24) {$\nu=0$};
    \node[blue!70!black,anchor=south east] at (-1.10,1.58) {$\nu=1$};
    \node[blue!70!black,anchor=north east] at (-1.10,-1.58) {$\nu=2$};
    \fill[black] (0,0) circle (1.15pt) node[below left] {$0$};
    \node[anchor=west] at (0.58,1.82) {$u=n^{1/m}(z-a)$};
    \node[align=center,black!65] at (0,-2.12)
        {shown for $m=3$, after rotation};
\end{scope}
\end{tikzpicture}
\captionsetup{type=figure,hypcap=false}
\caption{The reciprocal-root map in Corollary~\ref{cor:micro-higher-zero}. Write $\zeta=X/n$ for the scaled zero variable of $\Pi_n^{(\alpha,m)}(n\cdot)$. In the microscopic coordinate $u=n^{1/m}(z-a)$ one has $\zeta=-\lambda/u^m$, equivalently $u=\varepsilon_\nu\eta\,\zeta^{-1/m}$ with $\eta^m=-\lambda$ and $\varepsilon_\nu=e^{2\pi i\nu/m}$. Thus the support interval $[0,c_m]$ is sent to $m$ exterior radial rays in the finite $u$-plane: the endpoint $c_m$ gives the inner radius $r_0=|\eta|c_m^{-1/m}=|\lambda|^{1/m}c_m^{-1/m}$, while $\zeta=0$ corresponds to infinity, so increasing $\zeta$ moves inward along each ray. The right panel is drawn for $m=3$, after rotating so that $\eta$ is positive real; the dashed circle marks this inner boundary, not a density contour.}
\label{fig:microscopic-root-pushforward}
\end{minipage}
\end{center}

For $m=1$, Proposition~\ref{prop:micro-higher} gives
$$\Pi_n^{(\alpha,1)}(x)=n!L_n^{(-\alpha-1)}(x),$$
so the induced microscopic zero law in the original $z$-variable is as follows.

\begin{cor}[Marchenko--Pastur scaling]\label{cor:mp}
Let $a\in\C$, $\alpha\in\mathbb Z$, and $\lambda\in\C^\times$, and set
$$g_\alpha(z):=(z-a)^\alpha \exp\!\left(\frac{\lambda}{z-a}\right).$$
Let $N_n$ be the number of finite zeros of $g_\alpha^{(n)}$, counted with multiplicity. Then $N_n=n$ if $\alpha<0$, while $N_n=n-\alpha-1$ for $\alpha\ge0$ and $n\ge\alpha+1$. In particular, $N_n/n\to1$. For every sufficiently large $n$, let $x_{1,n},\dots,x_{N_n,n}$ be the nonzero zeros of $\Pi_n^{(\alpha,1)}(x)=n!L_n^{(-\alpha-1)}(x)$, counted with multiplicity. Then
$$z_{k,n}=a-\frac{\lambda}{x_{k,n}},\qquad k=1,\dots,N_n,$$
are the finite zeros of $g_\alpha^{(n)}$, and
$$\frac{1}{N_n}\sum_{k=1}^{N_n}\delta_{x_{k,n}/N_n}\xrightarrow{\;w\;}\mu_{\mathrm{MP}},$$
where $\mu_{\mathrm{MP}}$ is the Marchenko--Pastur law on $[0,4]$. Set \(\widehat\mu_{1,\lambda}:=(-\lambda/x)_\#\mu_{\mathrm{MP}}\), where \(x\mapsto-\lambda/x\) is extended by \(0\mapsto\infty\). Consequently,
$$\frac{1}{N_n}\sum_{k=1}^{N_n}\delta_{\,N_n(z_{k,n}-a)}\xrightarrow{\;w\;}\widehat\mu_{1,\lambda}$$
weakly on $\widehat\C$. Since $\widehat\mu_{1,\lambda}$ has no mass at $\infty$, the sequence is tight in $\C$, so the convergence is also weak on $\C$. Replacing $N_n$ by $n$ in the scaling variables, and also in the normalizing prefactor if desired, does not change the limit because $N_n/n\to1$.
\end{cor}

\begin{proof}
By Proposition~\ref{prop:micro-higher} with $m=1$,
$$g_\alpha^{(n)}(z)=(-1)^n(z-a)^{\alpha-n}\Pi_n^{(\alpha,1)}\!\left(-\frac{\lambda}{z-a}\right)\exp\!\left(\frac{\lambda}{z-a}\right),$$
so the finite zeros of $g_\alpha^{(n)}$ are obtained from the nonzero zeros of $\Pi_n^{(\alpha,1)}$ by $z=a-\lambda/x$.
If $\alpha<0$, then the constant term in the explicit formula for $\Pi_n^{(\alpha,1)}$ is $(-\alpha)_n\ne0$, so $N_n=n$. If $\alpha\ge0$ and $\varkappa:=\alpha+1$, then
$$\Pi_n^{(\alpha,1)}(x)=n!L_n^{(-\varkappa)}(x)=(-x)^\varkappa (n-\varkappa)!\,L_{n-\varkappa}^{(\varkappa)}(x)\qquad(n\ge\varkappa),$$
hence $\ord_0\Pi_n^{(\alpha,1)}=\varkappa$ and $N_n=n-\varkappa$ for $n\ge\varkappa$. Thus $N_n/n\to1$.

Proposition~\ref{prop:micro-higher-limit} with $m=1$ and Remark~\ref{rem:mp-from-sheffer} give the Marchenko--Pastur limit for the rescaled zero-counting measures of $\Pi_n^{(\alpha,1)}(n\zeta)$. When $\alpha\ge0$, deleting the fixed zero at $0$ and renormalizing by $N_n$ changes the integral against any bounded continuous test function by $O(n^{-1})$, so the same limit holds for the nonzero zeros. Since $N_n/n\to1$, replacing $n$ by $N_n$ in either the $x$- or the $z$-scaling does not change the weak limit. Applying the continuous mapping theorem on $\widehat\C$ to the extension of $x\mapsto-\lambda/x$ with $0\mapsto\infty$ gives the spherical convergence, because $\mu_{\mathrm{MP}}(\{0\})=0$. The limit has no mass at $\infty$, so spherical convergence implies tightness in $\C$ and therefore weak convergence on $\C$; replacing $N_n$ by $n$ in the scaling variables or in the normalizing prefactor changes integrals by $o(1)$.
\end{proof}

For a meromorphic function $g$ that is nonzero on a punctured neighborhood in the local coordinate $x$, we call
$$d-d\log g$$
the \emph{scalar connection trivialized by $g$}. Two such local connections are \emph{holomorphically gauge equivalent at $x=\infty$} if their connection forms differ by $(\log h)'(x)\,dx$ for some holomorphic function $h$ with $h(\infty)\ne0$.

\begin{prop}[rigidity of the Laguerre reduction]\label{prop:rigid-laguerre}
Fix $c_j\in \Z(T)$, set $\alpha_j:=p_j-\nu_j$, and let $E(z)=\sum_{s=1}^{m_j}\lambda_{j,s}(z-c_j)^{-s}+E_j^{\mathrm{reg}}(z)$ with $\lambda_{j,m_j}\ne0$. In the coordinate $x=(z-c_j)^{-1}$, the scalar connection trivialized by $f$ becomes, after holomorphic gauge equivalence at $x=\infty$,
$$d-\left(Q_j'(x)-\frac{\alpha_j}{x}\right)dx,\qquad Q_j(x):=\sum_{s=1}^{m_j}\lambda_{j,s}x^s.$$
Since a holomorphic gauge at $x=\infty$ can only modify the connection coefficient by $O(x^{-2})$ terms, the polynomial part $Q_j'(x)$ of the coefficient is invariant; hence its degree $m_j-1$, equivalently the degree $m_j$ of the irregular type $Q_j$, is invariant. In particular, after a linear rescaling of $x$, the local connection is holomorphically equivalent to
$$d-\left(\lambda-\frac{\alpha_j}{x}\right)dx,\qquad \lambda\ne0,$$
if and only if $m_j=1$.
\end{prop}

\begin{proof}
By Proposition~\ref{prop:local} with $n=0$, $f(z)=(z-c_j)^{\alpha_j}\phi_{j,0}(z)e^{E(z)}$ near $c_j$ with $\phi_{j,0}(c_j)\ne0$. Setting $x=(z-c_j)^{-1}$ gives
$$f(z)=x^{-\alpha_j}e^{Q_j(x)}h_j(x),\qquad
h_j(x):=\phi_{j,0}(c_j+x^{-1})e^{E_j^{\mathrm{reg}}(c_j+x^{-1})},$$
where $h_j$ is holomorphic near $x=\infty$ and $h_j(\infty)\ne0$. The connection coefficient of $d-d\log f$ is
$$Q_j'(x)-\frac{\alpha_j}{x}+\frac{h_j'(x)}{h_j(x)},$$
which differs from $Q_j'(x)-\alpha_j/x$ by the logarithmic derivative of a holomorphic nonvanishing gauge. Any further holomorphic gauge $g$ near $x=\infty$ with $g(\infty)\ne0$ satisfies $(\log g)'(x)=O(x^{-2})$, so the polynomial part $Q_j'(x)$ is invariant, and hence so is $\deg Q_j=m_j$.
\end{proof}

\subsection{Microscopic law for the full local factor}\label{sec:full-local-factor}

We now pass from the highest-term reduction to the full local factor. For $m\ge1$, put
$$c_m:=\left(\frac{m+1}{m}\right)^{m+1}.$$
Let $\mu_m$ be the probability measure on $[0,c_m]$ characterized in Proposition~\ref{prop:micro-higher-limit}; equivalently, its Cauchy transform is
$$C_m(\zeta)=1-v(\zeta)^{-m},\qquad v^{m+1}=m\zeta(v-1),\qquad v(\zeta)=1+\frac1{m\zeta}+O(\zeta^{-2})$$
as $\zeta\to\infty$. By Proposition~\ref{prop:micro-higher-limit}, $\mu_m(\{0\})=0$. For $\lambda\in\C^\times$, choose $\eta$ with $\eta^m=-\lambda$ and define continuous maps $\phi_\nu:[0,c_m]\to\widehat\C$ by
$$
\phi_\nu(0)=\infty,
\qquad
\phi_\nu(x)=e^{2\pi i\nu/m}\eta\,x^{-1/m}\quad(0<x\le c_m),
\qquad 0\le\nu\le m-1,
$$
where $x^{-1/m}$ is the positive real branch. Set
$$
\widehat\mu_{m,\lambda}:=\frac1m\sum_{\nu=0}^{m-1}(\phi_\nu)_\#\mu_m.
$$
Since $\mu_m(\{0\})=0$, the measure $\widehat\mu_{m,\lambda}$ is a probability measure on $\widehat\C$ with no atom at $\infty$. Changing $\eta$ cyclically permutes the $m$ summands. In particular,
\[
\operatorname{supp}\widehat\mu_{m,\lambda}
\subset
\{\infty\}\cup
\bigcup_{\nu=0}^{m-1}e^{2\pi i\nu/m}\eta\,[c_m^{-1/m},\infty),
\]
so every finite point in the support satisfies
\[
        |u|\ge|\eta|\,c_m^{-1/m}=|\lambda|^{1/m}\left(\frac{m}{m+1}\right)^{(m+1)/m}.
\]
For $m=1$ this recovers, after the inversion $x\mapsto-\lambda/x$, the reciprocal Marchenko--Pastur law in Corollary~\ref{cor:mp}. For $m\ge2$, Proposition~\ref{prop:micro-mop-ensemble} identifies $\mu_m$ (for negative exponent) as the total-degree-scaled Laguerre Muttalib--Borodin average-characteristic-polynomial zero law; Proposition~\ref{prop:micro-higher-limit} then removes the sign restriction on the integral exponent. Thus the microscopic essential-singularity measure is the equal-weight average of the reciprocal $m$th-root push-forwards of this multiple-Laguerre law.

\begin{lem}[compact Picard--Lefschetz deformation for the microscopic phase]\label{lem:micro-compact-pl}
Let $m\ge1$ and $\lambda\in\C^\times$. For $u\in\C^\times$ put
\[
\zeta(u):=-\lambda u^{-m},\qquad
\Phi_u(t):=\zeta(u)(1-(1-t)^{-m})-\log t,
\]
where \(\log t\) is read on the logarithmic cut surface specified below. Let
\[
\mathcal D_\lambda:=\{u\in\C^\times:\ -\lambda u^{-m}=c_m\},
\qquad c_m=\left(\frac{m+1}{m}\right)^{m+1}.
\]
Let $\mathcal S_\lambda$ be the union, in $\C^\times\setminus\mathcal D_\lambda$, of the Stokes arcs on which two lifted critical values of $\Phi_u$ have equal imaginary parts. Let $V$ be a connected component of $\C^\times\setminus(\mathcal D_\lambda\cup\mathcal S_\lambda)$, and let $K_0\Subset V$. Then there are finitely many simply connected open sets $U\Subset V$ covering $K_0$ with the following data.

For each such $U$, the critical points $t_0(u),\ldots,t_m(u)$ are holomorphic and simple on $U$. There are pairwise disjoint closed discs \(D_0,D_1,D_\infty\) about \(0,1,\infty\), independent of \(u\in U\), and a fixed Jordan cut \(\mathfrak c\) joining \(\partial D_0\) to \(\partial D_\infty\), disjoint from \(D_1\) and from the critical graphs, such that
\[
X_U:=\overline{\mathbb P^1\setminus(D_0\cup D_1\cup D_\infty\cup\mathfrak c)}
\]
is a compact bordered surface on which a branch of \(\log t\) is fixed. The coefficient core \(\gamma\) is the lift of a positively oriented small circle about \(0\), cut at its intersection with \(\mathfrak c\). The two endpoints of this lift are paired by the deck translation \(\log t\mapsto\log t+2\pi i\). For every one-form used below, the boundary values on the two sides of \(\mathfrak c\) are paired by this deck translation; in particular \(e^{n\Phi_u}\) is paired because \(n\in\mathbb Z\). Thus the integral over the closed coefficient circle equals the integral over the lifted relative core.

For each \(u\in U\) there is a locally trivial boundary split \(A_u\cup B_u=\partial X_U\), with the paired cut sides included in the same boundary subcomplex, such that, away from the paired cut sides, the downward field of \(h_u:=\Re\Phi_u\) points outward on the relative interiors of the \(B_u\)-arcs and inward on the relative interiors of the \(A_u\)-arcs, with a margin uniform on compact subsets of \(U\). Downward separatrices from critical points land on \(B_u\), upward separatrices land on \(A_u\), all landings are transverse, and there are no saddle connections.

The downward thimbles $\Gamma_{\nu,u}\in H_1(X_U,B_u;\mathbb Z)$ and upward thimbles $\Gamma^{\ast}_{\nu,u}\in H_1(X_U,A_u;\mathbb Z)$ form dual integral bases. Thus
\[
[\gamma]=\sum_{\nu=0}^m\sigma_{\nu,U}[\Gamma_{\nu,u}],
\qquad
\sigma_{\nu,U}=[\gamma]\cdot[\Gamma^{\ast}_{\nu,u}]\in\mathbb Z,
\]
and the integers are independent of $u\in U$. Since the coefficient core $[\gamma]$ is nonzero and the thimbles form a basis,
\[
I_U:=\{\nu:\sigma_{\nu,U}\ne0\}
\]
is nonempty; set
\[
M_U(u):=\max_{\nu\in I_U}\Re\Phi_u(t_\nu(u)).
\]
There are disjoint Morse neighborhoods of the active critical graphs and a constant $c_U>0$ such that the chain
\[
C_u:=\sum_{\nu\in I_U}\sigma_{\nu,U}\Gamma_{\nu,u}
\]
represents $[\gamma]$ modulo a boundary chain $E_u\subset B_u$ and a two-chain in $X_U$, has uniformly bounded length and multiplicity on compact subsets of $U$, contains no zero-coefficient thimble at chain level, and satisfies
\[
\operatorname{supp} C_u\setminus\bigcup_{\nu\in I_U}\mathcal N_\nu(u)
\subset\{\Re\Phi_u\le M_U(u)-c_U\},
\qquad
\operatorname{supp} E_u\subset\{\Re\Phi_u\le M_U(u)-c_U\}.
\]
Consequently, for every holomorphic one-form $\Omega$ on a neighborhood of $X_U$ whose boundary values agree on the paired cut sides,
\[
\int_\gamma\Omega=\int_{C_u}\Omega+\int_{E_u}\Omega .
\]
\end{lem}

\begin{proof}
The critical equation is
\[
(1-t)^{m+1}+m\zeta(u)t=0.
\]
Writing \(y=1-t\), a multiple root occurs exactly when
\[
y^{m+1}+m\zeta(1-y)=0,\qquad (m+1)y^m-m\zeta=0,
\]
hence \(y=(m+1)/m\) and \(\zeta=c_m\). Thus all critical points are simple on \(\C^\times\setminus\mathcal D_\lambda\). Since \(K_0\Subset V\), their lifted imaginary critical values are separated by a positive margin after \(U\) is chosen simply connected and small enough, and the critical graphs stay a positive distance from \(0\), \(1\), and \(\infty\).

Choose \(D_0,D_1,D_\infty\) and the cut \(\mathfrak c\) as in the statement, shrinking \(U\) if necessary. Choose numbers \(\alpha_-<\alpha_+\) such that, for all \(u\in\overline U\), the lifted critical values \(g_u(t_\nu(u))\) lie in \((\alpha_-,\alpha_+)\). After a harmless enlargement of this interval, choose the boundary circles so that the levels \(g_u=\alpha_\pm\) meet the boundary transversely and no boundary tangency of \(-\nabla h_u\) occurs on \(\partial X_U\cap g_u^{-1}([\alpha_-,\alpha_+])\).

The boundary signs are explicit. On \(|t|=\varepsilon\), with the outward normal of \(X_U\) pointing toward \(0\),
\[
h_u(t)=-\log|t|+O_U(\varepsilon),\qquad
\partial_{\mathbf n}h_u=\varepsilon^{-1}+O_U(1),
\]
so the downward field points inward there. Near \(t=1\), write \(s=1-t=re^{i\theta}\). Then
\[
h_u(t)=-r^{-m}\Re(\zeta(u)e^{-im\theta})+O_U(1),\qquad
\partial_{\mathbf n}h_u=-m r^{-m-1}\Re(\zeta(u)e^{-im\theta})+O_U(r^{-m}),
\]
where \(\mathbf n=-\partial_r\) is the outward normal of \(X_U\) on the boundary of the deleted disc about \(1\). At a boundary tangency near \(1\), the first displayed real part is \(O_U(r)\); hence \(|g_u(t)|\ge c_Ur^{-m}\) for \(r\) small, uniformly in \(u\in\overline U\). Shrinking \(D_1\) therefore removes all such tangencies from \(g_u^{-1}([\alpha_-,\alpha_+])\). On \(\partial D_1\cap g_u^{-1}([\alpha_-,\alpha_+])\) the outward part consists of exactly the \(m\) compact intervals contained in the sectors \(\Re(\zeta(u)e^{-im\theta})>0\), and their endpoints are transverse intersections with \(g_u=\alpha_\pm\). On the outer boundary \(|t|=R\),
\[
h_u(t)=\Re\zeta(u)-\log R+O_U(R^{-1}),\qquad
\partial_{\mathbf n}h_u=-R^{-1}+O_U(R^{-2}),
\]
so for \(R\) large the downward field points outward on the part of the outer boundary lying over \([\alpha_-,\alpha_+]\). This part is one compact interval on the cut surface, with the paired cut sides assigned to the same boundary subcomplex. Define \(B_u\) to be the union of these \(m\) intervals near \(1\) and this outer interval, and let \(A_u\) be the complementary closed boundary subcomplex. The remaining boundary arcs are entrance arcs on \(g_u^{-1}([\alpha_-,\alpha_+])\). The endpoint equations \(g_u=\alpha_\pm\) and the normal-sign inequalities are transverse, and compactness gives one positive sign margin, local triviality of the boundary triads, and, after shrinking \(D_1\) toward \(1\) and moving \(\partial D_\infty=\{|t|=R\}\) farther out if necessary, a number \(b_U>0\) such that
\[
\sup_{u\in\overline U}\sup_{t\in B_u}h_u(t)
\le
\inf_{u\in\overline U}\min_\nu h_u(t_\nu(u))-4b_U .
\]
For this last estimate, on the exit intervals near \(1\) the condition \(g_u(t)\in[\alpha_-,\alpha_+]\) forces \(\Im(\zeta(u)e^{-im\theta})=O_U(r^m)\); on the exit sectors \(\Re(\zeta(u)e^{-im\theta})\) is then bounded below by a positive constant, so \(h_u(t)\to-\infty\) uniformly as \(r\downarrow0\). On the outer exit interval \(h_u(t)=\Re\zeta(u)-\log R+O_U(R^{-1})\to-\infty\) uniformly as \(R\to\infty\).

Lemma~\ref{lem:gradient-landing}, applied on this cut surface and with the paired cut sides transported by the deck identification, gives finite-time landing of downward separatrices on \(B_u\), finite-time landing of upward separatrices on \(A_u\), absence of saddle connections, and isotopic continuation under \(u\)-variation. The equality of boundary values of the coefficient one-forms on the paired cut sides means that no extra boundary integral is created by the cut.

The topological calculation is explicit. The cut reduces the three-punctured sphere to an annular surface; \(B_u\) has \(m+1\) interval components, namely the \(m\) exit intervals at \(t=1\) and the outer exit interval, transported through the cut identification. The lifted coefficient core is nonzero in \(H_1(X_U,B_u;\mathbb Z)\), since its projection winds once around \(0\). Therefore
\[
\operatorname{rank}H_1(X_U,B_u;\mathbb Z)=m+1,
\qquad
\operatorname{rank}H_1(X_U,A_u;\mathbb Z)=m+1.
\]
After collar perturbing endpoints away from \(A_u\cap B_u\), a downward and an upward thimble through the same saddle meet once with the chosen orientation. If thimbles attached to distinct saddles met in the interior, the common point would lie on trajectories on which \(g_u=\Im\Phi_u\) is constant, forcing equality of the two lifted critical imaginary values; this is excluded on \(U\). For the two thimbles attached to the same saddle, strict monotonicity of \(h_u=\Re\Phi_u\) away from the saddle excludes any second interior intersection. Poincar\'e--Lefschetz duality for the boundary split gives a perfect integral pairing; the identity intersection matrix then implies that the downward thimbles and upward thimbles are dual integral bases. The integers \(\sigma_{\nu,U}\) are locally constant because all representatives are continued by the boundary-triad isotopy.

It remains to build the filtered representative. Let \(I_U=\{\nu:\sigma_{\nu,U}\ne0\}\) and set \(C_u=\sum_{\nu\in I_U}\sigma_{\nu,U}\Gamma_{\nu,u}\). In Morse neighborhoods of the critical graphs,
\[
\Phi_u(t)=\Phi_u(t_\nu(u))-v^2/2
\]
with uniform radii. Shrinking the neighborhoods once, every downward branch leaves its own neighborhood at height at most \(\Re\Phi_u(t_\nu(u))-2\varepsilon\), with \(\varepsilon>0\) uniform on compact subsets of \(U\). Since \(h_u\) decreases along the rest of the branch and the exit boundary is below the minimum critical height by a fixed amount, the stated height estimates hold after replacing \(\varepsilon\) by a smaller \(c_U\). The length bounds follow from the Morse-coordinate radii, the transverse landing flow boxes, and the estimate
\[
\operatorname{length}(\alpha)\le
\frac{\operatorname{osc}_{X_U}h_u}{\inf_\alpha|\nabla h_u|}
\]
on each compact flow-box complement; only finitely many branches occur and the integers \(\sigma_{\nu,U}\) are fixed.

Since \([C_u]=[\gamma]\) in \(H_1(X_U,B_u;\mathbb Z)\), the boundary of \(C_u\) has total coefficient zero on each interval component of \(B_u\). On each such component, order the boundary points and connect them by the cumulative-sum chain, as in Lemma~\ref{lem:filtered-wright}. This gives \(E_u\subset B_u\) with \(\partial E_u=-\partial C_u\), uniformly bounded length and multiplicity, and support below \(M_U-c_U\). The remaining difference \(\gamma-C_u-E_u\) bounds a two-chain in \(X_U\). Stokes' theorem, together with the paired equality of the cut-side boundary values, gives the integral identity.
\end{proof}

\begin{lem}[sublinear perturbations on fixed thimble decompositions]\label{lem:sublinear-cycle-perturb}
Assume the notation and conclusions of Lemma~\ref{lem:micro-compact-pl} on $U$, and let $0<\beta<1$. Let $X_U$ be the corresponding cut compact surface. For every compact $L\Subset U$, assume that there is a fixed neighborhood $\mathcal N_L$ of
\[
\{(u,t):u\in L,\ t\in X_U\}
\]
on which $R_n(u,t)$ and $a_n(u,t)$ are holomorphic, locally uniformly in $u$, and satisfy
\[
\sup_{\mathcal N_L}|R_n|=O_L(n^\beta),
\qquad
\sup_{\mathcal N_L}\log^+ |a_n|=O_L(\log n).
\]
After shrinking $\mathcal N_L$ if necessary, Cauchy's estimates then give $\partial_t^jR_n=O_{L,j}(n^\beta)$ for every fixed $j$. Assume that the one-forms
\[
a_n(u,t)\exp\!\left(n\Phi_u(t)+R_n(u,t)\right)\,\dd t
\]
have paired boundary values on the two sides of the cut in Lemma~\ref{lem:micro-compact-pl}. Suppose moreover that, on the active Morse neighborhoods over each compact $L\Subset U$, $a_n$ is nonvanishing and $\bigl|\log |a_n(u,t)|\bigr|=O_L(\log n)$. Define
\[
J_n(u):=\frac1{2\pi i}\int_\gamma a_n(u,t)
\exp\!\left(n\Phi_u(t)+R_n(u,t)\right)\,\dd t .
\]
Then for every compact $L\Subset U$,
\[
\limsup_{n\to\infty}\sup_{u\in L}
\left(\frac1n\log|J_n(u)|-M_U(u)\right)\le0.
\]
If, on $L$, a single active saddle $t_{\nu_\ast}$ satisfies
\[
\Re\Phi_u(t_{\nu_\ast}(u))\ge
\Re\Phi_u(t_\nu(u))+\delta
\qquad(\nu\in I_U\setminus\{\nu_\ast\})
\]
for some $\delta>0$, then
\[
\frac1n\log|J_n(u)|\longrightarrow M_U(u)
\]
uniformly on $L$. The same assertions hold for the unperturbed integral obtained by taking $R_n\equiv0$.
\end{lem}

\begin{proof}
Use the chain identity from Lemma~\ref{lem:micro-compact-pl}. On the boundary chain and on the nonlocal parts of the active thimbles one has
\[
\Re(n\Phi_u+R_n)\le nM_U(u)-c_Un+O(n^\beta),
\]
so these pieces are $O(\exp(nM_U-c n))$ after decreasing $c>0$, because $\beta<1$ and the lengths and multiplicities are uniformly bounded.

Near an active saddle, write a parameter-dependent Morse coordinate $v$ with
\[
\Phi_u(t)=\Phi_u(t_\nu(u))-v^2/2,
\qquad v(t_\nu(u))=0.
\]
After shrinking the Morse neighborhood, Cauchy's estimates applied to the assumed bounds for $R_n$ give $\partial_v^jR_n=O_L(n^\beta)$ for every fixed $j$. The perturbed critical equation
\[
-nv+\partial_vR_n(u,v)=0
\]
has a unique solution $v_{\nu,n}(u)=O_L(n^{\beta-1})$ by the implicit function theorem, and
\[
\partial_v^2\bigl(-nv^2/2+R_n(u,v)\bigr)\big|_{v=v_{\nu,n}}
=-n+O_L(n^\beta).
\]
The local piece of the thimble may be deformed, with endpoints kept on the boundary of the Morse neighborhood where $\Re\Phi_u$ is lower by a fixed amount, to the steepest segment through this perturbed critical point. If $b_n(u,v)=a_n(u,t(v))dt/dv$, then $b_n$ is nonvanishing there and $\bigl|\log|b_n|\bigr|=O_L(\log n)$. The one-dimensional saddle estimate, uniformly for $u\in L$, is
\[
\begin{aligned}
&\int b_n(u,v)
\exp\left(n\Phi_u(t_\nu(u))-nv^2/2+R_n(u,v)\right)\dd v \\
&\qquad=
\exp\left(n\Phi_u(t_\nu(u))+R_n(u,v_{\nu,n})\right)
 b_n(u,v_{\nu,n})
\left(\frac{2\pi}{n-\partial_v^2R_n(u,v_{\nu,n})}\right)^{1/2}
\left(1+O_L(n^{\beta-1})\right),
\end{aligned}
\]
with the square-root branch determined by the chosen thimble orientation. In particular each active local contribution has logarithmic size
\[
n\Re\Phi_u(t_\nu(u))+O_L(n^\beta)+O_L(\log n),
\]
and the upper bound follows after summing finitely many active saddles and adding the exponentially lower remainder.

Under the displayed strict dominance hypothesis, the $\nu_\ast$ contribution is larger than every other active contribution and the boundary remainder by $\exp(\delta n+O(n^\beta))$. Its prefactor is nonzero and has logarithm $O_L(\log n)$, so no cancellation is possible at exponential scale. Hence $n^{-1}\log|J_n(u)|\to M_U(u)$ uniformly on $L$. The case $R_n\equiv0$ is the same argument with the perturbation terms omitted.
\end{proof}

\begin{lem}[coalescing-saddle upper bound]\label{lem:micro-coalescing-upper}
Let $m\ge1$ and $\lambda\in\C^\times$. Put
\[
\zeta(u):=-\lambda u^{-m},\qquad
\Phi_u(t):=\zeta(u)(1-(1-t)^{-m})-\log t,
\qquad
c_m:=\left(\frac{m+1}{m}\right)^{m+1},
\]
where $\log t$ is read on a cut surface as in Lemma~\ref{lem:micro-compact-pl}. Define
\[
\mathcal D_\lambda:=\{u\in\C^\times:\ -\lambda u^{-m}=c_m\},
\]
and let $\mathcal S_\lambda$ denote the union, in $\C^\times\setminus\mathcal D_\lambda$, of the noncoalescing Stokes arcs on which two lifted critical values of $\Phi_u$ have equal imaginary parts. Let $u_0\in\mathcal D_\lambda$. There is a neighborhood $U_0$ of $u_0$ and a continuous subharmonic function $M$ on $U_0$ which agrees on each component of
\[
U_0\setminus(\mathcal D_\lambda\cup\mathcal S_\lambda)
\]
with the corresponding chamber envelope $M_U$ from Lemma~\ref{lem:micro-compact-pl}.

After shrinking $U_0$, choose one compact cut surface $X_0$ obtained from the $t$-sphere by deleting fixed disjoint discs about $0,1,\infty$ and one fixed cut, with the coalescing critical point and all nearby critical points in its interior; let $\gamma\subset X_0$ be the lifted coefficient core. Let $J_n$ be coefficient-circle integrals on $U_0$ of the form
\[
J_n(u)=\frac1{2\pi i}\int_\gamma a_n(u,t)
\exp\!\left(n\Phi_u(t)+R_n(u,t)\right)\,\dd t.
\]
Assume that one exponent $0<\beta<1$ works on $U_0$ in the following sense: for every compact $L\Subset U_0$, the functions $R_n$ and $a_n$ are holomorphic on a fixed neighborhood of
\[
\{(u,t):u\in L,\ t\in X_0\},
\]
their boundary values are paired on the two sides of the cut, and
\[
\sup_{u\in L,\,t\in X_0} |R_n(u,t)|=O_L(n^\beta),
\qquad
\sup_{u\in L,\,t\in X_0} \log^+|a_n(u,t)|=O_L(\log n).
\]
Then
\[
\limsup_{n\to\infty}\sup_{u\in L}
\left(\frac1n\log|J_n(u)|-M(u)\right)
\le0
\]
for every compact $L\Subset U_0$.
\end{lem}

\begin{proof}
At $u_0$ exactly two critical points merge. The map $u\mapsto\zeta=-\lambda u^{-m}$ is locally biholomorphic at $u_0$, because $u_0\ne0$ and $\zeta(u_0)=c_m\ne0$. In the parameter $\zeta$ the collision occurs at $\zeta=c_m$ and at the double critical point $t_0=-1/m$; the third $t$-derivative of $\Phi$ there is nonzero, and all other critical points remain simple after the parameter neighborhood is shrunk. Pass to the two-sheeted cover $s^2=\zeta-c_m$. On this cover the two merging critical points are holomorphic and are exchanged by $s\mapsto -s$. The Chester--Friedman--Ursell normal form, obtained here from Weierstrass preparation and the holomorphic Morse lemma with parameter, gives a holomorphic coordinate $x$ in a fixed neighborhood of $t_0$ such that
\[
\Phi_u(t)=\Phi_c(s)+\frac{x^3}{3}-\alpha(s)x,
\qquad
\alpha(s)=a s^2+O(s^4),\quad a\ne0,
\]
up to replacing $\Phi_c$ by another holomorphic function of $s$ only. After choosing the holomorphic square root of $\alpha(s)$ on the $s$-cover, the two critical points in the chart are $x=\pm\alpha(s)^{1/2}$, and their critical values are the two cubic saddle values.

On the $s$-cover define $M^\sharp$ near $s=0$ as the maximum of the finitely many height branches which occur in the min--max class of the coefficient contour: the noncoalescing active saddle heights and, in the cubic chart, the Airy saddle heights whose local thimbles have nonzero coefficient in the corresponding sector. On each sector of the lifted complement of the Stokes arcs this maximum is exactly the chamber envelope $M_U$ from Lemma~\ref{lem:micro-compact-pl}. The two Airy height branches coalesce at $s=0$, and the remaining branches stay separated; hence $M^\sharp$ extends continuously across $s=0$. Being the maximum of finitely many harmonic functions on the cover, $M^\sharp$ is subharmonic. The construction is invariant under $s\mapsto -s$, so on the punctured neighborhood $M^\sharp$ is the pullback of a continuous function $M$ of $\zeta-c_m=s^2$, hence of $u$. The descended function is subharmonic away from $u_0$ and locally bounded above at $u_0$; the removable-singularity theorem for subharmonic functions therefore makes $M$ subharmonic on $U_0$. Thus $M$ agrees with the chamber envelopes on the components of $U_0\setminus(\mathcal D_\lambda\cup\mathcal S_\lambda)$.

Outside the cubic chart the critical points remain simple and the compact Picard--Lefschetz construction of Lemma~\ref{lem:micro-compact-pl} gives representatives with a uniform height gap below this envelope. Inside the chart, the standard level geometry of $x^3/3-\alpha x$ gives a finite set of Airy descent arcs representing the restrictions of the same relative coefficient class from the adjacent noncoalescing sectors. The arcs can be chosen continuously on the $s$-cover; their union is invariant under the involution $s\mapsto -s$, so it descends to the original $u$-plane. Whenever the local representative has an Airy part in a fixed sector, let $I_{\rm Ai}\subset\{\pm1\}$ be the nonempty set of Airy saddle labels that occur in that part. Along the corresponding Airy arcs,
\[
\Re\left(\frac{x^3}{3}-\alpha x\right)
\le
\max_{\varepsilon\in I_{\rm Ai}}
\Re\left(\frac{(\varepsilon\alpha^{1/2})^3}{3}-\alpha(\varepsilon\alpha^{1/2})\right),
\]
and, because $\Phi_u(t)=\Phi_c(s)+x^3/3-\alpha x$, the right-hand side plus $\Re\Phi_c(s)$ is bounded by $M(u)$. The portions on the boundary of the chart are lower by a fixed amount after the chart is chosen small enough. With the perturbation $R_n$ included, Taylor's formula in the cubic coordinate gives the uniform bound
\[
\left|\int_{\text{cubic chart}} a_n(u,t)e^{n\Phi_u(t)+R_n(u,t)}\,dt\right|
\le C_L n^A\exp\left(nM(u)+C_L n^\beta\right)
\]
for some fixed $A$, because the chart has bounded length in Airy coordinates and the amplitude has only polynomial size. The same estimate, with an exponentially smaller right-hand side, holds outside the chart. Since $\beta<1$, division by $n$ gives the asserted upper bound.
\end{proof}

\begin{lem}[compact perturbation for the full local factor]\label{lem:full-local-perturb}
Let $m\ge1$, $\alpha\in\mathbb Z$, $\lambda_1,\ldots,\lambda_m\in\C$ with $\lambda_m\ne0$, and let $\phi$ be holomorphic and nonvanishing on a neighborhood of the closed disk $\{|z|\le r\}$. Define $H_n$ by
$$
F^{(n)}(z)=(-1)^n z^{\alpha-(m+1)n}
\exp\!\left(\sum_{s=1}^m\lambda_s z^{-s}\right)\phi(z)H_n(z),
$$
where
$$
F(z)=z^\alpha\exp\!\left(\sum_{s=1}^m\lambda_s z^{-s}\right)\phi(z)
$$
on $0<|z|<r$. Then $H_0\equiv1$, each $H_n$ is holomorphic on $|z|<r$, and
\[
H_{n+1}=-z^{m+1}H_n'
+\left(((m+1)n-\alpha)z^m+\sum_{s=1}^m s\lambda_s z^{m-s}-z^{m+1}\frac{\phi'}{\phi}\right)H_n,
\]
so $H_n(0)=(m\lambda_m)^n$. For $u\ne0$, $|u|<n^{1/m}r$, set
$$
Y_n(u):=\frac{(-z)^n}{n!}\frac{F^{(n)}(z)}{F(z)},\qquad z=n^{-1/m}u,
$$
and let $Y_n^{(0)}$ be obtained by putting $\lambda_1=\cdots=\lambda_{m-1}=0$ and $\phi\equiv1$. If $G_n(u):=H_n(n^{-1/m}u)$, then
\begin{equation}\label{eq:H-Y-relation}
G_n(u)=\frac{n!}{n^n}\,u^{mn}Y_n(u),
\qquad u\ne0.
\end{equation}
Moreover,
\begin{equation}\label{eq:full-local-perturb-L1}
\frac1n\log|Y_n|-\frac1n\log|Y_n^{(0)}|
\longrightarrow0
\qquad\text{in }L^1_{\mathrm{loc}}(\C^\times).
\end{equation}
\end{lem}

\begin{proof}
The recurrence follows by differentiating the defining relation for $H_n$ and comparing with the same relation for $n+1$; evaluation at $z=0$ gives $H_n(0)=(m\lambda_m)^n$. The relation \eqref{eq:H-Y-relation} is obtained by substituting $z=n^{-1/m}u$ in the definitions.

Taylor's formula for $F(z-zt)/F(z)$ gives, for $u\ne0$,
\begin{align}\label{eq:full-local-coeff}
Y_n(u)&=[t^n](1-t)^\alpha
\exp\!\left(n\zeta(u)(1-(1-t)^{-m})
+\sum_{s=1}^{m-1}n^{s/m}\lambda_su^{-s}\bigl((1-t)^{-s}-1\bigr)\right)  \\
&\hspace{35mm}\times
\frac{\phi(n^{-1/m}u(1-t))}{\phi(n^{-1/m}u)},
\qquad \zeta(u):=-\lambda_m u^{-m}.\nonumber
\end{align}
The reduced coefficient is
$$Y_n^{(0)}(u)=\frac1{n!}\Pi_n^{(\alpha,m)}\!\left(n\zeta(u)\right),$$
with $\Pi_n^{(\alpha,m)}$ from Proposition~\ref{prop:micro-higher}.

Let $K\Subset\C^\times$. In the Cauchy integral for \eqref{eq:full-local-coeff}, write
\[
\Phi_u(t):=\zeta(u)(1-(1-t)^{-m})-\log t,
\]
\[
R_n(u,t):=\sum_{s=1}^{m-1}n^{s/m}\lambda_su^{-s}\bigl((1-t)^{-s}-1\bigr),
\qquad
 a_n(u,t):=t^{-1}(1-t)^\alpha
\frac{\phi(n^{-1/m}u(1-t))}{\phi(n^{-1/m}u)}.
\]
Choose once and for all a number $\beta_0$ with $(m-1)/m<\beta_0<1$. The functions $R_n$ and $a_n$ are single-valued on the cut surface and have matching boundary values on the paired cut sides; since $e^{n\Phi_u}$ is invariant under $\log t\mapsto\log t+2\pi i$, the one-forms used in Lemma~\ref{lem:sublinear-cycle-perturb} have paired boundary values. On every compact surface used in Lemma~\ref{lem:micro-compact-pl}, after thickening the surface slightly and restricting $u$ to a compact set $K\Subset\C^\times$, the functions $R_n$ and $a_n$ are holomorphic on a fixed neighborhood and satisfy
\[
\sup |R_n|=O_K(n^{\beta_0}),
\qquad
\sup\log^+|a_n|=O_K(\log n).
\]
Cauchy's estimates give the corresponding bounds for the first two $t$-derivatives of $R_n$. In the active Morse neighborhoods, $a_n$ is nonvanishing and $|\log|a_n||=O_K(\log n)$. For the reduced coefficient one has $R_n\equiv0$ and $a_n=t^{-1}(1-t)^\alpha$, so the same hypotheses hold with the same $\beta_0$.

Put
\[
\mathcal D:=\{u\in\C^\times:\ -\lambda_m u^{-m}=c_m\},
\]
and let $\mathcal S$ be the union, in $\C^\times\setminus\mathcal D$, of the noncoalescing Stokes arcs on which two lifted critical values of $\Phi_u$ have equal imaginary parts. The chamber envelopes $M_U$ from Lemma~\ref{lem:micro-compact-pl} are restrictions of a single min--max envelope away from $\mathcal D$. Indeed, on a small parameter neighborhood $O\Subset\C^\times\setminus\mathcal D$ choose one cut surface and set
\[
M_O(u):=\inf_{[C]=[\gamma]}\ \sup_{\operatorname{supp}C}\Re\Phi_u ,
\]
where $C$ ranges over relative one-cycles on that surface in the coefficient-contour class. On $O\cap U$, with $U$ contained in a Stokes chamber, the filtered representative gives $M_O\le M_U$. Conversely, any representative of $[\gamma]$ has intersection number $\sigma_{\nu,U}$ with the upward thimble of every active saddle; when $\sigma_{\nu,U}\ne0$ it must meet that upward thimble, whose height is at least $\Re\Phi_u(t_\nu(u))$. Taking the maximum over active $\nu$ gives $M_O\ge M_U$. Hence the value depends only on the phase and the coefficient-contour class, not on the local Picard--Lefschetz chart. The functions $M_U$ therefore glue across ordinary overlaps and across noncoalescing Stokes arcs; denote the glued envelope by $M$. It is locally a maximum of finitely many harmonic critical-height branches and is continuous and subharmonic on $\C^\times\setminus\mathcal D$.

This min--max description also gives the upper estimate needed at noncoalescing Stokes points. For a fixed $u_0\notin\mathcal D$ and $\varepsilon>0$, choose a representative $C$ with
\[
\sup_{\operatorname{supp}C}\Re\Phi_{u_0}\le M(u_0)+\varepsilon .
\]
After shrinking the parameter neighborhood, the same relative class is represented by cycles with
\[
\sup_{\operatorname{supp}C_u}\Re\Phi_u\le M(u)+2\varepsilon .
\]
The amplitudes in the integrals for $Y_n$ and $Y_n^{(0)}$ have only polynomial growth, and the perturbation $R_n$ is $O_K(n^{\beta_0})$ on these fixed cycles. Thus, on every compact set $L\Subset\C^\times\setminus\mathcal D$,
\[
\limsup_{n\to\infty}\sup_{u\in L}\left(\frac1n\log|Y_n(u)|-M(u)\right)\le 2\varepsilon,
\qquad
\limsup_{n\to\infty}\sup_{u\in L}\left(\frac1n\log|Y_n^{(0)}(u)|-M(u)\right)\le 2\varepsilon .
\]
Since $\varepsilon$ is arbitrary, the upper bound required in Lemma~\ref{lem:upper-envelope-l1} holds away from $\mathcal D$. At each point of the finite collision set $\mathcal D$, Lemma~\ref{lem:micro-coalescing-upper}, applied with the exponent $\beta_0$, gives a local continuous subharmonic envelope. On every punctured overlap with a noncoalescing chamber, both this envelope and the previously defined $M$ are equal to the same min--max value of the coefficient-contour class; hence they agree. The local envelopes therefore glue to $M$ and extend it subharmonically across $\mathcal D$. For all sufficiently large $n$, any fixed parameter neighborhood with compact closure in $\C^\times$ lies in $\{|u|<n^{1/m}r\}$. On such a neighborhood the coefficient formula \eqref{eq:full-local-coeff} shows that $Y_n$ is holomorphic; it is not identically zero because \eqref{eq:H-Y-relation} and $G_n(0)=(m\lambda_m)^n$ imply $G_n\not\equiv0$. The same argument applies to $Y_n^{(0)}$ with the reduced data. Thus, on a neighborhood of $K$, both $Y_n$ and $Y_n^{(0)}$ satisfy the upper-bound and holomorphy hypotheses required in Lemma~\ref{lem:upper-envelope-l1} for a single continuous subharmonic envelope $M$.

It remains to identify a full-measure dense set on which both logarithmic sequences converge to $M$. On $\C^\times\setminus\mathcal D$ the critical points are locally holomorphic. If two distinct critical branches had critical values with difference identically real-constant or imaginary-constant, then the difference would be constant. Using $\zeta$ as a local parameter and writing $y=1-t$, one has
\[
\frac{d}{d\zeta}\Phi_\zeta(t(\zeta))=1-y(\zeta)^{-m}.
\]
Constancy of a difference gives $y_\mu=\omega y_\nu$ for some $\omega^m=1$; the critical equation $y^{m+1}+m\zeta(1-y)=0$ then forces $\omega=1$, so the branches coincide. Hence each pairwise equal-real-part set, and also each Stokes arc in $\mathcal S$, is a proper real-analytic subset of $\C^\times\setminus\mathcal D$ and has planar Lebesgue measure zero on compact subsets.

Let $\mathcal G$ be the complement in $\C^\times$ of $\mathcal D$, $\mathcal S$, and all pairwise equal-real-part sets. Then $\mathcal G$ is open, dense, and of full planar measure. On every compact $L\Subset\mathcal G$, after a finite cover by chamber subdomains, one active saddle has a positive dominance gap on each member of the cover. Lemma~\ref{lem:sublinear-cycle-perturb} gives
\[
\frac1n\log|Y_n|\to M,
\qquad
\frac1n\log|Y_n^{(0)}|\to M
\]
uniformly on those compact subsets. Lemma~\ref{lem:upper-envelope-l1}, applied separately to $Y_n$ and $Y_n^{(0)}$, yields convergence to $M$ in $L^1(K)$ for both sequences. Since $K\Subset\C^\times$ was arbitrary, \eqref{eq:full-local-perturb-L1} follows.
\end{proof}

\begin{lem}[no microscopic mass at the origin]\label{lem:full-local-no-origin}
Under the hypotheses and notation of Lemma~\ref{lem:full-local-perturb},
\begin{equation}\label{eq:full-local-no-origin-mass}
\lim_{\delta\downarrow0}\limsup_{n\to\infty}
\frac1{mn}\sum_{\substack{G_n(u)=0\\0<|u|<\delta}}
\operatorname{mult}_u(G_n)=0.
\end{equation}
\end{lem}

\begin{proof}
By Lemma~\ref{lem:full-local-perturb}, $G_n(0)=H_n(0)=(m\lambda_m)^n$. Fix $0<\delta<1$ so small that $\tau:=(m|\zeta(u)|)^{-1}<1/4$ when $|u|=\delta$. Cauchy's estimate in \eqref{eq:full-local-coeff} on $|t|=\tau$ gives, uniformly for $|u|=\delta$,
$$
\log|Y_n(u)|
\le n\log(m|\zeta(u)|)+n+O\!\left(\frac{n}{|\zeta(u)|}\right)+O_\delta\!\left(n^{(m-1)/m}\right).
$$
Indeed $1-(1-t)^{-m}=-mt+O(t^2)$ on this circle, and $((1-t)^{-s}-1)=O_s(t)$ for $s<m$. Combining this estimate with \eqref{eq:H-Y-relation} and Stirling's formula gives
$$
\sup_{|u|=\delta}\log|G_n(u)|
\le n\log|m\lambda_m|+O(n\delta^m)+O_\delta\!\left(n^{(m-1)/m}\right)+O(\log n).
$$
Choose $\delta$ so small that the preceding estimates also hold with $2\delta$ in place of $\delta$. Since $\log|G_n(0)|=n\log|m\lambda_m|$, Jensen's formula in $|u|<2\delta$ gives
$$
\sum_{\substack{G_n(u)=0\\0<|u|<\delta}}
\operatorname{mult}_u(G_n)
\le \frac{O(n\delta^m)+O_\delta(n^{(m-1)/m})+O(\log n)}{\log2}.
$$
Divide by $mn$, take $n\to\infty$, and then let $\delta\downarrow0$. This proves \eqref{eq:full-local-no-origin-mass}.
\end{proof}

\begin{theorem}[full local microscopic law]\label{thm:full-local-microscopic}
Let $a\in\C$, $m\ge1$, $\alpha\in\mathbb Z$, $\lambda_1,\ldots,\lambda_m\in\C$ with $\lambda_m\ne0$, and let $\phi$ be holomorphic and nonvanishing on a neighborhood of $a$. Set
$$
F(z):=(z-a)^\alpha
\exp\!\left(\sum_{s=1}^m\lambda_s(z-a)^{-s}\right)\phi(z)
$$
on a punctured neighborhood of $a$. Then, for every $\chi\in C_c(\C)$,
\begin{equation}\label{eq:full-local-vague}
\frac1{mn}\sum_{\substack{F^{(n)}(\zeta)=0\\0<|\zeta-a|<r}}
\operatorname{mult}_\zeta(F^{(n)})\,
\chi\!\left(n^{1/m}(\zeta-a)\right)
\longrightarrow
\int_\C\chi\,\dd\widehat\mu_{m,\lambda_m},
\end{equation}
where $r>0$ is any fixed radius on which the displayed local representation of $F$ is valid and for which $\phi$ is holomorphic and nonvanishing on a neighborhood of $\overline{D(a,r)}$; for a fixed $\chi$, the compact support condition makes the sum independent of the admissible choice of $r$ for all large $n$.

If $r_n\downarrow0$, $n^{1/m}r_n\to\infty$, and
\begin{equation}\label{eq:full-local-total-hyp}
\frac1{mn}\sum_{\substack{F^{(n)}(\zeta)=0\\0<|\zeta-a|<r_n}}
\operatorname{mult}_\zeta(F^{(n)})\longrightarrow1,
\end{equation}
then
\begin{equation}\label{eq:full-local-weak}
\frac1{mn}\sum_{\substack{F^{(n)}(\zeta)=0\\0<|\zeta-a|<r_n}}
\operatorname{mult}_\zeta(F^{(n)})\,
\delta_{\,n^{1/m}(\zeta-a)}
\xrightarrow{\;w\;}
\widehat\mu_{m,\lambda_m}
\end{equation}
weakly on $\widehat\C$ as finite Borel measures.
\end{theorem}

\begin{proof}
Translate $a$ to $0$ and choose $r>0$ so that the displayed local representation is valid on $0<|z|<r$ and $\phi$ is holomorphic and nonvanishing on a neighborhood of $\{|z|\le r\}$. Define $H_n$ and $G_n(u):=H_n(n^{-1/m}u)$ as in Lemma~\ref{lem:full-local-perturb}. The recurrence in that lemma gives $H_0\equiv1$, $H_n(0)=(m\lambda_m)^n$, and holomorphy of $H_n$ on $|z|<r$. Thus the zeros counted in \eqref{eq:full-local-vague} are exactly the zeros of $H_n$ in the corresponding punctured disk.

For $u\ne0$, put $z=n^{-1/m}u$. Let $\psi\in C_c^\infty(\C^\times)$ and set $L=\operatorname{supp}\psi$. For all sufficiently large $n$, the set $L$ is contained in $\{|u|<n^{1/m}r\}$, and \eqref{eq:H-Y-relation} holds on a neighborhood of $L$. Thus $Y_n$ and $G_n$ have the same zero divisor on that neighborhood, because the factor $u^{mn}$ has no zeros on $L$. Therefore
\[
\frac1{mn}\left(\sum_{G_n(u)=0}\operatorname{mult}_u(G_n)\psi(u)
-\sum_{Y_n^{(0)}(u)=0}\operatorname{mult}_u(Y_n^{(0)})\psi(u)\right)
=\frac1{2\pi m}\left\langle \Delta\left(\frac1n\log|Y_n|-\frac1n\log|Y_n^{(0)}|\right),\psi\right\rangle.
\]
The right-hand side tends to $0$ by \eqref{eq:full-local-perturb-L1}, so the difference of the full and reduced normalized zero-counting measures tends to $0$ distributionally on $\C^\times$. To pass from smooth tests to vague convergence, fix $L\Subset\C^\times$ and choose $\psi\in C_c^\infty(\C^\times)$, $\psi\ge0$, with $\psi\equiv1$ on $L$. Let $\mu_n^{\mathrm{full}}$ and $\mu_n^{\mathrm{red}}$ denote these two normalized measures. The reduced measures have bounded mass on $\operatorname{supp}\psi$, and
\[
\int\psi\,\dd\mu_n^{\mathrm{full}}
=\int\psi\,\dd\mu_n^{\mathrm{red}}+o(1),
\]
so the full measures are locally bounded. Local boundedness and distributional convergence against $C_c^\infty$ tests imply vague convergence by regular approximation of continuous compactly supported functions. For the reduced sequence, Corollaries~\ref{cor:micro-higher-zero} and~\ref{cor:mp} give the corresponding convergence with normalization $mN_n$, where $N_n$ is the number of nonzero zeros of $\Pi_n^{(\alpha,m)}$ and $N_n/n\to1$. Hence
\[
\frac1{mn}\sum_{Y_n^{(0)}(u)=0}\operatorname{mult}_u(Y_n^{(0)})\delta_u
\longrightarrow \widehat\mu_{m,\lambda_m}
\]
vaguely on $\C$. The possible zero of $\Pi_n^{(\alpha,m)}$ at $0$ has $O(1)$ multiplicity and corresponds to $u=\infty$, so it does not affect compact subsets of $\C$ under the normalization $mn$.

It remains only to pass across $u=0$. The support of $\widehat\mu_{m,\lambda_m}$ is bounded away from $0$, and Lemma~\ref{lem:full-local-no-origin} shows that the full normalized zero mass in $0<|u|<\delta$ tends to $0$ as $\delta\downarrow0$. Applying the convergence on $\C^\times$ to a cutoff of $\chi$ outside $|u|<\delta$ and then letting $\delta\downarrow0$ proves \eqref{eq:full-local-vague} for all $\chi\in C_c(\C)$.

Assume now \eqref{eq:full-local-total-hyp}. If $\chi\in C_c(\C)$ and $\operatorname{supp}\chi\subset\{|u|\le R\}$, then $Rn^{-1/m}<r_n$ for all large $n$, so the measures in \eqref{eq:full-local-weak} have the same $\chi$-integrals as in \eqref{eq:full-local-vague}. Thus they converge vaguely on $\C$ to $\widehat\mu_{m,\lambda_m}$. Their total masses tend to $1$ by \eqref{eq:full-local-total-hyp}; since $\widehat\mu_{m,\lambda_m}$ is a probability measure on $\widehat\C$ with no atom at $\infty$, vague convergence on $\C$ plus convergence of total masses is equivalent to weak convergence on $\widehat\C$.
\end{proof}

\section{Sublinear zero measures in essential Voronoi cells}\label{sec:sublinear}

The fixed-scale limit gives zero mass to compact subsets of punctured open Voronoi cells. In a cell whose nearest site is an essential singularity, the first nonzero normalization in the original $z$-plane is the $n^{m_i/(m_i+1)}$ zero-counting normalization supplied by Theorem~\ref{prop:local-coef-unified}. The theorem below is local on one leading Stokes chamber away from the singularity, with all asymptotics obtained on compact subsets of that chamber; it identifies the sublinear zero measure and the final set relative to that open chamber, with active labels computed explicitly in Proposition~\ref{prop:active-wright-coefficients} and signs given by the corresponding Picard--Lefschetz intersections. It makes no assertion on the chamber boundary; in particular, zero-carrying curves contained in the leading Stokes set are outside its scope. It also makes no assertion on Voronoi edges or vertices, or at transition scales. Its limiting density is not locally finite at the singularity, consistently with the $O(n)$ collapsing cluster which gives the atom in Theorem~\ref{thm:limit}.

\begin{theorem}[sublinear zero measure and relative final set in an essential chamber]\label{thm:sublinear-essential}
Fix an essential site $a_i\in\Z(T)$, so $m_i\ge1$. Let $D\Subset\mathcal V_i^\circ\setminus\Sigma$ be a simply connected domain, choose the branch $\eta_D$ from Theorem~\ref{prop:local-coef-unified}\textup{(2)}, let $V$ be a connected component of $D\setminus\mathfrak S_{i,D}$, and let $I_{i,V}$ be the active set computed by Proposition~\ref{prop:active-wright-coefficients}. Put
\[
\alpha_i:=\frac{m_i}{m_i+1},\qquad
\varphi_{i,\nu}(z):=\frac{m_i+1}{m_i}\,\omega_\nu\eta_D(z),\qquad
\Theta_{i,V}(z):=\max_{\nu\in I_{i,V}}\Re\varphi_{i,\nu}(z).
\]
For $n\ge1$ define the sublinear zero-counting measure on $V$ by
\[
\mathfrak z_{i,n,V}:=n^{-\alpha_i}
\sum_{\substack{\zeta\in V\\ B_n(\zeta)=0}}
\operatorname{mult}_\zeta(B_n)\,\delta_\zeta .
\]
Then
\[
\mathfrak z_{i,n,V}\longrightarrow
\mathfrak z_{i,V}:=\frac{1}{2\pi}\Delta\Theta_{i,V}
\]
vaguely on $V$. Consequently every compact $K\Subset V$ contains $O_K(n^{\alpha_i})$ zeros of $B_n$, counted with multiplicity. Because $V\cap\Sigma=\varnothing$, the same convergence and count hold with the zeros of $f^{(n)}$ in place of the zeros of $B_n$.

Moreover, with support taken relative to $V$ and with no assertion about boundary points of $V$,
\[
L(f)\cap V=\operatorname{supp}\mathfrak z_{i,V}.
\]
For $\nu\ne\mu$ in $I_{i,V}$ set
\[
\Gamma_{\nu\mu}:=
\{z\in V:\Re\varphi_{i,\nu}(z)=\Re\varphi_{i,\mu}(z)=\Theta_{i,V}(z)\}.
\]
Then, with support and union taken relative to $V$,
\[
\operatorname{supp}\mathfrak z_{i,V}
=\bigcup_{\substack{\nu<\mu\\ \nu,\mu\in I_{i,V}}}\Gamma_{\nu\mu}.
\]
Each nonempty $\Gamma_{\nu\mu}$ is a union of smooth real-analytic arcs in $V$; endpoints in $\partial V$ do not contribute to the measure on $V$. On each such arc the density is
\[
\dd\mathfrak z_{i,V}
=\frac{|\omega_\nu-\omega_\mu|}{2\pi}\frac{|\eta_D(z)|}{|z-a_i|}\,\dd\ell
=\frac{|\omega_\nu-\omega_\mu|\,|m_i\lambda_{i,m_i}|^{1/(m_i+1)}}{2\pi |z-a_i|^{(2m_i+1)/(m_i+1)}}\,\dd\ell .
\]
If $K\Subset V$ is disjoint from $\operatorname{supp}\mathfrak z_{i,V}$, then neither $B_n$ nor $f^{(n)}$ has zeros on $K$ for all sufficiently large $n$.
\end{theorem}

\begin{proof}
Set
\[
Z_{i,n}(z):=[\zeta^n]G_{i,z}(\zeta)
=\frac{B_n(z)}{n!}\left(\frac{a_i-z}{W(z)}\right)^n,
\qquad
F_n(z):=n^{-\theta_i}Z_{i,n}(z).
\]
The coefficient identity follows from the definition of $G_{i,z}$ in Theorem~\ref{prop:local-coef-unified}. Since $D\cap\Sigma=\varnothing$, the omitted factor $((a_i-z)/W(z))^n/n!$ is holomorphic and nonvanishing on $D$, so $F_n$, $Z_{i,n}$, and $B_n$ have the same zeros in $D$. By Proposition~\ref{prop:local}, \(B_n\not\equiv0\); hence \(F_n\) is not identically zero on any component of \(D\). The expansion \eqref{eq:Bn-local-unified-essential} gives, uniformly on compact subsets of $V$,
\[
F_n(z)=
\sum_{\nu\in I_{i,V}}\sigma_{i,\nu,V}\mathcal A_{i,\nu}(z)\exp(\Xi_{i,\nu}(z;n))(1+o(1))+
O\!\left(\exp(M_{i,V}(z;n)-c n^{\alpha_i})\right),
\]
where $n^{-\alpha_i}\Xi_{i,\nu}(z;n)\to\varphi_{i,\nu}(z)$ locally uniformly. Hence, for every compact $K\Subset V$,
\[
\limsup_{n\to\infty}\sup_{z\in K}\bigl(n^{-\alpha_i}\log|F_n(z)|-\Theta_{i,V}(z)\bigr)\le0.
\]
On any compact set on which one active branch has real part at least $\delta>0$ larger than all other active branches, the same expansion gives
\[
n^{-\alpha_i}\log|F_n|\longrightarrow \Theta_{i,V}
\]
uniformly, and $F_n$ is eventually zero-free there. The union of these one-branch regions is open and dense in $V$. Indeed,
\[
\eta_D'(z)=-\frac{m_i}{m_i+1}\frac{\eta_D(z)}{z-a_i}\ne0,
\]
so $\Re((\omega_\nu-\omega_\mu)\eta_D)$ is a nonconstant real-analytic function for every $\nu\ne\mu$. Its zero set has empty interior, and the complement of the one-branch regions is contained in the finite union of these zero sets.

By the compactness theorem for subharmonic functions, every subsequence of $n^{-\alpha_i}\log|F_n|$ has a further subsequence which either tends to $-\infty$ locally uniformly or converges in $L^1_{\mathrm{loc}}(V)$ to a subharmonic function. The first alternative is excluded by the preceding convergence on a one-branch region. Any $L^1_{\mathrm{loc}}$ limit is at most $\Theta_{i,V}$ and equals $\Theta_{i,V}$ on the dense open one-branch set; the complement is contained in a finite union of real-analytic level sets and has planar measure zero. Thus the limit is $\Theta_{i,V}$ as a subharmonic function. Therefore
\[
n^{-\alpha_i}\log|F_n|\longrightarrow\Theta_{i,V}
\qquad\text{in }L^1_{\mathrm{loc}}(V).
\]
Taking distributional Laplacians gives the vague convergence of $\mathfrak z_{i,n,V}$. The $O_K(n^{\alpha_i})$ bound follows by enclosing $K$ in a relatively compact open set whose boundary has zero $\mathfrak z_{i,V}$-mass.

Write $u_\nu:=\Re\varphi_{i,\nu}$. Since
\[
(\varphi_{i,\nu}-\varphi_{i,\mu})'(z)
=-(\omega_\nu-\omega_\mu)\frac{\eta_D(z)}{z-a_i}\ne0,
\]
each pairwise tie set is a smooth real-analytic arc. For fixed $z\in V$,
\[
\nu\mapsto u_\nu(z)=\frac{m_i+1}{m_i}\Re(\omega_\nu\eta_D(z))
\]
is, up to the common positive factor $(m_i+1)/m_i$, the orthogonal projection of the root of unity $\omega_\nu$ onto a fixed real line. If three distinct values were equal, the corresponding three roots of unity would lie on one affine line, impossible because a line meets the unit circle in at most two points. Hence no three distinct functions $u_\nu$ agree at one point.

Consequently, at every point of a dominant set $\Gamma_{\nu\mu}$ exactly the two branches $u_\nu$ and $u_\mu$ attain the value $\Theta_{i,V}$. Distinct co-dominant arcs therefore have no interior intersections in $V$. Pairwise tie sets for lower, non-dominant pairs may intersect or coincide, but locally they are dominated by another branch and do not affect $\Theta_{i,V}$. Since $\Theta_{i,V}$ is the maximum of finitely many harmonic functions with locally bounded gradients and is locally harmonic off the co-dominant arcs, its Riesz measure has no atomic part in $V$ and is supported exactly on the arcs where two active branches tie and dominate.

On a relative open subarc \(\Gamma\) of some co-dominant set \(\Gamma_{\nu\mu}\), the function \(\Theta_{i,V}\) equals \(u_\nu\) on one side and \(u_\mu\) on the other. If \(\mathbf n\) is a unit normal to \(\Gamma\), then \(u_\nu-u_\mu\) is constant on \(\Gamma\), so \(\nabla(u_\nu-u_\mu)\) is normal to \(\Gamma\); since \((\varphi_{i,\nu}-\varphi_{i,\mu})'\ne0\), this normal derivative is nonzero. The distributional jump formula for the maximum of two harmonic functions gives
\[
\frac{1}{2\pi}\bigl|\partial_{\mathbf n}(u_\nu-u_\mu)(z)\bigr|\,\dd\ell
\]
on $\Gamma$. Since $u_\nu-u_\mu$ is constant on $\Gamma$, its gradient is normal to $\Gamma$, and for a holomorphic function $\phi$ one has $|\nabla\Re\phi|=|\phi'|$. Therefore
\[
\bigl|\partial_{\mathbf n}(u_\nu-u_\mu)(z)\bigr|
=\bigl|(\varphi_{i,\nu}-\varphi_{i,\mu})'(z)\bigr|
=|\omega_\nu-\omega_\mu|\frac{|\eta_D(z)|}{|z-a_i|}.
\]
This gives the first displayed density in the theorem, and the identity
\[
\eta_D(z)^{m_i+1}=m_i\lambda_{i,m_i}(z-a_i)^{-m_i}
\]
gives the second.

If $K$ is disjoint from the support, compactness gives a finite cover by one-branch regions with a uniform positive dominance gap, and the zero-free assertion follows from the preceding one-branch asymptotics. This proves $L(f)\cap V\subset\operatorname{supp}\mathfrak z_{i,V}$. Conversely, if $z_0\in\operatorname{supp}\mathfrak z_{i,V}$ and $U$ is a neighborhood of $z_0$ in $V$, choose an open set $G$ with $\overline G\Subset U$ and $\mathfrak z_{i,V}(G)>0$ whose boundary has zero $\mathfrak z_{i,V}$-mass. Vague convergence gives
\[
\sum_{\substack{\zeta\in G\\ B_n(\zeta)=0}}\operatorname{mult}_\zeta(B_n)
=n^{\alpha_i}\mathfrak z_{i,V}(G)+o(n^{\alpha_i}),
\]
so $B_n$ has a zero in $U$ for all large $n$. Thus $\operatorname{supp}\mathfrak z_{i,V}\subset L(f)\cap V$. The equivalence with zeros of $f^{(n)}$ follows from \eqref{eq:fnrep} on $V$.
\end{proof}

Figure~\ref{fig:active-saddle-rays} shows the adjacent-active saddle mechanism and the restriction to ray portions lying inside the chosen chamber.

\begin{center}
\begin{minipage}{\linewidth}
\centering
\begin{tikzpicture}[
    x=1cm,
    y=1cm,
    line cap=round,
    line join=round,
    every node/.style={font=\small},
    activesaddle/.style={circle,fill=blue!70!black,inner sep=1.55pt},
    inactivesaddle/.style={circle,fill=black!35,inner sep=1.20pt},
    ray/.style={line width=1.05pt,blue!70!black},
    stokes/.style={line width=0.70pt,black!45,densely dashed},
    chamberbd/.style={line width=0.70pt,black!28}
]
% Active saddles in the critical-value plane.  This panel records only the
% adjacent edges of the visible chain, because those are the pairs relevant
% to Corollary~\ref{cor:sublinear-rays}.
\begin{scope}[shift={(-3.85,0.20)}]
    \draw[->,black!45] (-2.05,0) -- (2.32,0) node[right] {$\Re v$};
    \draw[->,black!45] (0,-1.72) -- (0,1.88) node[above] {$\Im v$};
    \draw[black!20] (0,0) circle (1.36);

    \coordinate (v0) at (44:1.36);
    \coordinate (v1) at (116:1.36);
    \coordinate (v2) at (188:1.36);
    \coordinate (v3) at (260:1.36);
    \coordinate (v4) at (332:1.36);

    \draw[black!30,line width=0.65pt] (v0)--(v1)--(v2)--(v3)--(v4)--cycle;
    \draw[blue!70!black,line width=1.30pt] (v1)--(v0)--(v4)--(v3);

    \node[activesaddle,label={[blue!70!black]above right:$v_0$}] at (v0) {};
    \node[activesaddle,label={[blue!70!black]above:$v_1$}] at (v1) {};
    \node[inactivesaddle,label={[black!55]left:$v_2$}] at (v2) {};
    \node[activesaddle,label={[blue!70!black]below:$v_3$}] at (v3) {};
    \node[activesaddle,label={[blue!70!black]below right:$v_4$}] at (v4) {};

    \node[blue!70!black,fill=white,inner sep=1pt] at (70:1.36) {$01$};
    \node[blue!70!black,fill=white,inner sep=1pt] at (6:1.36) {$04$};
    \node[blue!70!black,fill=white,inner sep=1pt] at (296:1.36) {$34$};

    \node[align=center,text width=4.25cm,font=\scriptsize,black!68] at (0,-2.48)
        {active chain $v_1-v_0-v_4-v_3$;\\
         candidate ties $\Gamma_{01},\Gamma_{04},\Gamma_{34}$};
\end{scope}

% One physical Stokes chamber.  The dashed Stokes curves are boundaries of V,
% not curves in its interior.  The singularity a_i is an excluded endpoint.
\begin{scope}[shift={(2.85,0)}]
    \begin{scope}[shift={(-1.15,0)}]
        % chosen chamber V, drawn as a punctured sector inside a local window
        \path[fill=black!4]
            (-38:0.34) -- (-38:2.50)
            arc[start angle=-38,end angle=52,radius=2.50]
            -- (52:0.34)
            arc[start angle=52,end angle=-38,radius=0.34]
            -- cycle;
        \draw[chamberbd]
            (-38:0.34) -- (-38:2.50)
            arc[start angle=-38,end angle=52,radius=2.50]
            -- (52:0.34)
            arc[start angle=52,end angle=-38,radius=0.34]
            -- cycle;

        % Stokes boundaries of this chamber
        \draw[stokes] (-38:0.25) -- (-38:2.68);
        \draw[stokes] (52:0.25) -- (52:2.68)
            node[pos=0.66,anchor=south west,xshift=-1pt,yshift=3pt,rotate=-36,black!55,font=\scriptsize] {$\mathfrak S_{i,D}$};

        % Ray portions that do lie in this chamber and hence may carry the measure
        \draw[blue!70!black,line width=2.8pt,opacity=0.16] (7:0.42) -- (7:2.35);
        \draw[blue!70!black,line width=2.8pt,opacity=0.16] (32:0.42) -- (32:2.20);
        \draw[ray] (7:0.42) -- (7:2.35)
            node[pos=0.66,anchor=north west,xshift=5pt,yshift=1pt] {$\Gamma_{04}$};
        \draw[ray] (32:0.42) -- (32:2.20)
            node[pos=0.56,anchor=south west,xshift=9pt,yshift=-5pt] {$\Gamma_{34}$};

        % excluded singular endpoint
        \fill[white] (0,0) circle (2.8pt);
        \draw[black] (0,0) circle (2.8pt);
        \node[anchor=east] at (-0.13,-0.02) {$a_i$};
        \node[anchor=north east,black!60,font=\scriptsize] at (-0.08,-0.25) {excluded};

        \node[blue!70!black] at (20:0.95) {$V$};
    \end{scope}

    \node[align=center,text width=5.0cm,font=\scriptsize] at (0.42,2.42)
        {chosen component $V\subset D\setminus\mathfrak S_{i,D}$};
    \node[align=center,text width=5.1cm,font=\scriptsize,black!68] at (0.38,-2.18)
        {only relatively open portions of adjacent-active tie rays lying inside $V$ carry $\mathfrak z_{i,V}$};
\end{scope}
\end{tikzpicture}
\captionsetup{type=figure,hypcap=false}
\caption{Schematic, not-to-scale form of the saddle mechanism in Theorem~\ref{thm:sublinear-essential} and Corollary~\ref{cor:sublinear-rays}. The left panel records the visible active chain of limiting saddle values and its adjacent active edges, labelled with the smaller index first; these are the only pairs that can produce chamberwise co-dominant ties. The right panel shows one connected component $V\subset D\setminus\mathfrak S_{i,D}$: the dashed Stokes curves are boundaries of this chamber, the endpoint $a_i$ is excluded, and only the relatively open ray portions lying in $V$ contribute to $\mathfrak z_{i,V}=(2\pi)^{-1}\Delta\Theta_{i,V}$ and to the relative final set described by the theorem.}
\label{fig:active-saddle-rays}
\end{minipage}
\end{center}

\begin{cor}[directions of the chamberwise sublinear support]\label{cor:sublinear-rays}
In the setting of Theorem~\ref{thm:sublinear-essential}, write $\lambda:=\lambda_{i,m_i}$; all assertions below are made after intersection with the chosen chamber $V$. For $0\le\nu<\mu\le m_i$, the pairwise tie set $\Re((\omega_\nu-\omega_\mu)\eta_D)=0$ is contained in the rays
\[
a_i+r e^{i\Theta_{\nu\mu k}},\qquad r>0,
\]
where, for any choice of $\arg\lambda$,
\[
\Theta_{\nu\mu k}:=\frac{\arg\lambda+\pi(\nu+\mu)+\pi k(m_i+1)}{m_i}\pmod{2\pi},\qquad k\in\mathbb Z.
\]
Let $z_\ast\in V$, set $v_\nu:=\omega_\nu\eta_D(z_\ast)$, and let
\[
\begin{aligned}
\mathcal R_{i,V}:=\{\nu:\ &v_\nu\text{ lies on the boundary chain of }
\operatorname{conv}\{v_0,\dots,v_{m_i}\}\\
&\text{visible from }+\infty\text{ in the real direction}\},
\end{aligned}
\]
with the order inherited from that chain. Equivalently, with $\operatorname{Arg}\in(-\pi,\pi]$,
\[
\mathcal R_{i,V}=\left\{\nu:
\left|\operatorname{Arg}v_\nu\right|<\frac\pi2+\frac\pi{m_i+1}\right\}.
\]
Then $I_{i,V}=\mathcal R_{i,V}$. The active co-dominant pairs in Theorem~\ref{thm:sublinear-essential} are exactly the adjacent indices in this ordered chain whose tie rays meet $V$. Thus only relatively open portions of the displayed rays lying in the chosen chamber $V$ and associated with adjacent indices of $\mathcal R_{i,V}$ belong to $L(f)\cap V$ or carry $\mathfrak z_{i,V}$. A formal tie ray lying in the leading Stokes set is outside the theorem. In the local statement a contributing segment is followed only while it remains in $V$; for the global chamber in $\mathcal V_i^\circ\setminus\{a_i\}$, it ends only at a Voronoi edge, a leading Stokes curve, or infinity.

For P\'olya's example $f(z)=z^{-1}e^{-1/z}$ one has $m_i=1$, $a_i=0$, $\lambda=-1$, and both vertices are visible from $+\infty$ in the real direction in the chamber containing $(0,\infty)$. Hence the unique outgoing direction is the positive real axis,
\[
n^{-1/2}\sum_{\substack{f^{(n)}(\zeta)=0\\ \zeta\ne0}}
\operatorname{mult}_\zeta(f^{(n)})\,\delta_\zeta
\xrightarrow{\;v\;}\frac{1}{\pi}x^{-3/2}\,\dd x
\]
vaguely on $\C^\times$, and $L(f)\cap\C^\times=(0,\infty)$.
\end{cor}

\begin{proof}
The tie condition is $\Re((\omega_\nu-\omega_\mu)\eta_D(z))=0$. Using
\[
\eta_D(z)^{m_i+1}=m_i\lambda(z-a_i)^{-m_i}
\]
and
\[
\arg(\omega_\nu-\omega_\mu)\equiv \frac{\pi(\nu+\mu)}{m_i+1}-\frac{\pi}{2}\pmod{\pi}
\]
gives the pairwise angle formula. Proposition~\ref{prop:active-wright-coefficients} gives $I_{i,V}=\mathcal R_{i,V}$. The maximum of the real projection over the vertices of a convex chain can have a two-vertex face only along an edge of that chain; hence the active co-dominant pairs are precisely adjacent indices of the ordered chain $\mathcal R_{i,V}$ whose tie rays meet $V$.

For P\'olya's example, on the chamber $\C\setminus(-\infty,0]$ choose $\eta_D(z)=iz^{-1/2}$. The two vertices of the critical-value segment are both visible from $+\infty$ in the real direction, so both saddle coefficients are nonzero by Proposition~\ref{prop:active-wright-coefficients}. The density in Theorem~\ref{thm:sublinear-essential} is
\[
\frac{|1-(-1)|}{2\pi}|z|^{-3/2}\,\dd\ell=\frac1\pi x^{-3/2}\,\dd x
\]
on the positive ray. The classical identity
\[
f^{(n)}(z)=(-1)^n n!\,z^{-n-1}e^{-1/z}L_n(1/z)
\]
and the classical positivity of the zeros of $L_n$ show that all finite zeros are positive, so no point of $\C^\times\setminus(0,\infty)$ lies in the final set. The local vague convergence on the chamber is the stated vague convergence on $\C^\times$, and the positive limiting density gives zeros in every subinterval of $(0,\infty)$ for all large $n$.
\end{proof}

\begin{figure}[t]
\centering
\begin{tikzpicture}
\begin{groupplot}[
    group style={group size=2 by 1, horizontal sep=1.95cm},
    width=0.335\linewidth,
    height=0.235\linewidth,
    scale only axis,
    tick align=outside,
    axis line style={line width=0.45pt},
    tick style={line width=0.45pt},
    grid=both,
    major grid style={draw=black!10},
    minor grid style={draw=black!5},
    label style={font=\small},
    tick label style={font=\small},
    title style={font=\small},
    legend style={font=\footnotesize, draw=none, fill=none},
]

\nextgroupplot[
    xmode=log,
    xmin=0.25, xmax=40,
    ymin=0, ymax=0.86,
    xlabel={$u=nz=n/y$},
    ylabel={density},
    title={microscopic reciprocal law},
    ytick={0,0.2,0.4,0.6,0.8},
    extra x ticks={0.25},
    extra x tick labels={$1/4$},
    legend style={at={(0.98,0.98)}, anchor=north east},
]
    \addplot+[
        ybar interval,
        mark=none,
        area legend,
        fill=black!12,
        draw=black!35,
        line width=0.25pt,
    ] table[x=u, y=density, row sep=\\] {
u density\\
0.25 0.46438\\
0.27512 0.75352\\
0.30277 0.82166\\
0.33320 0.82129\\
0.36668 0.79152\\
0.40353 0.73979\\
0.44408 0.67224\\
0.48871 0.62782\\
0.53782 0.55507\\
0.59187 0.50439\\
0.65135 0.44560\\
0.71680 0.40491\\
0.78883 0.34691\\
0.86811 0.31523\\
0.95535 0.27776\\
1.05140 0.23662\\
1.15700 0.21502\\
1.27330 0.18236\\
1.40120 0.16570\\
1.54200 0.13982\\
1.69700 0.12705\\
1.86750 0.10657\\
2.05520 0.096836\\
2.26180 0.080661\\
2.48900 0.073295\\
2.73920 0.060547\\
3.01440 0.055018\\
3.31740 0.047495\\
3.65070 0.040886\\
4.01760 0.035089\\
4.42140 0.031884\\
4.86570 0.027269\\
5.35460 0.023230\\
5.89270 0.019701\\
6.48490 0.017902\\
7.13660 0.015106\\
7.85380 0.013726\\
8.64300 0.011513\\
9.51160 0.0095901\\
10.46700 0.0087144\\
11.51900 0.0071988\\
12.67700 0.0065414\\
13.95100 0.0059441\\
15.35300 0.0048611\\
16.89600 0.0039264\\
18.59400 0.0040139\\
20.46200 0.0032421\\
22.51800 0.0025778\\
24.78100 0.0023424\\
27.27200 0.0021285\\
30.01200 0.0016578\\
33.02800 0.0017575\\
36.34700 0.0013689\\
40.00000 0\\
    };
    \addlegendentry{zeros, $n=1200$}

    \addplot[
        black,
        line width=1.05pt,
        mark=none,
        domain=0.25001:40,
        samples=500,
    ] {sqrt(4*x-1)/(2*pi*x^2)};
    \addlegendentry{$\rho_U(u)$}

\nextgroupplot[
    xmode=log,
    ymode=log,
    xmin=0.02, xmax=20,
    ymin=0.002, ymax=200,
    xlabel={fixed ray coordinate $x=z>0$},
    ylabel={$n^{3/2}\rho_U(nx)$},
    ylabel near ticks,
    ylabel style={xshift=0.35em},
    title={sublinear fixed-ray limit},
    legend style={at={(0.03,0.03)}, anchor=south west},
]
    \addplot[
        black!55,
        densely dotted,
        line width=0.95pt,
        mark=none,
        domain=0.0313:20,
        samples=240,
    ] {1/(pi*pow(x,1.5))*sqrt(1 - 1/(32*x))};
    \addlegendentry{$n=8$}

    \addplot[
        black!65,
        dashed,
        line width=0.85pt,
        mark=none,
        domain=0.02:20,
        samples=240,
    ] {1/(pi*pow(x,1.5))*sqrt(1 - 1/(128*x))};
    \addlegendentry{$n=32$}

    \addplot[
        black,
        line width=1.15pt,
        mark=none,
        domain=0.02:20,
        samples=240,
    ] {1/(pi*pow(x,1.5))};
    \addlegendentry{$\pi^{-1}x^{-3/2}$}
\end{groupplot}
\end{tikzpicture}
\caption{Microscopic and sublinear zero laws for P\'olya's example $f(z)=z^{-1}e^{-1/z}$. In the left panel the scaled zeros $u_{k,n}=nz_{k,n}=n/y_{k,n}$, where $y_{k,n}$ are the zeros of $L_n$, are shown as a histogram normalized in the $u$-variable and compared with the reciprocal Marchenko--Pastur density $\rho_U(u)=\sqrt{4u-1}/(2\pi u^2)$ for $u\ge1/4$ and $\rho_U(u)=0$ for $u<1/4$. The right panel plots the smooth scaled fixed-ray densities $n^{3/2}\rho_U(nx)$ for $n=8$ and $n=32$, together with their limiting density $\pi^{-1}x^{-3/2}$ on the fixed positive ray.}
\label{fig:laguerre-microscopic-sublinear}
\end{figure}

\subsection{Coefficient functions and divisor limits}\label{subsec:essential-partition}

Fix an essential site $a_i\in\Z(T)$ of order $m_i\ge1$, and let $D\Subset\mathcal V_i^\circ\setminus\Sigma$ be simply connected. For a domain $U$ and a holomorphic function $Z\not\equiv0$ on $U$, write
\[
\operatorname{div}_U Z:=
\sum_{\substack{w\in U:\,Z(w)=0}}
\operatorname{mult}_w(Z)\,\delta_w .
\]
Then, as measures on $U$,
\[
\operatorname{div}_U Z=\frac{1}{2\pi}\Delta\log|Z|.
\]

The coefficient generating function is
\[
G_{i,z}(\zeta)=\mathcal C(z,(a_i-z)\zeta)
=\sum_{n\ge0}Z_{i,n}(z)\zeta^n,
\]
where
\begin{equation}\label{eq:essential-canonical-partition}
Z_{i,n}(z):=[\zeta^n]G_{i,z}(\zeta)
=\frac{C_n(z)}{n!}(a_i-z)^n
=\frac{B_n(z)}{n!}\left(\frac{a_i-z}{W(z)}\right)^n .
\end{equation}
Since $D$ is disjoint from $a_i$ and from the zeros of $W$, the factor multiplying $B_n$ in \eqref{eq:essential-canonical-partition} is holomorphic and nonvanishing on $D$. Hence
\[
\operatorname{div}_D Z_{i,n}=\operatorname{div}_D B_n=\operatorname{div}_D(f^{(n)}|_D),
\]
with multiplicities.

The local factorization \eqref{eq:Gi-local-unified} gives, near the unique singularity of $G_{i,z}$ on $|\zeta|=1$,
\[
G_{i,z}(\zeta)=
(1-\zeta)^{\beta_i}
\exp\!\left(\sum_{s=1}^{m_i}\widetilde\lambda_{i,s}(z)(1-\zeta)^{-s}\right)R_i(z,\zeta),
\qquad
\widetilde\lambda_{i,s}(z)=\lambda_{i,s}(z-a_i)^{-s}.
\]
Cauchy's formula with $t=1-\zeta$ has phase
\[
\Phi_{i,z,n}(t)=\sum_{s=1}^{m_i}\widetilde\lambda_{i,s}(z)t^{-s}-(n+1)\log(1-t).
\]
Put
\[
M_{i,n}:=n^{m_i/(m_i+1)},
\qquad
\theta_i:=-\frac{2\beta_i+m_i+2}{2(m_i+1)},
\]
and choose the branch $\eta_D$ with $\eta_D^{m_i+1}=m_i\widetilde\lambda_{i,m_i}$. With $t=n^{-1/(m_i+1)}\tau$,
\[
M_{i,n}^{-1}\Phi_{i,z,n}\bigl(n^{-1/(m_i+1)}\tau\bigr)
=\widetilde\lambda_{i,m_i}(z)\tau^{-m_i}+\tau+O\!\left(n^{-1/(m_i+1)}\right)
\]
uniformly for $z$ in compact subsets of $D$ and $\tau$ in compact subsets of $\C^\times$. The critical points of the leading phase are $\tau_\nu(z)=\omega_\nu\eta_D(z)$, and their critical values are
\[
\varphi_{i,\nu}(z)=\frac{m_i+1}{m_i}\omega_\nu\eta_D(z),
\qquad 0\le\nu\le m_i.
\]

Let $V$ be a connected component of $D\setminus\mathfrak S_{i,D}$, and let $I_{i,V}$ be the active set from Proposition~\ref{prop:active-wright-coefficients}. Define
\[
\widehat Z_{i,n}(z):=n^{-\theta_i}Z_{i,n}(z),
\qquad
p_{i,n}(z):=M_{i,n}^{-1}\log|\widehat Z_{i,n}(z)|,
\]
and
\[
p_{i,V}(z):=\max_{\nu\in I_{i,V}}\Re\varphi_{i,\nu}(z).
\]
The proof of Theorem~\ref{thm:sublinear-essential} gives
\[
p_{i,n}\longrightarrow p_{i,V}
\qquad\text{in }L^1_{\mathrm{loc}}(V).
\]
Since $n^{-\theta_i}$ is independent of $z$, $Z_{i,n}$ and $\widehat Z_{i,n}$ have the same divisor. Applying Poincar\'e--Lelong to the preceding $L^1_{\mathrm{loc}}$ convergence gives
\begin{equation}\label{eq:essential-fisher-zero-convergence}
\frac1{M_{i,n}}
\sum_{\substack{w\in V\;:\; Z_{i,n}(w)=0}}
\operatorname{mult}_w(Z_{i,n})\,\delta_w
\longrightarrow
\frac{1}{2\pi}\Delta p_{i,V}
\end{equation}
vaguely on $V$. On a two-phase arc where the active phases $\nu$ and $\mu$ dominate, the limiting line density is
\[
\frac{1}{2\pi}|(\varphi_{i,\nu}-\varphi_{i,\mu})'(z)|\,\dd\ell
=
\frac{|\omega_\nu-\omega_\mu|\,|m_i\lambda_{i,m_i}|^{1/(m_i+1)}}{2\pi |z-a_i|^{(2m_i+1)/(m_i+1)}}\,\dd\ell .
\]

Removing the analytic insertion $R_i$ gives the singular factor
\begin{equation}\label{eq:essential-singular-factor}
\Xi^{\mathrm{sing}}_{i,z}(\zeta):=(1-\zeta)^{\beta_i}
\exp\!\left(\sum_{s=1}^{m_i}\widetilde\lambda_{i,s}(z)(1-\zeta)^{-s}\right).
\end{equation}
For $|\zeta|<1$, with the branch determined by expansion at $\zeta=0$,
\[
(1-\zeta)^{\beta_i}=\exp\!\left(-\beta_i\sum_{k\ge1}\frac{\zeta^k}{k}\right),
\qquad
(1-\zeta)^{-s}=\sum_{k\ge0}\binom{k+s-1}{s-1}\zeta^k.
\]
Thus
\begin{equation}\label{eq:essential-cluster-activities}
\Xi^{\mathrm{sing}}_{i,z}(\zeta)
=e^{a_0(z)}\exp\!\left(\sum_{k\ge1}a_k(z)\zeta^k\right),
\end{equation}
where
\[
a_0(z)=\sum_{s=1}^{m_i}\widetilde\lambda_{i,s}(z),
\qquad
a_k(z)=\sum_{s=1}^{m_i}\widetilde\lambda_{i,s}(z)\binom{k+s-1}{s-1}-\frac{\beta_i}{k}\quad(k\ge1).
\]
Let $b_k(z):=k!a_k(z)$, and let $\mathfrak B_n$ denote the complete exponential Bell polynomial, normalized by
\[
\exp\!\left(\sum_{k\ge1}b_k\frac{\zeta^k}{k!}\right)
=\sum_{n\ge0}\mathfrak B_n(b_1,\ldots,b_n)\frac{\zeta^n}{n!}.
\]
The exponential formula gives
\[
e^{-a_0(z)}[\zeta^n]\Xi^{\mathrm{sing}}_{i,z}(\zeta)
=\frac{1}{n!}\mathfrak B_n\bigl(b_1(z),\ldots,b_n(z)\bigr).
\]
Equivalently, coefficient extraction gives the finite occupation-number identity
\begin{equation}\label{eq:essential-occupation}
[\zeta^n]\Xi^{\mathrm{sing}}_{i,z}(\zeta)
=e^{a_0(z)}
\sum_{\substack{N_1,N_2,\ldots\ge0\\ \sum_{k\ge1}kN_k=n}}
\prod_{k\ge1}\frac{a_k(z)^{N_k}}{N_k!}.
\end{equation}
The set-partition form is
\begin{equation}\label{eq:essential-labelled-polymer}
n!\,e^{-a_0(z)}[\zeta^n]\Xi^{\mathrm{sing}}_{i,z}(\zeta)
=\sum_{\pi\in\mathfrak P_n}\prod_{B\in\pi}b_{|B|}(z),
\end{equation}
where $\mathfrak P_n$ is the set of set partitions of $\{1,\ldots,n\}$.
For fixed $z\in D$, write $R_i(z,\zeta)=\sum_{\ell\ge0}r_{i,\ell}(z)\zeta^\ell$ at $\zeta=0$; its radius of convergence is greater than $1$. The full coefficient is
\begin{equation}\label{eq:essential-full-insertion}
Z_{i,n}(z)=\sum_{\ell=0}^n r_{i,\ell}(z)
[\zeta^{n-\ell}]\Xi^{\mathrm{sing}}_{i,z}(\zeta).
\end{equation}
The divisor limit \eqref{eq:essential-fisher-zero-convergence} is for the full coefficient $Z_{i,n}$, not for the singular factor alone. In the saddle expansion at the scale \(M_{i,n}\), \(R_i\) contributes the nonzero leading factor \(R_i(z,1)\); higher Taylor terms at \(\zeta=1\) are lower order, and the exponential phases determining the limiting divisor are unchanged.

\begin{prop}[simple-pole singular factor]\label{prop:essential-simple-laguerre}
Assume $m_i=1$ and put
\[
\ell_i(z):=\widetilde\lambda_{i,1}(z)=\frac{\lambda_{i,1}}{z-a_i}.
\]
For
\[
S_{i,n}(z):=[\zeta^n]\Xi^{\mathrm{sing}}_{i,z}(\zeta)
\]
one has
\[
e^{-\ell_i(z)}S_{i,n}(z)=L_n^{(-\beta_i-1)}(-\ell_i(z)).
\]
If $\beta_i<0$ and $n\ge1$, then the zeros in the $\ell_i$-variable are $\ell_i=-x$, where $x$ runs over the $n$ positive zeros of $L_n^{(-\beta_i-1)}$. Consequently the zeros in the $z$-plane are
\[
z=a_i-\frac{\lambda_{i,1}}{x},
\]
for those same $x$, with only points belonging to $D$ retained. This ray localization is for the singular factor; the analytic insertion $R_i$ need not preserve it.
\end{prop}

\begin{proof}
For $m_i=1$,
\[
\Xi^{\mathrm{sing}}_{i,z}(\zeta)=(1-\zeta)^{\beta_i}\exp\!\left(\frac{\ell_i(z)}{1-\zeta}\right)
=e^{\ell_i(z)}(1-\zeta)^{\beta_i}\exp\!\left(\frac{\ell_i(z)\zeta}{1-\zeta}\right).
\]
The Laguerre generating function
\[
\sum_{n\ge0}L_n^{(\gamma)}(x)\zeta^n=(1-\zeta)^{-\gamma-1}\exp\!\left(-\frac{x\zeta}{1-\zeta}\right)
\]
with $\gamma=-\beta_i-1$ and $x=-\ell_i(z)$ gives the displayed identity. If $\beta_i<0$, then $\gamma=-\beta_i-1\in\mathbb Z_{\ge0}$, so $L_n^{(\gamma)}$ has $n$ simple positive zeros. The factor $e^{\ell_i(z)}$ is nonvanishing, and solving $\lambda_{i,1}/(z-a_i)=-x$ gives the stated points.
\end{proof}

\begin{prop}[positivity of occupation weights]\label{prop:essential-cluster-potts}
Fix $z\in D$, and put
\[
q^{\mathrm{cl}}_{i,z}(k):=
 k\sum_{s=1}^{m_i}\widetilde\lambda_{i,s}(z)\binom{k+s-1}{s-1}-\beta_i,
\qquad k\ge1,
\]
so that $a_k(z)=q^{\mathrm{cl}}_{i,z}(k)/k$. Ignoring the common nonzero factor $e^{a_0(z)}$, a single fugacity rotation $\zeta=e^{-i\theta}\xi$ makes every summand on the right-hand side of \eqref{eq:essential-occupation}, for every coefficient $n$, nonnegative real if and only if
\begin{equation}\label{eq:essential-phase-locking}
e^{-ik\theta}a_k(z)\in[0,\infty)\qquad(k\ge1).
\end{equation}
For the activities \eqref{eq:essential-cluster-activities}, condition \eqref{eq:essential-phase-locking} for some $\theta$ is equivalent to $a_k(z)\in[0,\infty)$ for every $k\ge1$, and this is equivalent to
\[
\widetilde\lambda_{i,s}(z)\in\mathbb R\quad(1\le s\le m_i),
\qquad
q^{\mathrm{cl}}_{i,z}(k)\ge0\quad(k\ge1).
\]
Under these conditions $e^{a_0(z)}>0$. For $n\ge1$ set
\[
\Omega_n:=\left\{(N_1,\ldots,N_n)\in\mathbb Z_{\ge0}^n:
\sum_{k=1}^n kN_k=n\right\}
\]
and
\[
\mathcal Q^{\mathrm{sing}}_{i,n}(z):=e^{-a_0(z)}[\zeta^n]\Xi^{\mathrm{sing}}_{i,z}(\zeta)
=\sum_{N\in\Omega_n}\prod_{k=1}^n\frac{a_k(z)^{N_k}}{N_k!}.
\]
Whenever $\mathcal Q^{\mathrm{sing}}_{i,n}(z)>0$,
\[
\mathbb P_{i,n,z}(N)=
\frac{1}{\mathcal Q^{\mathrm{sing}}_{i,n}(z)}
\prod_{k=1}^n\frac{a_k(z)^{N_k}}{N_k!},
\qquad N\in\Omega_n,
\]
is a probability measure on $\Omega_n$. Equivalently, for any $0<x<1$, it is the law of independent Poisson variables with means $a_k(z)x^k$ conditioned on $\sum_{k\ge1}kN_k=n$; the auxiliary $x$ cancels under conditioning. If $R_i(z,\zeta)=\sum_{\ell\ge0}r_{i,\ell}(z)\zeta^\ell$ has $r_{i,\ell}(z)\in[0,\infty)$ for all $\ell$, then \eqref{eq:essential-full-insertion} gives the analogous positive convolution with one marked insertion of size $\ell$, whenever the normalizing coefficient is positive.
\end{prop}

\begin{proof}
The rotation $\zeta=e^{-i\theta}\xi$ replaces $a_k(z)$ by $e^{-ik\theta}a_k(z)$ in \eqref{eq:essential-occupation}. Thus \eqref{eq:essential-phase-locking} is sufficient. Conversely, the one-cluster configuration $N_k=1$, $N_j=0$ for $j\ne k$, occurs in the coefficient of $\xi^k$, so nonnegativity of all rotated summands forces \eqref{eq:essential-phase-locking} for every $k$.

Since $\widetilde\lambda_{i,m_i}(z)\ne0$,
\[
a_k(z)=\frac{\widetilde\lambda_{i,m_i}(z)}{(m_i-1)!}k^{m_i-1}+O_z(k^{m_i-2}),
\]
with the evident interpretation $O_z(k^{-1})$ when $m_i=1$. Hence $a_{k+1}(z)/a_k(z)\to1$ and only finitely many $a_k(z)$ vanish. If \eqref{eq:essential-phase-locking} holds, write $a_k(z)=e^{ik\theta}c_k$ with $c_k\ge0$. For all large $k$, $c_k>0$, and
\[
\frac{a_{k+1}(z)}{a_k(z)}=e^{i\theta}\frac{c_{k+1}}{c_k}.
\]
The left-hand side tends to $1$ and $c_{k+1}/c_k$ is positive real, so $e^{i\theta}=1$. Therefore phase locking is equivalent to $a_k(z)\ge0$ for every $k\ge1$.

Multiplying by $k$ gives $q^{\mathrm{cl}}_{i,z}(k)=ka_k(z)$. If $a_k(z)\ge0$ for all $k$, then the imaginary part of the polynomial $q^{\mathrm{cl}}_{i,z}$ vanishes on infinitely many integers and hence vanishes identically. Since the polynomials $k\binom{k+s-1}{s-1}$, $1\le s\le m_i$, have distinct degrees, all $\widetilde\lambda_{i,s}(z)$ are real; the inequalities $q^{\mathrm{cl}}_{i,z}(k)\ge0$ are exactly $a_k(z)\ge0$. The converse is immediate, and then $a_0(z)$ is real, so $e^{a_0(z)}>0$.

Dividing \eqref{eq:essential-occupation} by $e^{a_0(z)}$ and restricting to $1\le k\le n$ gives the formula for $\mathcal Q^{\mathrm{sing}}_{i,n}$ and hence the probability measure after normalization. For $0<x<1$ the series $\sum_{k\ge1}a_k(z)x^k$ converges; independent Poisson variables with means $a_k(z)x^k$ have joint weights proportional to $x^{\sum kN_k}\prod_k a_k(z)^{N_k}/N_k!$, and conditioning on total size $n$ cancels $x^n$. The assertion about $R_i$ follows from the convolution formula \eqref{eq:essential-full-insertion}.
\end{proof}

Thus \(Z_{i,n}\) has the same local divisor as \(f^{(n)}\) on \(D\), and \eqref{eq:essential-fisher-zero-convergence} is the chamberwise zero-divisor limit for these coefficient functions. Viewed in the terminology of Fisher zeros \cite{Fisher1965}, \(z\) is the complex parameter with respect to which the zeros are taken, while \(Z_{i,n}(z)\) is the finite-\(n\) coefficient partition function determined by the generating function \(G_{i,z}\). The singular part of \(G_{i,z}\) has the occupation-number and labelled-polymer expansions \eqref{eq:essential-occupation}--\eqref{eq:essential-labelled-polymer}, and the analytic factor enters through the convolution \eqref{eq:essential-full-insertion}. A probabilistic interpretation is available only at parameter values satisfying the real-activity conditions of Proposition~\ref{prop:essential-cluster-potts}, with a positive normalizing coefficient; for the full coefficient one also needs \(r_{i,\ell}(z)\ge0\) for all \(\ell\). Away from this positivity regime, the construction is a complex-weight coefficient identity and divisor-limit statement, not a positivity statement about the zeros.

\section{Consequences for final sets and local clusters}\label{sec:consequences}

The two corollaries in this section combine the fixed-scale theorem with the refinements proved afterward. The first uses the sublinear chamber theorem to describe the final set away from the finite exceptional loci. The second connects the atomic part of the fixed-scale measure with the microscopic laws at essential singularities.

\begin{cor}[final set away from the exceptional finite-plane sets]\label{cor:regular-final-set}
For each essential site $a_i\in\Z(T)$, let
\[
\mathfrak S_i\subset\mathcal V_i^\circ\setminus\{a_i\}
\]
be the leading Stokes set, defined locally by \eqref{eq:local-stokes-sets}; changing the local branch of $\eta$ only permutes the saddle labels, so this set is well defined. Let $\mathcal C_i$ be the set of connected components of $(\mathcal V_i^\circ\setminus\{a_i\})\setminus\mathfrak S_i$. On each $V\in\mathcal C_i$, let $\mathfrak z_{i,V}$ be the chamber measure from Theorem~\ref{thm:sublinear-essential}, defined locally on simply connected subdomains of $V$; the local definitions agree on overlaps by the branch-relabeling statement in Proposition~\ref{prop:active-wright-coefficients}. Set
\[
\Omega_{\mathrm{reg}}
:=
\mathbb C\setminus
\left(
\Sigma\cup \operatorname{Vor}(\Sigma)\cup
\bigcup_{a_i\in\Z(T)}\mathfrak S_i
\right).
\]
Then, with supports taken relative to the corresponding chamber,
\[
L(f)\cap\Omega_{\mathrm{reg}}
=
\bigcup_{a_i\in\Z(T)}
\ \bigcup_{V\in\mathcal C_i}
\operatorname{supp}\mathfrak z_{i,V}.
\]
Moreover,
\[
\operatorname{Vor}(\Sigma)\cup\Z(T)\cup
\bigcup_{a_i\in\Z(T)}
\ \bigcup_{V\in\mathcal C_i}
\operatorname{supp}\mathfrak z_{i,V}
\subseteq L(f).
\]
Apart from the displayed inclusion, no complete description is asserted at ordinary poles, on Voronoi edges or vertices, or on the Stokes sets \(\mathfrak S_i\). When \(h>0\), the statement is only a fixed finite-plane assertion and does not describe zeros whose distance tends to infinity.
\end{cor}

\begin{proof}
The equality on $\Omega_{\mathrm{reg}}$ is local. Let $z_0\in\Omega_{\mathrm{reg}}$. Then $z_0$ lies in a unique Voronoi interior. If its nearest site is an ordinary pole, Theorem~\ref{prop:local-coef-unified}\textup{(1)} gives a nonzero asymptotic for $B_n$ on a compact neighborhood of $z_0$, so $z_0\notin L(f)$. If its nearest site is an essential singularity, choose a simply connected $D\Subset\mathcal V_i^\circ\setminus\Sigma$ whose Stokes chamber $V$ contains a compact neighborhood of $z_0$. Theorem~\ref{thm:sublinear-essential} gives
\[
z_0\in L(f)\quad\Longleftrightarrow\quad z_0\in\operatorname{supp}\mathfrak z_{i,V},
\]
which proves the equality.

The chamber-support part of the inclusion follows from the same theorem. For the remaining part, let $x\in\operatorname{Vor}(\Sigma)\cup\Z(T)$ and let $U$ be a neighborhood of $x$. If $x\in\Z(T)$, choose a disk $G=D(x,r)\Subset U$ such that $G\cap\Sigma=\{x\}$ and $G\setminus\{x\}$ lies in the Voronoi interior of the site $x$. Then $\mu_{\mathrm{fix}}(G)>0$ by the atomic term in \eqref{eq:measure}, and $\mu_{\mathrm{fix}}(\partial G)=0$. Otherwise $x\in\operatorname{Vor}(\Sigma)$. Choose $G\Subset U\setminus\Sigma$ so that $G$ meets the relative interior of one nonempty Voronoi edge in a set of positive arclength and $\mu_{\mathrm{fix}}(\partial G)=0$; then $\mu_{\mathrm{fix}}(G)>0$ by the edge density in \eqref{eq:measure}. In either case, Theorem~\ref{thm:limit} gives $\mu_n(G)>0$ for all sufficiently large $n$, hence $B_n$ has a zero in $G$. In the first case Proposition~\ref{prop:local} gives $B_n(x)\ne0$, and in the second case $G\cap\Sigma=\varnothing$; therefore the zero lies in $G\setminus\Sigma$, where it is a zero of $f^{(n)}$ by \eqref{eq:fnrep}. Thus $x\in L(f)$.
\end{proof}

\begin{cor}[atomic mass and microscopic local laws]\label{cor:atomic-local-bridge}
Fix $c_j\in \Z(T)$ and choose $r>0$ so small that $\overline{D(c_j,r)}\cap\Sigma=\{c_j\}$, $c_j$ is the unique nearest site for every point of $\overline{D(c_j,r)}\setminus\{c_j\}$, and $P_\sharp$ has no zero on $\overline{D(c_j,r)}$. Then $\mu_{\mathrm{fix}}(\partial D(c_j,r))=0$ and
$$\sum_{\substack{\zeta:\,B_n(\zeta)=0\\ |\zeta-c_j|<r}}\operatorname{mult}_\zeta(B_n)=m_j n+o(n),$$
where $\operatorname{mult}_\zeta(B_n)$ denotes the multiplicity of $\zeta$ as a zero of $B_n$.
Equivalently, the same asymptotic count, with multiplicity, holds for the zeros of $f^{(n)}$ in $D(c_j,r)\setminus\{c_j\}$. Thus the atom $(m_j/\kappa)\delta_{c_j}$ in Theorem~\ref{thm:limit} represents a cluster of $m_j n+o(n)$ zeros collapsing to $c_j$.

Let
$$E(z)=\sum_{s=1}^{m_j}\lambda_{j,s}(z-c_j)^{-s}+E_j^{\mathrm{reg}}(z),\qquad \lambda_{j,m_j}\ne0,$$
and set $\alpha_j:=p_j-\nu_j$. For every $\chi\in C_c(\C)$,
\begin{equation}\label{eq:full-local-global-bridge}
\frac1{m_jn}\sum_{\substack{f^{(n)}(\zeta)=0\\0<|\zeta-c_j|<r}}
\operatorname{mult}_\zeta(f^{(n)})\,
\chi\!\left(n^{1/m_j}(\zeta-c_j)\right)
\longrightarrow
\int_\C\chi\,\dd\widehat\mu_{m_j,\lambda_{j,m_j}}.
\end{equation}
If $r_n\downarrow0$ and $n^{1/m_j}r_n\to\infty$, then
\[
\nu_{j,n}:=\frac1{m_jn}
\sum_{\substack{f^{(n)}(\zeta)=0\\0<|\zeta-c_j|<r_n}}
\operatorname{mult}_\zeta(f^{(n)})\,
\delta_{\,n^{1/m_j}(\zeta-c_j)}
\]
satisfies
\[
\nu_{j,n}\xrightarrow{\;w\;}\widehat\mu_{m_j,\lambda_{j,m_j}}
\]
weakly on $\widehat\C$ as a finite Borel measure. In particular $\nu_{j,n}(\widehat\C)\to1$.

For the reduced model
$$g_{\alpha_j,m_j}(z):=(z-c_j)^{\alpha_j}\exp\!\left(\frac{\lambda_{j,m_j}}{(z-c_j)^{m_j}}\right),$$
this specializes to Corollary~\ref{cor:mp} when $m_j=1$ and to Corollary~\ref{cor:micro-higher-zero} when $m_j\ge2$.
\end{cor}

\begin{proof}
Choose $r$ as in the statement. Then $\overline{D(c_j,r)}$ meets no Voronoi edge and contains no site other than $c_j$, so
$$\mu_{\mathrm{fix}}\bigl(D(c_j,r)\bigr)=\frac{m_j}{\kappa}$$
and $\mu_{\mathrm{fix}}(\partial D(c_j,r))=0$. By Theorem~\ref{thm:limit},
$$\mu_n\bigl(D(c_j,r)\bigr)\longrightarrow \frac{m_j}{\kappa}.$$
Since $\mu_n$ is the normalized zero-counting measure of $B_n$ and Proposition~\ref{prop:deg} gives $\deg B_n=\kappa n+O(1)$, it follows that
$$\sum_{\substack{\zeta:\,B_n(\zeta)=0\\ |\zeta-c_j|<r}}\operatorname{mult}_\zeta(B_n)
=\deg B_n\,\mu_n\bigl(D(c_j,r)\bigr)
=m_j n+o(n).$$
By Proposition~\ref{prop:local}, one has $B_n(c_j)\ne0$, and away from $\Sigma$ the zeros of $f^{(n)}$ are exactly the zeros of $B_n$, which gives the equivalent formulation for $f^{(n)}$.

Proposition~\ref{prop:local} with $n=0$ gives, near $c_j$,
$$
f(z)=(z-c_j)^{\alpha_j}
\exp\!\left(\sum_{s=1}^{m_j}\lambda_{j,s}(z-c_j)^{-s}\right)\phi_j(z),
\qquad \phi_j(c_j)\ne0.
$$
By the choice of $r$, $\phi_j$ is holomorphic and nonvanishing on a neighborhood of $\overline{D(c_j,r)}$. Theorem~\ref{thm:full-local-microscopic} applied to this germ gives \eqref{eq:full-local-global-bridge}.

Let $r_n\downarrow0$ with $n^{1/m_j}r_n\to\infty$. Then $r_n<r$ for all large $n$. For $R>0$ with $\widehat\mu_{m_j,\lambda_{j,m_j}}(\{|u|=R\})=0$, \eqref{eq:full-local-global-bridge} and standard approximation of the indicator of $\{|u|<R\}$ give
$$
\frac1{m_jn}
\sum_{\substack{f^{(n)}(\zeta)=0\\0<n^{1/m_j}|\zeta-c_j|<R}}
\operatorname{mult}_\zeta(f^{(n)})
\longrightarrow
\widehat\mu_{m_j,\lambda_{j,m_j}}(\{|u|<R\}).
$$
Fix $\varepsilon>0$ and choose such an $R$ with $\widehat\mu_{m_j,\lambda_{j,m_j}}(\{|u|<R\})>1-\varepsilon$. For all large $n$, the disk $|\zeta-c_j|<Rn^{-1/m_j}$ is contained in $|\zeta-c_j|<r_n$, so the preceding display gives a lower bound $1-\varepsilon-o(1)$ for the normalized number of zeros in $0<|\zeta-c_j|<r_n$. Since $\varepsilon$ is arbitrary, the liminf is at least $1$. The fixed-radius count already proved gives the limsup at most $1$, because $r_n<r$ for all large $n$. Thus
\[
\frac1{m_jn}
\sum_{\substack{f^{(n)}(\zeta)=0\\0<|\zeta-c_j|<r_n}}
\operatorname{mult}_\zeta(f^{(n)})\longrightarrow1.
\]
The local germ $F=f|_{0<|z-c_j|<r}$ has the same derivatives as $f$ in the punctured disk, so this is exactly the total-mass hypothesis \eqref{eq:full-local-total-hyp} for Theorem~\ref{thm:full-local-microscopic}. The second part of that theorem gives the asserted weak convergence of $\nu_{j,n}$ on shrinking disks. Setting $\lambda_{j,s}=0$ for $s<m_j$ and $\phi_j\equiv1$ gives the reduced model, whose explicit forms are Corollaries~\ref{cor:mp} and~\ref{cor:micro-higher-zero}.
\end{proof}

\section{Open problems}\label{sec:open}

For the class \eqref{eq:mixed}, Theorem~\ref{thm:limit} settles the fixed-scale normalized zero-counting law, and Theorem~\ref{thm:full-local-microscopic} settles the leading microscopic law at each finite essential singularity. The following problems concern finer transition and escape scales and are not needed for the fixed-scale law. Theorem~\ref{thm:sublinear-essential} and Corollary~\ref{cor:regular-final-set} determine $L(f)$ only after intersection with $\Omega_{\mathrm{reg}}$ from Corollary~\ref{cor:regular-final-set} and give the stated inclusions on the excluded sets. Proposition~\ref{prop:active-wright-coefficients} gives a closed formula for the vanishing pattern of the Wright coefficients in each essential chamber; what remains open is the corresponding transition theory and the outer scale when $h>0$.

\paragraph{Escaping mass at infinity.}
When the polynomial part $H$ of $S/T$ is nonconstant, Theorem~\ref{thm:limit} identifies only the mass $h/\kappa$ that escapes to $\infty$. A sharper result should describe the distribution of these escaping zeros after the natural rescaling determined by $H$ and should separate this outer scale from the finite-plane Voronoi geometry.

\paragraph{Transition asymptotics.}
Theorem~\ref{thm:limit}, Corollary~\ref{cor:regular-final-set}, and Theorem~\ref{thm:sublinear-essential} identify the fixed-scale law and the final set on $\Omega_{\mathrm{reg}}$. What remains open is the local zero process at the natural transverse scale near active anti-Stokes arcs, near leading Stokes walls where, in an order-$m_i$ essential chamber, the finite-$n$ Stokes walls may be displaced by $O(n^{-1/(m_i+1)})$, on genuine Stokes-connection curves of the coefficient contour, and near Voronoi vertices, where the special transition kernels should be resolved separately.

\paragraph{Positive realizations.}
The coefficient identities in Subsection~\ref{subsec:essential-partition} are positive only under the real inequalities of Proposition~\ref{prop:essential-cluster-potts}. It remains open whether a natural positive finite-volume model realizes the subharmonic envelope $p_{i,V}$, the effective volume $n^{m_i/(m_i+1)}$, or the full perturbed local laws of Theorem~\ref{thm:full-local-microscopic} away from that locus.

\appendix
\section{Characterization of hyperexponential functions}\label{app:mixed}
\begin{prop}
Let $S\subset \mathbb P^1(\mathbb C)$ be finite and set
\[\Omega=\mathbb P^1(\mathbb C)\setminus S.\]
Let $f$ be a meromorphic function on $\Omega$, not identically zero; $f'/f$ is understood on the complement of the zeros and poles of $f$. Then the following are equivalent:
\begin{enumerate}
    \item there exist $c\in \mathbb C^\times$, $R\in\mathbb C(z)^\times$, and $H\in \mathbb C(z)$, with $H$ holomorphic on $\Omega$, such that
    \[
    f(z)=c\,R(z)e^{H(z)} \quad \text{on } \Omega;
    \]
    \item the logarithmic derivative $f'/f$ belongs to $\mathbb C(z)$.
\end{enumerate}
\end{prop}

\begin{proof}
The implication $(1)\Rightarrow(2)$ is immediate, since
\[\frac{f'}{f}=\frac{R'}{R}+H'\in\mathbb C(z).\]

Assume conversely that
\[r(z):=\frac{f'(z)}{f(z)}\in \mathbb C(z).\]
Let $A\subset\mathbb C$ be the finite set of finite poles of $r$. For each $a\in A$, choose $\varepsilon>0$ so small that the positively oriented circle
\[\gamma_a(t)=a+\varepsilon e^{it}, \qquad 0\le t\le 2\pi,\]
lies in $\Omega$, encloses no finite pole of $r$ other than $a$, and contains no zero or pole of $f$ on the path. This is possible because the zeros and poles of a meromorphic function are discrete in $\Omega$, so the forbidden radii form a countable set. This includes the case $a\in S$: only the center is removed, not the circle. Then $f\circ\gamma_a$ is a loop in $\mathbb C^\times$, and
\[\operatorname{Res}_{z=a}(r)=\frac{1}{2\pi i}\int_{\gamma_a}\frac{f'(z)}{f(z)}\,\dd z\]
is its winding number about $0$. Hence
\[n_a:=\operatorname{Res}_{z=a}(r)\in\mathbb Z\]
for every finite pole of $r$, including punctures in $S$. If $a\in\Omega$, the local form $f(z)=(z-a)^k u(z)$, $u(a)\ne0$, also shows that the pole is simple and has residue $k$.

Define, allowing negative exponents in the product,
\[R(z):=\prod_{a\in A}(z-a)^{n_a}\in\mathbb C(z)^\times,\qquad
s(z):=r(z)-\frac{R'(z)}{R(z)}.\]
Then $s$ has zero residue at every finite point. At each finite point of $\Omega$ it has no pole: if $r$ is regular this is clear, and if $r$ has a pole there the preceding local form shows that subtracting $R'/R$ removes it. If $\infty\in\Omega$, write $w=1/z$ and $f=w^k u(w)$ near $w=0$, with $u(0)\ne0$. Then $r(z)=-k/z+O(z^{-2})$ as $z\to\infty$; by the residue theorem, $R'(z)/R(z)=-k/z+O(z^{-2})$, and hence $s(z)=O(z^{-2})$, so $s$ is holomorphic at $\infty$. At every finite point of $S$ the residue of $s$ is zero by construction. If $\infty\in S$, then $\operatorname{Res}_\infty s=-\sum_{a\in\mathbb C}\operatorname{Res}_a s=0$. Thus $s$ is holomorphic on $\Omega$ and has only zero-residue poles at points of $S$.

The partial fraction expansion of $s$ has no simple finite-pole terms: its polynomial part integrates to a polynomial, and each higher-order principal part integrates to a rational function. Thus there is $H\in\mathbb C(z)$ with $H'=s$. The preceding paragraph shows that $H$ is holomorphic on $\Omega$; in particular, when $\infty\in\Omega$, the estimate $s(z)=O(z^{-2})$ makes $H$ holomorphic at infinity.

Now put
\[g(z):=\frac{f(z)}{R(z)e^{H(z)}}.\]
This is a meromorphic function on the connected surface $\Omega$, and
\[\frac{g'(z)}{g(z)}=\frac{f'(z)}{f(z)}-\frac{R'(z)}{R(z)}-H'(z)=0\]
on the complement of its zeros and poles. Hence $g'\equiv0$ as a meromorphic differential, so $g$ is constant. Since $f\not\equiv0$, this constant is some $c\in\mathbb C^\times$, proving $(2)\Rightarrow(1)$.
\end{proof}

\end{document}